\documentclass[12pt]{scrartcl}

\usepackage[utf8]{inputenc}
\usepackage[numbers]{natbib}
\usepackage{amsthm}
\usepackage{amsmath}
\usepackage{amssymb}
\usepackage{myprob}
\usepackage{mythm}
\usepackage{bbold}
\usepackage{array}
\usepackage{float}
\usepackage[hidelinks]{hyperref}
\usepackage{enumerate}
\usepackage{mathtools}
\usepackage{manfnt}
\usepackage{tabu}
\usepackage{colortbl}

\newif\iflongversion
\longversiontrue % comment out to hide typical distance part

\usepackage{tikz}
\usetikzlibrary{calc, arrows}
\newcommand*{\RightDashedArrow}[1][]{{\tikz [baseline=-0.15ex,-latex, dashed,#1] \draw [#1] (0pt,0.5ex) -- (1.0em,0.5ex);}}%
\newcommand*{\LeftDashedArrow}[1][]{{\tikz [baseline=-0.15ex,-latex, dashed,#1] \draw [#1] (1.0em,0.5ex) -- (0pt,0.5ex);}}%
\newcommand*{\Vertex}[1][]{{\tikz [baseline=-0.75ex,-latex,#1] \draw [fill=black!60,#1] (0,0) circle (2pt);}}%
\newcommand*{\Outclaw}[1][]{
    {
        \tikz [semithick,baseline=-0.57ex,-latex] 
        {
            \draw (0,0) -- (3.5ex,0ex);
            \draw (0,0) .. controls (0.4ex,1.2ex) .. (3ex,2ex);
            \draw (0,0) .. controls (0.4ex,-1.2ex) .. (3ex,-2ex);
            \draw [thin,fill=black!60,#1] (-2pt,0) circle (2pt);
            \draw [thin,fill=black!60,#1] (3.3ex,2ex) circle (2pt);
            \draw [thin,fill=black!60,#1] (3.3ex,-2ex) circle (2pt);
            \draw [thin,fill=black!60,#1] (3.8ex,0ex) circle (2pt);
        }
    }
}%
\newcommand*{\Inclaw}[1][]{
    {
        \tikz [semithick,baseline=-0.75ex,-latex] 
        {
            \draw (3.5ex,0ex) -- (0,0);
            \draw (3ex,2ex) .. controls (0.6ex,1.2ex) .. (0,0);
            \draw (3ex,-2ex) .. controls (0.6ex,-1.2ex) .. (0,0);
            \draw [thin,fill=black!60,#1] (-2pt,0) circle (2pt);
            \draw [thin,fill=black!60,#1] (3ex,2ex) circle (2pt);
            \draw [thin,fill=black!60,#1] (3ex,-2ex) circle (2pt);
            \draw [thin,fill=black!60,#1] (3.5ex,0ex) circle (2pt);
        }
    }
}%
\DeclarePairedDelimiter{\floor}{\lfloor}{\rfloor}
\DeclarePairedDelimiter{\ceil}{\lceil}{\rceil}
\newcommand\dtv[1]{{\left\|#1\right\|_{\textsc{tv}}}}

\DeclareRobustCommand{\stirling}{\genfrac\{\}{0pt}{}}

\newcommand\bigO[1]{{O\left( #1 \right)}}
\newcommand\smallo[1]{{o\left( #1 \right)}}

\newcommand\bigOm[1]{{\Omega\!\left( #1 \right)}}

\newcommand\numberthis{\addtocounter{equation}{1}\tag{\theequation}}

\newcommand{\Bin}{\mathop{\mathrm{Bin}}}

\newcommand{\Poi}{\mathop{\mathrm{Poi}}}

\newcommand{\dist}{\mathop{\mathrm{dist}}}
\newcommand{\arc}{\mathop{\mathrm{arc}}}

\newcommand\randdfa{{\cD_{n,k}}}
\newcommand\randdfaK{{\cD_{n,K}}}
\newcommand\randdfap{{\widehat{\cD}_{n,k}}}
\newcommand\randacycdfa{{\cD_{n,k}^{\mathrm A}}}
\newcommand\randsimple{{\cD_{n,k}^{*}}}
\newcommand\randsimpleu{{\cD_{n,k}^{**}}}
\newcommand\randdfaidx[1]{{\cD_{n,k}^{[#1]}}}

\newcommand\giant{{\cG_n}}
\newcommand\giantsize{{|\giant|}}
\newcommand\giantout{{\cG_n^c}}
\newcommand\giantoutsize{{\left|\cG_n^c \right|}}
\newcommand\giantoutgraph{\randdfa[\giantout]}

\newcommand\onecore{\cO_n}
\newcommand\onecoresize{|\onecore|}
\newcommand\onecoreout{\onecore^c}

\newcommand\oneshift{{\partial \cO_n}}

\newcommand{\spec}[1]{\cS_{#1}}
\newcommand\specone{\spec{1}}
\newcommand\speconesize{|\specone|}
\newcommand\specout[1]{\cS^{\prime}_{#1}}

\newcommand\specvout[1]{\cS^{*}_{#1}}
\newcommand\specvoutsize[1]{|\specvout{#1}|}
\newcommand\specvoutgraph[1]{\randdfa[\cS^{*}_{#1}]}

\newcommand\speconedist[1]{\cS_{#1}^{+}(v_1)}
\newcommand\spectwodist[1]{\cS_{#1}^{-}(v_2)}
\newcommand\speconedistm{\speconedist{m}}
\newcommand\spectwodistm{\spectwodist{\le m}}

\newcommand\vin{\cV_n}

\newcommand\vinsize{|\vin|}
\newcommand\vout{\vin^c}
\newcommand\voutgraph{\randdfa[\vin^c]}
\newcommand\voutsize{|\vin^c|}

\newcommand\tauk{{\tau_k}}
\newcommand\nuk{{\nu_k}}
\newcommand\muk{{e^{-\tauk}}}
\newcommand\lambdak{{\lambda_k}}
\newcommand\err{n^{1/2+\delta}}
\newcommand\errp{n^{{-1}/2+\delta}}

\newcommand\errpd{{n^{{1}/2-\delta}}}
\newcommand\inrange{[\nuk n-\err, \nuk n +\err]}
\newcommand\outrange{[\muk n-\err, \muk n +\err]}
\newcommand\Iin{{\cI_n}}
\newcommand\Iout{{\cI_n^c}}

\newcommand\cycnumgiant{C_{n}}
\newcommand\cyclgiant{C_{n,\ell}}

\newcommand\cycnum{C_n^*}
\newcommand\cycl{C_{n,\ell}^*}
\newcommand\cycsmall{\overline{C_{n}^*}}

\newcommand\ucyc{U}
\newcommand\ucycl{\ucyc_{\ell}}

\newcommand{\sh}{\mathop{\mathrm{Sh}}}

\newcommand{\vecs}{\vec{s}_m}
\newcommand{\vect}{\vec{t}_m}

\newcommand{\treev}{\cT_{v}}
\newcommand{\treevw}{W_{v}^{*}}

\newcommand\dfa{\textsc{dfa}}
\newcommand\scc{\textsc{scc}}
\newcommand\dagr{\textsc{dag}}

\title{The graph structure of a deterministic automaton chosen at
        random\iflongversion: full version\fi}
\author{\small Xing Shi Cai, Luc Devroye
    \\
    \small 
    School of Computer Science, McGill University of Montreal, Canada,\\
    \small 
    \texttt{xingshi.cai@mail.mcgill.ca}\\
    \small 
    \texttt{lucdevroye@gmail.com}
}

\date{\small \today}

\begin{document}

\maketitle

\begin{abstract}
    An \(n\)-state deterministic finite automaton over a \(k\)-letter alphabet can
    be seen as a digraph with \(n\) vertices which all have \(k\) labeled out-arcs.
    \citet{Grusho1973} proved that whp in a random \(k\)-out digraph there is a
    strongly connected component of linear size, i.e., a giant,
    and derived a central limit theorem. We show that whp the part
    outside the giant contains at most a few short cycles and mostly consists
    of tree-like structures, and present a new  proof of \citeauthor{Grusho1973}'s
    theorem. Among other things, we pinpoint the phase transition for strong
    connectivity.

    \medskip

    \noindent \textbf{Keywords:} random digraphs; deterministic finite automaton
\end{abstract}

\section{Introduction}

\subsection{The model and the history}

The deterministic finite automaton (\dfa{}) is widely used in computational complexity
theory. Formally, a \dfa{} is a \(5\)-tuple \((Q, \Sigma, \delta, q_0, F)\),
where \(Q\)
is a finite set called the set of states, \(\Sigma\) is a finite set called the
alphabet, \(\delta:Q\times\Sigma \to Q\) is the transition function, \(q_0 \in
Q\)
is the start state, and \(F \subseteq Q\) is the set of accept states. If
\(q_0\) and \(F\)
are ignored, a \dfa{} with \(n\) states and a \(k\)-alphabet can be seen as a digraph
with vertices \([n]\equiv\{1\ldots,n\}\) in which each vertex has \(k\) out-arcs
labeled by \(1, \ldots, k\) (a \emph{\(k\)-out digraph}). Note that such a digraph
can have self-loops and multi-arcs. For a basic introduction to \dfa{} and its
applications, see \citep{Sipser2012}.

%We choose the underlying digraph uniformly at random and study its
%graph structure.  
Let \(\randdfa\) denote a digraph chosen uniformly at random
from all \(k\)-out digraphs of \(n\) vertices. Equivalently \(\cD_{n,k}\) is
a random \(k\)-out digraph of \(n\) vertices with the endpoints of
its \(kn\) arcs chosen independently and uniformly at random.

When \(k=1\), \(\randdfa\) is equivalent to a uniform random mapping from
\([n]\) to itself, which has been well studied by \citet{Kolchin1986random},
\citet{Flajolet1990}, and \citet{Aldous1994}.  In \(\cD_{n,1}\), the largest
strongly connected component (\scc{}) has expected size
\(\Theta(\sqrt{n})\), and so does the size of the longest cycle.  However, as
shown later, for \(k \ge 2\), the largest \scc{} has expected size
\(\Theta(n)\).

From now on we assume that \(k \ge 2\).
Let \(\spec{v}\) (the \emph{spectrum} of
\(v\)) be the set of vertices in \(\randdfa\) that are reachable from vertex
\(v\),
including \(v\) itself. In \citeyear{Grusho1973} \citet{Grusho1973} first
proved that \((\speconesize-\nuk n)/\sigma_k \sqrt{n}\) converges in distribution
to a standard normal, where \(\nuk\) and \(\sigma_k\) are explicitly defined
constants.  

Given a set of vertices \(\cS \subseteq [n]\), call \(\cS\) \emph{closed} if
there are no arcs that start from vertices in \(\cS\) and end at vertices in
\(\cS^c \equiv [n] \setminus \cS\).  Let \(\giant\) be the set of vertices in the largest closed
\scc{} in \(\randdfa\). (If the largest closed \scc{} is not unique, let \(\giant\) be the vertex
set of the largest closed \scc{} that contains the smallest vertex-label.) We call \(\giant\) the giant.
\citeauthor{Grusho1973} also proved that \(\giantsize\) has the same limit
distribution as \(\speconesize\) by showing that with high probability (whp)
\(\giant\) is reachable from all vertices and that \(\speconesize - \giantsize
= o_p(\sqrt{n})\) (see \citep{Janson2011random} for the notation).  His proof
relies on a result by \citet{Sevastyanov1967} which approximates the
exploration of \(\specone\) with a Gaussian process.

In 2012 \citet*{Carayol12}
proved a local limit theorem for
\(\speconesize\) by analyzing the limit behavior of the
probability that \(\speconesize = s\) for an \(s\) close to \(\nuk n\).  Their proof
depends on a theorem by \citet*{Korshunov1978} which says that conditioned
on every vertex having in-degree at least one, the probability that \(\specone =
[n]\) tends to some constant.  \citeauthor{Carayol12} derived a simple and
explicit formula of this constant from their theorem.  (The same formula is
also proved by \citet{Lebensztayn2010} with a more analytic approach using
Lagrange series.)

Lately the simple random walk (SRW) on \(\randdfa\) has gained some attention
for its applications in machine learning.  \citet*{Perarnau2014} studied the
stationary distribution of the SRW by analyzing the distances in \(\randdfa\).
They proved that the diameter and the typical distance, rescaled by \(\log n\),
converge in probability to explicit constants.  \citet*{Angluin2015}
studied the rate of the convergence to the stationary distribution of the SRW.
They also suggested an algorithm for learning a uniformly random
\dfa{} under Kearns' statistical query model~\citep{Kearns1998}.

\subsection{Our results and a sketch of proof}

A digraph can be uniquely decomposed into \scc{}s which form a directed acyclic graph (\(\dagr{}\))
through a process called condensation
that contracts every \scc{} into a single vertex while keeping all the arcs
between \scc{}s \citep{Bang2009}.  The condensation \dagr{} of \(\randdfa\) is denoted by \(\randacycdfa\).

Let \(\giantout \equiv [n] \setminus \giant\), i.e., \(\giantout\) is the set
of vertices that are outside the giant. The
structure of \(\randacycdfa\) depends on \(\giantoutgraph\), the digraph
induced by \(\giantout\).  Our analysis shows that in \(\giantoutgraph\) the
total number of cycles and the number of cycles of a fixed length both converge
to Poisson distributions with constant means.  So the number of cycles and the
length of the longest cycle are both \(O_p(1)\) (see \citep{Janson2011random}).
Furthermore, these cycles are vertex-disjoint whp.  Therefore, almost every
vertex in \(\giantout\) is a \scc{} itself and \(\randacycdfa\) is very much
like \(\randdfa\) with the giant contracted into a single vertex.

The \emph{\(d\)-core} of an undirected graph is the maximum induced subgraph in which all vertices
have degree at least \(d\). Similarly the \emph{\(d\)-in-core} of a digraph can be defined
as the maximum induced sub-digraph in which all vertices have in-degree at
least \(d\). Let \(\onecore\) denote the set of vertices in the one-in-core of
\(\randdfa\). Note that
\(\giant \subseteq \onecore\) since a \scc{} induces a sub-digraph with each
vertex having in-degree at least one. Also note that cycles cannot exist
outside \(\onecore\), for otherwise they contradict the maximality of \(\onecore\).
Now assume that every vertex can reach \(\giant\), which happens whp by
\citet{Grusho1973}.  Then \(\randdfa\) can be divided into three layers: the
center is \(\giant\); then comes \(\onecore\setminus\giant\), which consists of
cycles outside \(\giant\) and paths from these cycles to \(\giant\); the outermost
is \(\onecoreout \equiv [n] \setminus \onecore\), which is acyclic.

\begin{figure}[ht!]
  \centering
    \begin{tikzpicture}
    \node[anchor=south west,inner sep=0] at (0,0) {\includegraphics{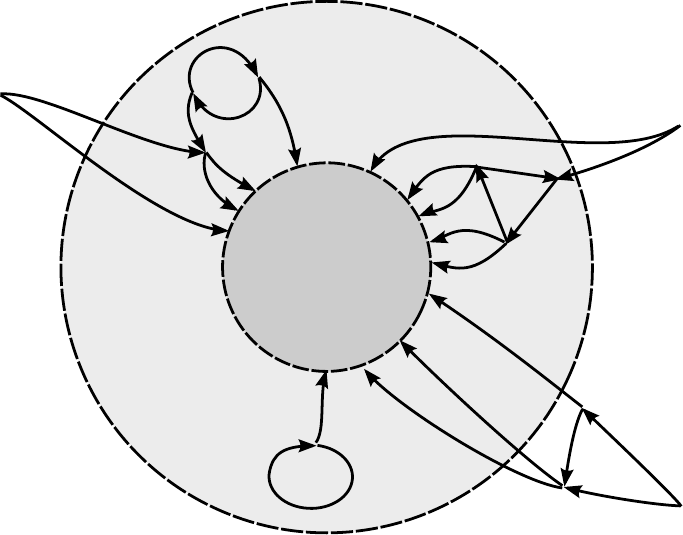}};
    \node at (3.4,2.7) {\(\giant\)};
    \node at (1.5,2.7) {\(\onecore \! \setminus \! \giant\)};
    \node at (-0.1,2.7) {\(\onecoreout\)};
    %\draw[step=1cm,gray,very thin] (-2,0) grid (6,6);
    \end{tikzpicture}
    \caption{Three layers of \(\randdfa\): the giant \(\giant\);
    the one-in-core \(\onecore\); and the whole graph.}
\end{figure}

Since there cannot be many vertices in cycles outside the giant, the middle
layer \(\onecore \setminus \giant\) must be very ``thin''.  Thus if we can
prove \((\onecoresize - \nuk n)/\sqrt{n}\) converges to a normal distribution,
then we can also prove it for \(\giantsize\).  The event \(\onecoresize = s\)
happens if and only if there is a set of vertices \(\cS\) with \(|\cS|=s\) such
that: (a) \(\randdfa[\cS]\), the sub-digraph induced by \(\cS\), has minimum
in-degree one (\emph{surjective}) and there are no arcs going from \(\cS\) to
\(\cS^c\) (\emph{closed}), which we refer to as \(\cS\) being a \emph{
\(k\)-surjection} (since \(\randdfa[\cS]\) is equivalent to a surjective
function from \([ks]\) to \([s]\)); (b) \(\randdfa[\cS^c]\) is acyclic. The
probability of (a) can be computed by counting the number of surjective
functions. And we are able to show that the probability of (b) converges to a
constant. Note that for a fixed set {\(\cS\)} (a) and (b) are independent because they depend on the
endpoints of two disjoint sets of arcs.  Thus we can get the limit of
\(\p{\onecore = \cS}\).  Since the one-in-core of a digraph is unique,
\(\p{\onecoresize = s} = \sum_{\cS \subseteq [n]:|\cS|=s} \p{\onecore = \cS}\).  Thus
we can finish the proof by computing the characteristic function of
\((\onecoresize - \nuk n)/\sqrt{n}\).

Note that although our formula for \(\p{\onecoresize =s}\) is inspired by and
resembles \citeauthor{Carayol12}'s formula for \(\p{\speconesize = s}\), we
actually prove the result from scratch without relying on previous work.
Since we are able to derive explicit expressions of all the constants in our
formula, the computation of the characteristic function becomes quite simple.
Furthermore, to our knowledge this is the first self-contained proof. Thus in
Section~\ref{sec:giant} we prove:
\begin{thm}[Central limit law]
    \label{thm:CLL}
    Let \(\cZ\) denote a standard normal random variable. 
    Then as \(n \to \infty\),
    \begin{align*}
            \frac{\onecoresize - \nuk n}{\sigma_k \sqrt{n}} \inlaw \cZ, \qquad
            \frac{\giantsize - \nuk n}{\sigma_k \sqrt{n}} \inlaw \cZ, \qquad
            \frac{\max_{v \in [n]} |\spec{v}| - \nuk n}{\sigma_k \sqrt{n}} \inlaw \cZ,
    \end{align*}
    where \(\nuk\) and \(\sigma_k\) are constants defined by
    \[
    \nuk \equiv \frac {\tauk} k, \qquad \qquad \qquad \sigma_k^2 \equiv
    \frac{{\tauk}}{k e^{{\tauk}}(1-ke^{{-\tauk}})},
    \]
    and \({\tauk}\) is the unique positive solution of \(1-{\tauk}/k -
    e^{{-\tauk}} = 0\).
\end{thm}

\begin{myRemark*}
Equivalently, \(\nuk\) is the unique positive solution of \(1-\nuk=e^{-k \nuk}\) and
\[
    \sigma_k^2 = \frac{\nuk (1-\nuk)}{1-k(1-\nuk)}.
\]
Let \(G(n,m)\) be a Erdős–Rényi random graph, i.e., a graph chosen uniformly at random from all
graphs with \(n\) vertices and \(m\) edges \cite{Erdos60onthe}.
It is well-known that for \(k > 1\),
\(|\cC_{\max}^n|\)---the size of the largest component in \(G(n,m=nk/2)\)---is \( (\nuk +o(1)) n\)
whp.
Moreover, \( (|\cC_{\max}^n|-\nuk n)/\sqrt{n}\) also converges in distribution 
to a normal random variable with variance \(\sigma_k^2\) (see, e.g., \citet{durrett2007random}).
Intuitively, this is because the in-degree of a vertex in \(\randdfa\) has asymptotically a Poisson
distribution of mean \(k\). Thus a backward exploration process from vertex in \(\randdfa\)
is approximately a Galton-Watson process with survival probability \(\nuk\),
as is the exploration process starting from a vertex in \(G(n,m=nk/2)\).
\end{myRemark*}

\tikzstyle{hackennode}=[draw,circle,fill=white,inner sep=0,minimum size=4pt]
\tikzstyle{hackenline}=[line width=3pt]

Section~\ref{sec:dag} studies the part of \(\randdfa\) outside the giant, which
determines the structure of \(\randacycdfa\) and supports the proof of
Theorem~\ref{thm:CLL}. Our results are summarized in two theorems, where all our logarithms
are natural:
\begin{thm}[Cycles outside the giant]
    \label{thm:cycle}
    We have:
    \begin{enumerate}[(a)]
        \item Let \(L_n\) be the length of the longest cycle in
            \(\giantoutgraph\). Then
            \(L_n = O_p(1)\).
        \item Let \(\cycnumgiant\) be the number of cycles in
            \(\giantoutgraph\). 
            Then
            \[\cycnumgiant \inlaw \Poi\left( \log \frac{1}{1-k \muk} \right),\]
            where \(\Poi(x)\) denotes the Poisson distribution with mean \(x\).
        \item Let \(\cyclgiant\) be the number of cycles of length \(\ell\) in
            \(\giantoutgraph\).
            Then for all fixed \(\ell \ge 1\),
            \[
                \cyclgiant \inlaw \Poi\left(\frac{(k
                    \muk)^{\ell}}{\ell}\right).
            \]
    \end{enumerate}
\end{thm}
\begin{thm}[Spectra outside the giant]
\label{thm:dag}
Let \(\specout{v} \equiv \spec{v} \cap \giantout\), i.e., \(\specout{v}\) is the
spectrum of \(v\) in \(\giantoutgraph\). Let \(\dist(v,u)\) be the distance
from \(v\)
to \(u\), i.e., the length of the shortest directed path from \(v\) to \(u\).  Then 
\begin{enumerate}[(a)]
    \item \(\p{\cup_{v \in \giantout}[\arc(\randdfa[\specout{v}]) -
        |\specout{v}| \ge 1]} = o(1)\), where \(\arc(\cdot)\) denotes the number of
        arcs. In other words, whp every spectrum in \(\giantoutgraph\) is a tree
        or a tree plus an extra arc.
    \item Let \(S_n \equiv \max_{v \in \giantout}|\specout{v}|\).
        Let \(\lambdak \equiv (k-{\tauk})\left( \frac{{\tauk}}{k-1}
        \right)^{k-1}\).
        Then
        \[\frac {S_n}{\log n} \inprob \frac
            1{\log(1/\lambdak)}.\]
    \item Let \(W_n \equiv \max_{v \in \giantout} \min_{u \in \giant}
        \dist(v,u)\), i.e., the maximum distance to \(\giant\).  Then
        \[\frac {W_n}{\log_k \log n} \inprob 1.\]
    \item Let \(M_n\) be the length of the longest path in \(\giantoutgraph\). Then
        \[\frac {M_n}{\log n} \inprob \frac 1
            {\log(e^{\tauk}/k)}.\]
    \item Let \(D_n \equiv \max_{v \in \giantout} \max_{u \in \specout{v}}
        \dist(v, u)\). Then
        \[\frac {D_n}{\log n} \inprob \frac 1
            {\log(e^{\tauk}/k)}.\]
\end{enumerate}
\end{thm}

The rest of the paper gives some other results regarding this model.
Section~\ref{sec:phase} shows that \(\randdfa\) exhibits a phase transition for strong connectivity.
Section~\ref{sec:simple} extends some of our results to simple \(k\)-out digraphs.
\iflongversion
Section~\ref{sec:typical:distance} analyzes the typical distances in
\(\randdfa\)
with a technique called path counting, which is very different from
the method used by \citeauthor{Perarnau2014} in~\citep{Perarnau2014}.
\fi
Section~\ref{sec:extension} suggests some extensions of this model.

\begin{myRemark*}
    Lemma \ref{lem:middle:layer} shows that \(\onecoresize-\giantsize = O_p(1)\).
    The intuition is that
    a digraph with minimal in-degree and out-degree at least one is likely to have a large \scc{}. 
    This phenomenon is also observed in \(D(n,p)\), which is a random digraph of \(n\) vertices with each possible arc existing
    independently with probability \(p\). \citet[thm.\ 1.3]{RSA:RSA20622} showed that in \(D(n,p)\) the
    \( (1,1)\)-core---the maximal induced sub-digraph in which each vertex has in-degree and
    out-degree at least one---differs from the largest \scc{} in size by at most \(O((\log n)^8)\),
    whp.  This intuition is also used for studying the asymptotic counts of strongly connected
    digraphs (see \citet{RSA:RSA20416,RSA:RSA20433}).
\end{myRemark*}

\section{The size of the one-in-core}

\label{sec:giant}

\subsection{The law of large numbers for the one-in-core}

To prove Theorem~\ref{thm:CLL}, we first need to narrow the range of
\(\onecoresize\) to close to \(\nuk n\).
\begin{thm}[Law of large numbers]
    \label{thm:LLL} 
    For all fixed \(\delta \in (0,1/2)\), 
    \[\p{\onecoresize \notin \cI_n} \le \frac {1+o(1)} n,\]
    where \(\cI_n \equiv \inrange\).
\end{thm}
\noindent Thus \(\onecoresize / n \inprob \nuk\), which gives the theorem
its name.

Let \(K_s\) be the number of \(k\)-surjections of size \(s\) in \(\randdfa\). Then it
suffices to show that \(\p{\sum_{s \notin \cI_n} K_s \ge 1} \le
(1+o(1))/n\).  As
argued in the introduction, for a set of vertices \(\cS\) to be the one-in-core, it
must also be a \(k\)-surjection, i.e., every vertex in \(\randdfa[\cS]\), the sub-digraph
induced by \(\cS\), must have minimum in-degree one (\(\cS\) is \emph{surjective}), and there are no arcs
going from \(\cS\) to \(\cS^{c}\) (\(\cS\) is \emph{closed}). 
Thus
\[
    \p{\cS \text{ is a \(k\)-surjection}}
= \p{\cS \text{ is surjective}~|~\cS \text{ is closed}}\p{\cS \text{ is
closed}}.
\]
Computing the limit of the two factors shows that:
\begin{lemma} We have
    \[
    \p{\sum_{s \notin \Iin} K_s \ge 1} \le \frac {1+o(1)} n.
    \]
    And for \(s \in \Iin\)
    \[
        \e{K_s} \sim \frac{1}{\sqrt{2 \pi(1-ke^{-\tauk})n}} ~
    g\left(\frac s n\right) ~ 
    \left[ f\left(\frac s n \right) \right]^{n},
    \]
    where
    \[
    g(x) \equiv \frac 1 {\sqrt{x(1-x)}}, \qquad \qquad
    f(x) \equiv 
        \left[
            \frac{x^{k-1} \gamma_k}{(1-x)^{(1-x)/x} }
        \right]^{x},
    \]
    and   \(\gamma_k \equiv \left( \frac k {e{\tauk}} \right)^k(e^{{\tauk}} -1)\).
    \label{lem:k:surj}
\end{lemma}
\noindent
Theorem~\ref{thm:LLL} follows immediately.  The proof of Lemma~\ref{lem:k:surj}
is postponed to the appendix.  (The two functions \(f(x)\) and \(g(x)\) are also
studied by \citet*{Carayol12}.)

\subsection{The central limit law of the one-in-core}

In this section we prove the part of Theorem~\ref{thm:CLL} about
\(\onecoresize\). The rest of the theorem appears as corollaries in
Section~\ref{sec:dag}.  Let \(\oneshift =
{\onecoresize - \nuk n }\).  Then \(\oneshift\) takes values in 
\(
[n] - \nuk n \equiv \{s : \nuk n + s \in [n] \}.
\)
As Theorem~\ref{thm:LLL} shows, whp \(\oneshift \le \err\) for all fixed
\(\delta \in (0,1/2) \). Thus it suffices to consider only the probability
that \(\oneshift\) takes value in the set
\[
\cJ_n \equiv ([n] - \nuk n) \cap \left[-\err,\err \right],
\]
for some fixed \(\delta \in (0,1/2)\).
Thus the characteristic function of \(\oneshift/\sqrt{n}\) is
\begin{align*}
\phi_{n}(t) 
%& = \sum_{s \in [n]-\nuk n} e^{its/\sqrt{n}} \p{\oneshift = s} \\
& = \sum_{s \in ([n]-\nuk n) \setminus \cJ_n} e^{its/\sqrt{n}} \p{\oneshift = s} +
\sum_{s \in \cJ_n} e^{its/\sqrt{n}} \p{\oneshift = s} \\
& = o(1) + \sum_{s \in \cJ_n} e^{its/\sqrt{n}} \p{\oneshift = s}.
\end{align*}

Let \(\cS\) be a set of vertices with \(|\cS| = \nuk n + s\) for some \(s \in \cJ_n\).
Recall that \(\onecore = \cS\) if and only if \(\cS\) is a
\(k\)-surjection and \(\randdfa[\cS^c]\) is acyclic, two events that are independent.
By Theorem~\ref{thm:cyc:v} in Section~\ref{sec:sub:cycle},
\(\p{\text{$\randdfa[\cS^c]$ is
acyclic}} \sim 1-ke^{{-\tauk}}.\) 
Also recall that {\(K_{x}\)} counts the number of \(k\)-surjections of size \(x\).
It follows from Lemma~\ref{lem:k:surj} that
\begin{align*}
\p{\oneshift = s}
& = \sum_{\cS \subseteq [n]:|\cS|=\nuk n + s} \p{\onecore = \cS} \\
& = \sum_{\cS \subseteq [n]:|\cS|=\nuk n + s} \p{\text{\(\cS\) is a \(k\)-surjection}}
\times \p{\randdfa[\cS^c] \text{ is acyclic}} \\
& \sim (1-k e^{{-\tauk}}) \e{K_{\nuk n + s}} \\
& = \sqrt{\frac{1-ke^{-\tauk}}{2 \pi}} ~ \frac{1}{\sqrt{n}} ~
g\left(\nuk +  \frac s n\right) ~ 
\left[ f\left(\nuk + \frac s n \right) \right]^{n},
\end{align*}
where \(K_{x}\), \(f(x)\) and \(g(x)\) are defined as in the previous subsection.

If \(s \in \cJ_n\), then Lemma~\ref{lem:function:h} in the appendix shows that
\[
g\left( \nuk + \frac s n \right) %= g(\nuk) + O\left( \frac{|s|}{{n}} \right)
= \left(1 + O\left( \frac{|s|}{{n}} \right)  \right)\frac{1}{\sigma_k\sqrt{1-k
    e^{-\tauk}}},
    \]
and
\[
f\left( \nuk + \frac s n \right) 
= \exp\left\{- \frac{s^2}{2 \sigma_k^2 n^2} \right\}  + O\left(
\frac{|s|^3}{{n^3}}
\right).
\]
Therefore, choosing \(\delta\) small enough, e.g., \(\delta = 1/9\), we have
\begin{align*}
    \sum_{s \in \cJ_n} e^{its/\sqrt{n}} \p{\oneshift = s}
& \sim \frac{1}{\sqrt{2 \pi \sigma_k^2}} \frac{1}{\sqrt{n}}
\sum_{s \in \cJ_n}
e^{its/\sqrt{n}} \exp\left\{-\frac{s^2}{2 \sigma_k n}\right\} 
%\left( 1+O\left( \frac{|s|^3}{{n^2}} \right) \right) 
\\
& = o(1) + \frac{1}{\sqrt{2 \pi \sigma_k^2}} ~ \int_{-n^{\delta}}^{n^{\delta}}
e^{itx} \exp\left\{-\frac{x^2}{2 \sigma_k^2}\right\} ~ {\mathrm d}x \\
& = o(1) + \frac{1}{\sqrt{2 \pi \sigma_k^2}} ~ \int_{-\infty}^{\infty}
e^{itx} \exp\left\{-\frac{x^2}{2 \sigma_k^2}\right\} ~ {\mathrm d}x \\
& = o(1) + \exp\left( \frac{\sigma_k^2 t^2}{2} \right).
\end{align*}
Thus the characteristic function of \(\oneshift/\sqrt{n}\) converges to
\(\exp(\sigma_k^2t^2/2)\), the characteristic function of \(\sigma_k \cZ\). It follows
from the central limit theorem that \(\oneshift/\sqrt{n}\) converges to
\(\sigma_k \cZ\) in distribution. Note that using the estimates of this section, we actually have
a local limit theorem for {\(\onecoresize\)}.

\section{The structure of the directed acyclic graph}

\label{sec:dag}

\subsection{De-randomizing the giant}

Since a \scc{} induces a sub-digraph in which each vertex has in-degree at
least one, a closed \scc{} is also a \(k\)-surjection.  Lemma~\ref{lem:k:surj}
implies that whp all \(k\)-surjections are of sizes in \(\cI_n \equiv
\inrange\). When this happens, as \(\nuk > 1/2\) (Lemma~\ref{lem:constant}),
there exists one and only one closed \scc{} and it is \(\giant\). And
if \(\giant\) is the only closed \scc{}, then every vertex 
must be able to reach it. This can be summarized as:
\begin{lemma}
    Whp \(\giantsize \in \cI_n\) and \(\giant\) is reachable from all vertices.
    \label{lem:giant:size}
\end{lemma}

Since \(\muk \equiv 1-\tauk/k \equiv 1 - \nuk\), the above lemma implies that
whp \(|\giantoutsize - \muk n| \le \err\).  Thus the structure of
\(\randdfa[\giantout]\), the sub-digraph induced by
\(\giantout\equiv[n]\setminus\giant\), should be close to that of a sub-digraph
induced by a fixed set of vertices whose size is close to \(\muk n\). Formally,
we have:
\begin{lemma}
    Let \(f_n\) be a sequence of integer-valued functions on a sequence of digraphs.  Let \(X\) be
    an integer-valued random variable.  If there exists a sequence
    \(\varepsilon_n \to 0\) such that
    \[
    \sup_{\vin \subseteq [n]:\vinsize \in \Iin}\dtv{f_n(\randdfa[\vout]), X}
    \le \varepsilon_n,
    \]
    where
    \(\vout \equiv [n] \setminus \vin\) and
    \(\dtv{\,\cdot\,,\,\cdot}\) denotes the total variation distance, then
    \[
        f_n(\randdfa[\giantout]) \inlaw X.
    \]
    \label{lem:giant:deterministic}
\end{lemma}
\begin{proof}
    Define the event \(E_n = [\giantsize \in \Iin]\). Let \(m\) be an integer,  let
    \(\vin \subseteq [n]\) be a fixed set of vertices with \(\vinsize \in
    \Iin\).  
    Recall that since \(\nuk > 1/2\), \(\vinsize > n/2 \) for large \(n\).  Thus the event \([\giant = \vin]\) depends
only on the induced sub-digraph \(\randdfa[\vin]\), which is independent of \(\randdfa[\vout]\).  Therefore the two events \([\giant = \vin]\) and
\([f_n(\randdfa[\vout]) = m]\) are independent.
    Using this observation and Lemma~\ref{lem:giant:size}, we have
    \begin{align*}
        & \p{f_n(\giantoutgraph) = m} \\
        & = \p{[f_n(\giantoutgraph) = m] \cap E_n^c} + \p{[f_n(\giantoutgraph) = m] \cap E_n} \\
        & = o(1) + \sum_{\vin \subseteq [n]:|\vin|\in\Iin} \p{f_n(\voutgraph) = m
        ~|~ \giant = \vin} \p{\giant = \vin} \\
        & \le o(1) + \sum_{\vin \subseteq [n]:|\vin|\in\Iin} 
        (\p{X = m} + \varepsilon_n) \p{\giant = \vin} \\
        & \le o(1) + \p{X=m}.
    \end{align*}
    Similarly we have \(\p{f_n(\giantoutgraph) =m } \ge \p{X=m} + o(1)\).
    Since this applies to all integers \(m\), \(f_n(\giantoutgraph) \inlaw X\).
\end{proof}

\begin{corollary}
    Let \(\cE_n\) be a sequence of sets of digraphs.  If there exists a sequence
    \(\varepsilon_n \to 0\) such that
    \[
    \sup_{\vin \subseteq [n]:\vinsize \in \Iin}\p{\voutgraph \notin \cE_n} \le
    \varepsilon_n,
    \]
    then whp \(\giantoutgraph \in \cE_n\).
    \label{cor:giant:determnistic}
\end{corollary}
\begin{proof}
    This follows from the previous lemma by taking \(X \equiv 1\) and \(f_n\) to be the
    indicator function that a digraph is in \(\cE_n\).  
\end{proof}

The rest of this section proves Theorem~\ref{thm:cycle} and
Theorem~\ref{thm:dag}. But instead of working on \(\giantout\) directly, we prove
similar theorems on fixed sets of vertices, and then apply the above lemma or its
corollary to get the final result.

\subsection{Cycles outside the giant}
\label{sec:sub:cycle}

In this subsection, we show the following:
\begin{thm}
    \label{thm:cyc:v}
    Let \(\omega_n \to \infty\) be an arbitrary sequence.
    There exists a sequence \(\varepsilon_n = o(1)\) such that for all fixed
    sets of vertices \(\vin \subseteq [n]\) with \(\vinsize \in \Iin\), we have:
    \begin{enumerate}[(a)]
        \item Let \(L_{n}^{*}\) be the length of the longest cycle in
            \(\voutgraph\). Then \(\p{L_n^{*} > \omega_n} \le \varepsilon_n\).
        \item The probability that \(\voutgraph\) contains vertex-intersecting
            cycles is at most \(\varepsilon_n\).
        \item Let \(\cycl\) be the number of cycles of length \(\ell\) in
            \(\voutgraph\).
            Let \(X_\ell = \Poi({(k \muk)^{\ell}}/{\ell})\). Then for all fixed
            \(\ell\), \(\dtv{\cycl,X_{\ell}} \le \varepsilon_n.\)
        \item Let \(\cycnum\) be the number of cycles in \(\voutgraph\). Let \(X =
            \Poi(\log\frac{1}{1-k \muk})\). Then \(\dtv{\cycnum,X} \le
            \varepsilon_n.\) As a result, \(|\p{\cycnum = 0} - (1-ke^{-\tauk})| \le 2
            \varepsilon_n\).
    \end{enumerate}
\end{thm}
\noindent Theorem~\ref{thm:cycle} follows from the above theorem and
Lemma~\ref{lem:giant:deterministic}. Our proof is inspired by
\citeauthor{Cooper2004}'s work on the directed configuration
model~\citep{Cooper2004}. Note that the Cooper-Frieze model is different from that studied by us.
In their model, both in-degrees and out-degrees are predetermined, whereas we require
all out-degrees to be \(k\) but the in-degrees are random.

The intuition behind Theorem~\ref{thm:cyc:v} is that when two cycles share
vertices, they contain fewer vertices than arcs. So if we fix the ``shape'' of a
pair of such cycles, the number of ways to label them times the probability
that they both exist is \(o(1)\).  Thus whp cycles in \(\vout\) are
vertex-disjoint and the total number of cycles has a distribution close to a
sum of independent indicator random variables.

In the following proof, instead of finding the exact
\(\varepsilon_n\), we derive implicit \(o(1)\) upper bounds for probabilities and total
variation distances which only requires that \(\vinsize \in \Iin\).
\begin{lemma} \label{lem:cyc:long}
    Let \(\cycsmall \equiv \sum_{1 \le \ell \le \omega_n}
    \cycl\). Then \(\p{\cycnum \ne \cycsmall} = o(1)\).
\end{lemma}
\begin{proof}
    Define \((x)_{\ell} \equiv x(x-1)\cdots(x-\ell+1)\). Then the number of all
    possible cycles of length \(\ell\) is \((\voutsize)_{\ell} k^{\ell}/\ell\). 
    (Note that we are also considering the labels on arcs, which makes the counting easier.)
    And
    the probability that such a cycle exists is \(n^{-\ell}\).
    Recalling that \(\voutsize \in \outrange\), we have
    \begin{align}
        \E{\cycl} 
        = 
        \frac 1 \ell (\voutsize)_\ell k^{\ell} \left( \frac{1}{n} \right)^{\ell}  
        %=
        %\frac 1 \ell (\muk n + \bigO{n^{1/2+\delta}})_\ell \left( \frac{k}{n} \right)^{\ell}  
        \le \left( k \muk \left(1+ \bigO{\errp}\right)
        \right)^\ell. \label{eq:cyc:expe}
    \end{align}
    Since \(k \muk \equiv k-\tauk < 1\) (Lemma~\ref{lem:constant}), there exists
    a constant \(c_1 < 1\) such that the above is less than
    \(c_1^{\ell}\) for \(n\) large enough.
    Since \(\cycnum \ne \cycsmall\) if and only if \(\sum_{\ell > \omega_n} \cycl \ge
    1\), 
    \[
    \p{\cycnum \ne \cycsmall} =
    \p{\sum_{\ell > \omega_n} \cycl \ge 1} \le \E{\sum_{\ell > \omega_n} \cycl}
    \le \bigO{c_1^{\omega_n} }
    = o(1).  \tag*{\qedhere}
    \]
\end{proof}

Since \(L_n^{*} > \omega_n\) if and only if \(\cycsmall \ne \cycnum\), part
\((a)\) of Theorem~\ref{thm:cyc:v} follows.  From now on let \(\omega_n = \log
\log n\). We show that:
\begin{lemma} Let \(X\) and \(X_\ell\) be as in Theorem~\ref{thm:cyc:v}. 
    Then \(\dtv{\Poi(\e{\cycsmall}), X} = o(1).\)
    And for all \(\ell \le \omega_n\),
    \(\dtv{\Poi(\e{\cycl}) ,X_\ell} = o(1).  \)
    \label{lem:cyc:poisson}
\end{lemma}
\begin{proof}
    For all \(\ell \le \omega_n\), by \eqref{eq:cyc:expe} we have
\begin{align*}
    \e{\cycl}
= \frac 1 \ell \left(\muk n + \bigO{\err}\right)_\ell k^{\ell} \left( \frac{1}{n}
\right)^{\ell}
= \frac {(k \muk)^\ell} \ell (1 + O(\ell \errp)).
\end{align*}
Thus
\begin{align*}
\e{\cycsmall} 
& = \sum_{1 \le \ell \le \omega_n} \E{\cycl} 
= \log \left(\frac{1}{1-k \muk}\right) + \bigO{\omega_n \errp}.
\end{align*}
Therefore \(\e{\cycsmall} \to \e X\) and \(\e{\cycl} \to \e X_\ell\), which implies
the lemma.
\end{proof}

\begin{proof}[Proof of Theorem~\ref{thm:cyc:v}]
By the two previous lemmas, it suffices to show that 
\[
\dtv{\cycsmall,\Poi(\e{\cycsmall})} = o(1), \quad
\dtv{\cycl,\Poi(\e{\cycl})} = o(1)\quad \text{for all fixed \(\ell\)}.
\]
We prove this by using a theorem of
\citeauthor{Arratia1989}~\citep{Arratia1989}. 
(A similar result is proved by \citet{Barbour1992poisson}).
The method is known as the Chen-Stein
method because it was first developed by \citet*{Chen1975} who applied
\citeauthor{Stein1972}'s theory~\citep{Stein1972} on probability metrics to
Poisson distributions.

Let \(\cC\) be the space of all possible cycles of length at most \(\omega_n\) in
\(\voutgraph\). For \(\alpha \in \cC\), let \(\cB_\alpha \subseteq \cC\) be the set of
cycles that are vertex-intersecting with \(\alpha\). Let \(\indd{\alpha}\) be the
indicator that a cycle \(\alpha\) appears in \(\voutgraph\). 
Define
\begin{align*}
    b_1 \equiv \sum_{\alpha \in \cC} \sum_{\beta \in \cB_\alpha} \e \indd{\alpha} \e
    \indd{\beta}, \qquad
    b_2 \equiv \sum_{\alpha \in \cC} \sum_{\beta \in \cB_\alpha:\beta\ne\alpha}
    \E{\indd{\alpha}\indd{\beta}}, \qquad
    & b_3 \equiv \sum_{\alpha \in \cC} s_\alpha,
\end{align*}
where
\[
s_{\alpha} = 
\e 
    \left| 
    \E{\indd{\alpha}| \sigma\left( \indd{\beta}:\beta \in \cC \setminus \cB_{\alpha} \right)} 
    - 
    \e \indd{\alpha} 
\right|,
\]
and \(\sigma(\cdot)\) denotes the sigma algebra generated by \((\cdot)\).
Theorem~1 of \citet{Arratia1989} states that
\[
\dtv{ \cycsmall, \, \Poi(\e{\cycsmall}) } \le 2(b_1 + b_2 + b_3).
\]
If \(\beta \in \cC \setminus \cB_{\alpha}\), then \(\alpha\) and \(\beta\)
are vertex-disjoint. Thus \(\indd{\alpha}\) and \(\indd{\beta}\) are independent and
\(s_{\alpha} = 0\) for all \(\alpha \in \cC\), i.e., \(b_3 = 0\).  It suffices
to show that \(b_1\) and \(b_2\) are \(o(1)\).

Let \(|\alpha|\) denote the length of a cycle
\(\alpha\).  Fix \(\ell_1 \le \omega_n\) and \(\ell_2 \le \omega_n\). There are
at most \(\voutsize^{\ell_1}
k^{\ell_1}\) cycles of length \(\ell_1\).  For \(|\alpha| = \ell_1\),
there are at most \(\ell_1 \voutsize^{\ell_2-1} k^{\ell_2}\) cycles of length
\(\ell_2\) that share at least one vertex with \(\alpha\).  Since
\((\voutsize)^{\ell} = (1+o(1))(\muk n)^{\ell}\) for \(\ell \le \omega_n\),
\begin{align*}
     \sum_{\alpha \in \cC:|\alpha| = \ell_1} \sum_{\beta \in \cB_\alpha:|\beta| = \ell_2} \e \indd{\alpha} \e
    \indd{\beta}
    & \le (1+o(1)) \left[(\muk n)^{\ell_1} k^{\ell_1}  \right]
    \left[\ell_1 (\muk n)^{\ell_2-1} k^{\ell_2}  \right]
    \left( \frac{1}{n} \right)^{\ell_1+\ell_2} \\
    & = (1+o(1)) \frac 1 {\muk n} \left[\ell_1 (\muk k)^{\ell_1} \right]
    \left[(\muk k)^{\ell_2}  \right].
\end{align*}
Therefore
\begin{align*}
    b_1 
    & = \sum_{1 \le \ell_1 \le \omega_n} \sum_{1 \le \ell_2 \le \omega_n} \sum_{\alpha \in \cC:|\alpha|
    = \ell_1} \sum_{\beta \in \cB_\alpha:|\beta| = \ell_2} \e \indd{\alpha} \e
    \indd{\beta} \\
    & \le (1+o(1)) \frac 1 {\muk n} \sum_{\ell_1 \ge 1} \sum_{\ell_2 \ge 1}
    \left[\ell_1 (k\muk)^{\ell_1} \right]
    \left[(k\muk )^{\ell_2}  \right] \\
    & \le (1+o(1)) \frac 1 {\muk n} 
    \left[\sum_{\ell_1 \ge 1}\ell_1  (k \muk)^{\ell_1} \right]
    \left[\sum_{\ell_2 \ge 1}  (k \muk)^{\ell_2} \right]
\end{align*}
which is \(\bigO{1/n}\) since both sums converge.

\smallskip

    \mbox{}\indent
To compute \(b_2\), we upper bound the number of pairs of vertex-intersecting cycles that could
possibly appear in \(\randdfa[\vout]\) at the same time.
Let \(\alpha\) and \(\beta\) be such a pair.
Let \(V(\alpha), A(\alpha), V(\beta), A(\beta)\) be the vertex set and (labeled) arc set of \(\alpha\) and
\(\beta\) respectively. Let \(\alpha \cup \beta\) be the digraph of vertex set \(V = V(\alpha) \cup
V(\beta)\) and arc set \(A = A(\alpha) \cup B(\beta)\). 
Assume that \(|V|=s\) and \(|A|=s+t\).
Note that as \(\alpha\) and \(\beta\) share at least one vertex, \(t \ge 1\).
Since \(V \subset [n]\), we can relabel the \(s\) vertices in \(\alpha \cup \beta\) with \([s]\)
such that the order of the vertex labels is maintained.
The result is a digraph with vertex set \([s]\) and \(s+t\) arcs labeled with
\([k]\). There are at most \( (s^2)^{s+t} k^{s+t}\) such digraphs, since there are at most
\(s^2\) choices of endpoints and \(k\) choices of labels for each of the \(s+t\) arcs.
Each digraph of this type corresponds to at most \(\binom{\voutsize}{s} \le \voutsize^{s}\) pairs of cycles like
\(\alpha\) and \(\beta\). Thus there are at most \(\voutsize^{s} (s^2)^{s+t} k^{s+t}\) such pairs.
Summing over \(s\) and \(t\), we have
\begin{align*}
    b_2 
    &
    \le 
    \sum_{1 \le s \le 2 \omega_n}
    \sum_{1 \le t \le 2 \omega_n}
    \voutsize^{s} (s^{2})^{s+t} k^{s+t} \E{\indd{\alpha}\indd{\beta}} \\
    &
    \le
    \sum_{1 \le s \le 2 \omega_n}
    \sum_{1 \le t \le 2 \omega_n}
    \left( e^{-\tauk}n + \err \right)^s 
    (2 \omega_n)^{2 \times 4\omega_n} 
    k^{s+t} 
    \frac{1}{n^{s+t}}
    \\
    &
    \le
    (2 \omega_n)^{8\omega_n}
    \sum_{1 \le s \le 2 \omega_n}
    \sum_{1 \le t \le 2 \omega_n}
    \frac{\left( n + e^{\tauk} \err \right)^s}{n^{s}} (k e^{-\tauk})^{s} \frac{k^t}{n^t}
    \numberthis \label{eq:b2} 
    \\
    &
    \le
    \bigO{\frac{1}{n}}
    (2 \omega_n k)^{8\omega_n} 
    \sum_{1 \le s \le 2 \omega_n}
    \sum_{1 \le t \le 2 \omega_n}
    (1+e^{\tauk}n^{-1/2+\delta})^{2\omega_n}
    \qquad (ke^{-\tauk} < 1/2)
    \\
    &
    \le
    \bigO{\frac{1}{n}}
    (2 \omega_n k)^{8\omega_n} 
    (2 \omega_n)^{2}
    \left( 
        1 + \bigO{n^{-1/2+\delta}\omega_n}
    \right)
    \to 0,
\end{align*}
where the last step we use that \(\omega_n = \log \log n\).

Thus part (d) of Theorem~\ref{thm:cyc:v} for \(\cycnum\) is proved.  We can prove
part (c) for \(\cycl\) using the same method by limiting \(\cC\) to contain only
cycles of a fixed length \(\ell\).  Note that the above inequality shows that the
probability that there exist vertex-intersecting cycles in \(\voutgraph\) is
\(o(1)\),
thus part (b) is also proved.
\end{proof}

The method used above can be easily adapted to prove similar
results for undirected cycles, like the following lemma which is needed in the study
of spectra in \(\giantoutgraph\):
\begin{lemma} 
    Let \(\psi_n \to \infty\) be an arbitrary sequence.
    There exists a sequence \(\varepsilon_n = o(1)\) such that
    for all fixed sets of vertices \(\vin\) with \(\vinsize \in \Iin\), we have:
    \begin{enumerate}[(a)]
        \item The probability that \(\voutgraph\) contains an undirected cycle of
            length greater than \(\psi_n\) is at most \(\varepsilon_n\). 
        \item The probability that \(\voutgraph\) contains vertex-intersecting
            undirected cycles is at most \(\varepsilon_n\).
        %\item The probability that \(\voutgraph\) contains vertex-disjoint
        %    undirected cycles connected by undirected path
        %    is at most \(\varepsilon_n\).
    \end{enumerate}
    \label{lem:cyc:undirected}
\end{lemma}
\begin{proof}
    Let \(\ucycl\) be the number of undirected cycles of length \(\ell\)
    in \(\voutgraph\). Then 
    \[
    \E{\ucycl} \le \frac 1 \ell (\voutsize)^{\ell} (2k)^{\ell} \frac 1 {n^{\ell}}
    \le \left( 2k \muk (1+ {\errp}) \right)^\ell,
    \]
    where the \(2\) comes from the fact that each edge in an
    undirected cycle has two possible directions.
    Since \(2k \muk = 2(k - \tauk) < 1\) (Lemma~\ref{lem:constant}), with exact
    the same argument of Lemma~\ref{lem:cyc:long}, we can show that
    \(\E{\sum_{\ell > \psi_n} \ucycl} = o(1)\) for all \(\psi_n \to \infty\).
    Thus (a) is proved.

    Now choose \(\psi_n = \log \log n\).
    Again we can show that whp there are no vertex-intersecting undirected
    cycles of length at most \(\psi_n\) by repeating the computation of
    \(b_2\) in the proof of Theorem~\ref{thm:cyc:v} with \(ke^{-\tauk}\) replaced
    by \(2ke^{-\tauk}\) in \eqref{eq:b2}.
%
%    For (c) note that two vertex isolated vertex undirected cycle connected by
%    a undirected path contains \(x\) vertices has exactly \(x+2\) arcs. Therefore, 
\end{proof}

\subsection{Spectra outside the giant}

In this section, we prove Theorem~\ref{thm:dag} (spectra outside the giant).
Instead of working on \(\giantout\) directly, we again prove similar results on a
fixed set of vertices and then apply Lemma~\ref{lem:giant:deterministic} to
finish the proof.

\subsubsection{The tree-like structure of some spectra}

\newcommand{\Vcond}{\vin \subseteq [n]:\vinsize\in\Iin}

We prove part (a) of Theorem~\ref{thm:dag}.  Let \(\vin \subseteq [n]\) with
\(\vinsize \in \Iin \equiv \inrange\) be a fixed set of vertices. 
For \(v \in \vout \equiv [n]
\setminus \vin\), let \(\specvout{v}\) be the spectrum of \(v\) in \(\voutgraph\), the sub-digraph induced
by \(\vout\). The following lemma shows that whp every spectrum in
\(\voutgraph\)
induces a sub-digraph that is a tree or a tree plus one extra arc:
\begin{lemma}
    We have
    \[
    \sup_{\Vcond} \p{\cup_{v \in \vout}[\arc(\randdfa[\specvout{v}]) - \specvoutsize{v} \ge 1]} 
    =o(1)
    ,
    \]
    where \(\arc(\cdot)\) denotes the number of arcs. 
    \label{lem:spec:backward}
\end{lemma}

\begin{proof}
    For \(v \in \vout\), if \(\arc(\specvoutgraph{v}) \ge \specvoutsize{v} +
    1\), then
    \(\specvoutgraph{v}\) must contain at least two undirected cycles.  By
    Lemma~\ref{lem:cyc:undirected}, whp all undirected cycles in
    \(\specvoutgraph{v}\) are vertex-disjoint. Therefore, if
    \(\specvoutgraph{v}\)
    contains two undirected cycles, then whp they are vertex-disjoint and
    connected by an undirected path. 
    
    Let \(X_{r,s,t}\) be the number of pairs of undirected cycles of length
    \(r\)
    and \(s\) respectively that are connected by an undirected path of length
    \(t\).
    In such a structure the number of arcs is \(r+s+t\) while the number of
    vertices is \(r+s+t-1\).  Since \(\vinsize \in \Iin\), we have
    \(\voutsize = n - \vinsize \in \Iout \equiv \outrange\). Thus
    \begin{align*}
        \e X_{r,s,t}
        \le (\voutsize)^{r+s+t-1} (2k)^{r+s+t} \left( \frac 1 n
        \right)^{r+s+t} 
        \le \bigO{\frac 1 n} \left( 2 k \muk + \frac{2k}{n^{1/2-\delta}} \right)^{r+s+t}.
    \end{align*}
    Summing over all possible \(r\), \(s\) and \(t\) shows that
    \begin{align*}
        \sum_{1 \le r \le n} \sum_{1 \le s \le n} \sum_{1 \le t \le n} \e X_{r,s,t}
        & \le \bigO{\frac 1 n} \sum_{1 \le r} \sum_{1 \le s} \sum_{1 \le t}
        \left(2 k \muk + \frac{2k}{n^{1/2-\delta}}\right)^{r+s+t} \\
        & \le \bigO{\frac 1 n} \left( \sum_{1 \le i} \left(2 k \muk
        +\frac{2k}{n^{1/2-\delta}}\right)^{i} \right)^{3},
    \end{align*}
    which is \(o(1)\) since the sum in the brackets converges.
\end{proof}

\subsubsection{The maximum size of spectra}

\label{sec:dag:spec:size}

This section proves part (b) of Theorem~\ref{thm:dag} (the sizes of
spectra outside the giant).
\begin{lemma}
    Let \(\varepsilon > 0\) be a constant. Then
    \[
    \sup_{\Vcond} \p{\left| \frac{\max_{v \in \vout} \specvoutsize{v}}{\log n} 
    - 
    \frac 1 {\log(1/\lambda_k)} \right| > \varepsilon} = o(1),
    \]
    where \(\lambdak \equiv (k-{\tauk})\left( \frac{{\tauk}}{k-1}
    \right)^{k-1}\). 
    \label{lem:spec:size}
\end{lemma}

The exploration of \(\specvoutgraph{v}\) can be coupled with a colouring process.
Initially, colour all vertices in \(\vin\) green, all vertices in \(\vout\) yellow,
and all arcs white. Then:
\begin{enumerate}[(i)]
    \item Colour the vertex \(v\) black, and colour the \(k\) arcs that start from
        \(v\) red. (Red arcs start from vertices in \(\specvout{v}\) but
        their endpoints are not determined yet.)
    \item Pick an arbitrary red arc. Choose its endpoint uniformly at random
        from all the \(n\) vertices. Colour this arc with the colour of its
        chosen endpoint vertex. (So a yellow arc goes to a vertex that is not
        already in \(\specvout{v}\), a black arc goes to a vertex that is already
        in \(\specvout{v}\).) If the chosen vertex is yellow, colour this vertex
        black and colour all its arcs red. 
        %(So a black vertex is in \(\specvout{v}\).)
    \item If there are no red arcs left, terminate. Otherwise go to the previous
        step.
\end{enumerate}
In the end, \(\specvout{v}\) consists of all black vertices, and arcs that start
from vertices in \(\specvout{v}\) have one of three colors: green arcs go to
\(\vin\); yellow arcs form a spanning tree of \(\specvoutgraph{v}\) rooted at
\(v\);
black arcs connect vertices in \(\specvout{v}\) but they are not part of the
yellow spanning tree, so they are in cycles in \(\specvoutgraph{v}\).
Figure~\ref{fig:colour} depicts the colouring process.

\begin{figure}[ht!]
\centering
    \begin{tikzpicture}
    \node[anchor=south west,inner sep=0] at (0,0) {\includegraphics{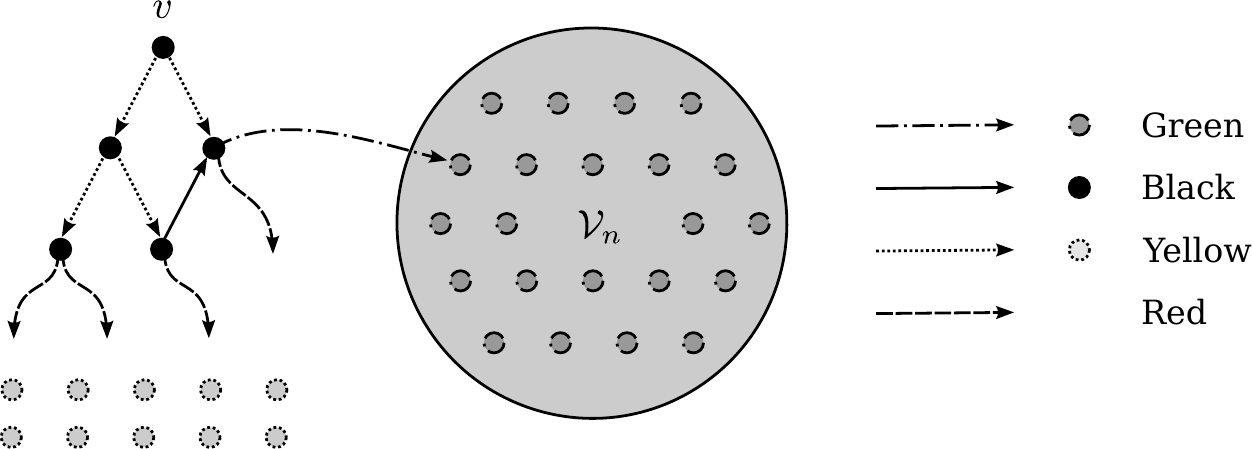}};
    %\node at (12,3.3) {Green};
    %\draw[step=1cm,gray,very thin] (0,0) grid (14,4);
    \end{tikzpicture}
    \caption{The colouring process.}
    \label{fig:colour}
\end{figure}

We use random variables \(R_t\) and \(Y_t\) to track the number of red
arcs and yellow vertices after the \(t\)-th red arc is colored.  Thus
\(R_0 = k\) and \(Y_0 = \voutsize -1\). When a red arc
is colored, if a yellow vertex is chosen as its endpoint, then
the number of red arcs increases by \((k-1)\) and the number of yellow vertices
decreases by one. Otherwise the number of red arcs decreases by one and the
number of yellow vertices remains unchanged.
Thus for \(t \ge 1\),
\[
R_t = R_{t-1} + k \xi_{t} - 1 = k \sum_{i=1}^{t} \xi_i - (t-k), 
\quad \text{and} \quad
Y_t = Y_{t-1} - \xi_{t} = \voutsize -1 - \sum_{i=1}^{t} \xi_i,
\]
where \(\xi_{t}\) are independent Bernoulli \(Y_t/n\) (the probability that a
yellow vertex is chosen).  Let \(T \equiv \min\{t:R_t \le 0\}\).  Then
\(\specvoutsize{v} = T/k\), since \(T\) is the total number arcs that have been
colored and \(\specvoutsize{v}\) is the total number of vertices that have been
colored.

\newcommand{\xio}[1]{\overline{\xi_{#1}}}
\newcommand{\xiu}[1]{\underline{\xi_{#1}}}
\newcommand{\To}{\overline{T}}
\newcommand{\Ro}{\overline{R}}
\newcommand{\Tu}{\underline{T}}

Let \((\xio{t})_{t \ge 1}\), be i.i.d.\ Bernoulli \((\muk + \errp)\).  Since \(Y_t/n
\le \voutsize/n \le \muk + \errp\), we have \(\xio{t} \succeq \xi_{t}\), where
\(\succeq\) denotes stochastically greater than (see~\citep{Shaked2007}).
Therefore there exists a coupling such that \(\xio{t} \ge \xi_{t}\) for all
\(t\)
almost surely.  Let \(\To_{t} \equiv \min\{t:k \sum_{i=1}^{t} \xio{i} - (t-k)
\le 0\}\).
Then \(\To \ge T\) almost surely.  (The random variable \(T\) is called the total progeny of a
Galton-Watson process
with offspring distribution \(\xio{1}\). For an introduction to Galton-Watson processes
see~\citep{Durrett2010probability}).
It is well know that if \(\e \xio{1} < 1\),
which is true in this case, then \(\e \To
= k /(1-\e \xio{1}) = O(1)\). Thus \(\e T = O(1)\).

\begin{proof}[Proof of the upper bound]
    Let \(\omega_n = \floor{(1+\varepsilon) \log n/ \log(1/\lambda_k)}+1\). 
    Since \(\To \ge T\),
\begin{align*}
    \p{T \ge k \omega_n} & \le \p{\To \ge k \omega_n} 
    \le \p{\frac {\sum_{i=1}^{k {\omega_n}} \xio{i}} {k \omega_n} 
    \ge \frac 1 {k_n}}
\end{align*}
where \(k_n = k \omega_n/(\omega_n -1)\). \citet*{Hoeffding1963}
showed that 
\[
\p{\frac {\Bin(m,p)} m \ge p+x} \le \left\{  
    \left(\frac{p}{p+x} \right)^{p+x}
    \left(\frac{1-p}{1-p-x} \right)^{1-p-x}
\right\}^{m}.
\]
where \(\Bin(m,p)\) denotes a binomial \((m,p)\) random variable.
Recalling that 
\(\e \xio{1} = \muk + \errp \equiv 1-\tauk/k + \errp\) and
\(\lambdak \equiv (k-{\tauk})\left( \frac{{\tauk}}{k-1} \right)^{k-1}\),
it follows from Hoeffding's inequality that \(\p{T \ge k \omega_n}\) is at most
\begin{align*}
    \left[ \left( \frac{\e \xio{1}}{1/k_n} \right) \left( \frac{1-\e
    \xio{1}}{1-1/k_n} \right)^{k_n-1} \right]^{\omega_n}
    & = \left[ (k-\tauk)\left(\frac \tauk {k-1}\right)^{k-1} + O(\errp)
        \right]^{\omega_n+O(1)} \\
        & = O(\lambda_k^{\omega_n}) \left(1  + O\left(\errp
    \right)\right)^{\omega_n} \\
    & = \bigO{n^{-(1+\varepsilon)}}.
    \label{eq:upper:bound:spectrum}
    \numberthis
\end{align*}
Since \(k \specvoutsize{v} = T\),
by the union bound
\[
\p{\cup_{v \in \vout} \specvoutsize{v} \ge \omega_n} \le n \p{\To \ge
k \omega_n} = \bigO{n^{-\varepsilon}}. \tag*{\qedhere}
\]
\end{proof}

\begin{proof}[Proof of the lower bound]

    Let \(\psi_n \equiv \ceil{(1-\varepsilon)\log n/\log(1/\lambdak)}\). To show that whp
    there exists a \(v \in \vout\) such that \(\specvoutsize{v} \ge \psi_n\), pick an
arbitrary yellow vertex and run the colouring process. If at least
\(\psi_n\) vertices are colored black (success) in the process then terminate.
Otherwise (failure) pick another yellow vertex and repeat the colouring process
until one trial succeeds.  If the colouring process is repeated for at most \(t_n
\equiv \floor{n/(\log n)^3}\) times, then at most \(a_n \equiv t_n \psi_n =
O(n/(\log n)^{2})\) vertices are colored black in the end. Therefore, the
probability that the number of red arcs increases after colouring one red arc is
at least \((\voutsize - a_n)/n\).

Let \((\xiu{i})_{i \ge 1}\) be i.i.d.\ Bernoulli \((\voutsize - a_n-\psi_n)/n\).
Let \(\Tu = \min\{t:k \sum_{i=1}^{t} \xiu{i} - (t-k) \le 0\}\). Then in each of
the first \(t_n\) iterations, the probability of a success is at least \(\p{\Tu
\ge k \psi_n} \ge \p{\Tu = k \psi_n}\). (For a detailed proof, see
\citeauthor{Van2014randomV1}'s discussion of the Erdős–Rényi
model~\citep[chap.~4.2.2]{Van2014randomV1}.)
By the hitting-time theorem of Galton-Watson processes~\citep{Van2008elementary}, 
\[
\p{\Tu = k \psi_n} 
= \frac {1} {\psi_n} \p{k \sum_{i = 1}^{k \psi_n} \xiu{i} = k (\psi_n -1)}.
\]
Since \(\sum_{i=1}^{k \psi_n} \xiu{i}\) is a binomial random variable, the
above equals
\[
\frac 1 {\psi_n} 
\binom{k \psi_n}{\psi_n -1} 
\left( \frac{\voutsize - a_n - \psi_n}{n} \right)^{\psi_n -1} 
\left( 1-\frac{\voutsize - a_n - \psi_n}{n} \right)^{k {\psi_n} - \left(
{\psi_n}-1 \right)}
%\ge \Theta\left( \frac 1 {\psi_n^{3/2}n^{1-\varepsilon}} \right)
\equiv b_n.
\]
By Stirling's approximation \cite[pp. 407]{Flajolet2009}
\[
\binom{k \psi_n}{\psi_n -1} =
\Theta(1) \binom{k \psi_n}{\psi_n} =
\frac 1 {\Theta\left( \sqrt{\psi_n} \right)}
\left[ \frac k {(1-1/k)^{k-1}} \right]^{\psi_n}.
\]
Recalling that \(a_n \equiv \bigO{n/(\log n)^2}\) and \(\psi_n \equiv
\ceil{(1-\varepsilon)\log n/\log(1/\lambdak)}\), we have,
in view of \(\voutsize =  e^{-\tauk}n + \bigO{n^{1/2+\delta}}\),
\begin{align*}
    \left( \frac{\voutsize - a_n - \psi_n}{n} \right)^{{\psi_n}-1} 
    = \left(\muk - O\left(\frac 1 {(\log n)^2} \right)\right)^{{\psi_n} - 1}
    = \Theta\left( e^{-\tauk \psi_n} \right),
\end{align*}
and
\begin{align*}
    \left( 1 - \frac{\voutsize - a_n - \psi_n}{n} \right)^{k {\psi_n} - \left( {\psi_n}-1 \right)}
    & = \left( 1- \muk + O\left( \frac 1 {(\log n)^2} \right) \right)^{k {\psi_n} -
    \left( {\psi_n}-1 \right)} \\
    & = \Theta\left( \left(\frac{\tauk}{k}\right)^{(k-1)\psi_n} \right).
\end{align*}
Recall that \(\muk \equiv 1 - \tauk/k\). Therefore
\[
\lambdak 
\equiv (k-{\tauk})\left( \frac{{\tauk}}{k-1} \right)^{k-1}
= k \muk \left( \frac{{\tauk}}{k-1} \right)^{k-1}
= \frac{k}{\left( 1-1/k \right)^{k-1}}
    e^{-\tauk}
    \left(\frac{\tauk}{k}\right)^{k-1}.
    \]
Putting everything together, we have
\begin{align*}
    b_n = \Theta\left( \frac 1 {\psi_n}
    \frac{1}{\sqrt{\psi_n}}
    \left[ \frac{k}{\left( 1-1/k \right)^{k-1}}
    e^{-\tauk}
    \left(\frac{\tauk}{k}\right)^{k-1} \right]^{\psi_n} \right) 
    = \Theta\left( \frac {\lambda_k^{\psi_n}}{\psi_n^{3/2}} \right)
    = \Theta\left( \frac {n^{-1+\varepsilon}}{\psi_n^{3/2}} \right).
\end{align*}
So the probability that all the first \(t_n \equiv \floor{n/(\log n)^3}\) trials fail is at most
\[
(1-b_n)^{t_n} \le \exp\left\{-b_n t_n \right\} = \exp\left\{ 
\Theta\left(-\frac{n^{\varepsilon}}{(\log n)^{9/2}} \right)
\right\} = o(1). \tag*{\qedhere}
\]
%Thus whp there exists a \(v \in \vout\) with \(\specvoutsize{v} \ge \psi_n\).
\end{proof}

By Lemma~\ref{lem:giant:size}, whp \(\giant\) is reachable from all vertices.
When this happens, \(\onecore \setminus \giant\) consists of vertices either on cycles
in \(\giantoutgraph\) or on paths from these cycles to \(\giant\).  Since the number of such
cycles and the length of the longest one of them are both \(O_p(1)\),
Lemma~\ref{lem:spec:size} implies that \(\onecoresize - \giantsize = O_{p}(\log
n)\). Thus 
\[
\frac{\giantsize - \nuk n}{\sqrt{n}} 
= \frac{\onecoresize - \nuk n}{\sqrt{n}} - O_p\left( \frac {\log n}{\sqrt{n}} \right)
\inlaw \cZ,
\]
which is the second part of Theorem~\ref{thm:CLL}.

In fact we can show that \(\onecoresize - \giantsize = O_p(1)\).  This seems to
be obvious since in \(\voutgraph\) the expected size of a spectrum is \(O(1)\) and
the number of cycles is \(O_{p}(1)\). However, it is not trivial because \(\ind{v
\text{ is on a cycle}}\) and \(\specvoutsize{v}\) are not independent. For a
proof using Cayley's formula, see~Lemma~\ref{lem:middle:layer} in the next section (Section
\ref{sec:mid}).

We can also use Lemma~\ref{lem:spec:size} to show that
\[
    \frac{\max_{v \in [n]} |\spec{v}|-\giantsize}{\log n} 
    \inprob 
    \frac 1 {\log(1/\lambda_k)}
    ,
\] 
which finishes the last part of Theorem~\ref{thm:CLL}, i.e., \((\max_{v \in
[n]} |\spec{v}|-\nuk n)/\sigma_k \sqrt{n} \inlaw \cZ\). Let \(A_n\) be the
event that every vertex can reach \(\giant\). Assuming \(A_n\) happens,
\(\giant \subseteq \spec{v}\) for all \(v \in [n]\).  Thus for all \(\varepsilon > 0\),
\begin{align*}
    &
    \p{
        \left| 
        \frac{\max_{v \in [n]} |\spec{v}|-\giantsize}{\log n} 
        -
        \frac 1 {\log(1/\lambda_k)}
        \right|
        > \varepsilon
    }
    \\
    &
    \le
    \p{
        \left[ 
            \left| 
            \frac{\max_{v \in [n]} |\specout{v}|}{\log n} 
            -
            \frac 1 {\log(1/\lambda_k)}
            \right|
            > \varepsilon
        \right]
        \cap
        A_n
    }
    +
    \p{A_n^c}
    =
    o(1)
    .
\end{align*}
Since \(\speconesize \le \max_{v \in [n]} |\spec{v}|\) and whp \(\speconesize
\ge \giantsize\), we also recover \citeauthor{Grusho1973}'s central limit law of
\(\speconesize\).

\subsubsection{The size of the middle layer}

\label{sec:mid}

Lemma~\ref{lem:middle:layer} and Corollary~\ref{cor:giant:determnistic} imply
that \(\onecoresize - \giantsize = O_{p}(1)\). 
\begin{lemma}
    Let \(\omega_n \to \infty\) be an arbitrary sequence of nonnegative numbers. Then
    \[
    \sup_{\Vcond} 
    \p{\sum_{v \in \cC(\vout)} \specvoutsize{v} \ge \omega_n} 
    = o(1)
    ,
    \]
    where \(\cC(\vout)\) denotes the set of vertices on cycles in
    \(\voutgraph\), and \(\specvout{v}\) is the spectrum of \(v\) in
    \(\voutgraph\), the sub-digraph induced by \(\vout\).
    \label{lem:middle:layer}
\end{lemma}
\noindent 
\begin{proof}
    By Theorem~\ref{thm:cyc:v} and Lemma~\ref{lem:spec:backward}, in
    \(\voutgraph\) whp: (a) there are at most \(\sqrt{\omega_n}\) vertices on
    cycles, i.e., \(|\cC(\vout)| \le \sqrt{\omega_n}\); (b) every
    \(\specvout{v}\) induces either a tree or a tree plus one extra arc; (c)
    \(\max_{v \in \giantout}\specvoutsize{v} = O(\log n)\).  Now assume all
    these events happen.  If \(\sum_{v \in \cC(\vout)} \specvoutsize{v} \ge
    \omega_n\), then (a) implies there is at least one vertex \(u \in
    \cC(\vout)\) with \(\specvoutsize{u} \ge {\sqrt{\omega_n}}\).  By (b),
    \(\specvout{u}\) induces a sub-digraph that consists of exactly one cycle and
    isolated trees with their roots on this cycle.  
    If \(\specvoutsize{u} = \ell\), we call the induced sub-digraph an
    \(\ell\)-eye.
    Note that by
    (c) there are no \(\ell\)-eyes with \(\ell > {(\log n)^2}\).

    \begin{figure}[ht!]
    \centering
        \begin{tikzpicture}
        \node[anchor=south west,inner sep=0] at (0,0) {\includegraphics{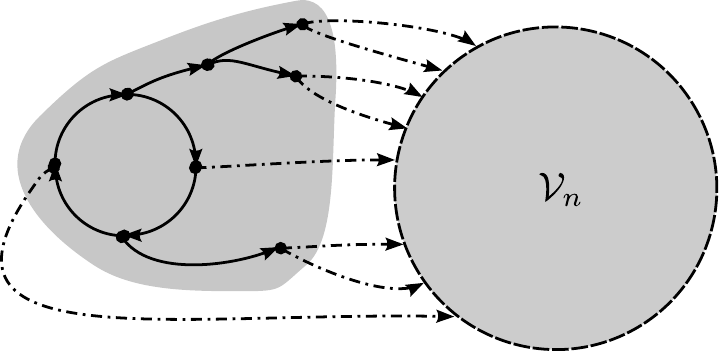}};
        %\node at (5.7,1.6) {\(\giant\)};
        %\draw[step=1cm,gray,very thin] (0,0) grid (6,6);
        \end{tikzpicture}
        \caption{The leftmost shaded part of this figure is an \(\ell\)-eye.}
        \label{fig:eye}
    \end{figure}

    Let \(\cS \subseteq \vout\) with \(|\cS|= \ell\) be a set of vertices.  If
    \(\cS\) induces an \(\ell\)-eye \(\cD_e\), then there are \(\ell\) arcs
    that start and end at specific vertices in \(\cS\) decided by \(\cD_e\),
    which happens with probability \((1/n)^{\ell}\).  If \(\cS=\specvout{u}\)
    for some vertex \(u \in \cS\), call \(\cS\) a \emph{partial spectrum}.  For
    \(\cS\) to be a partial spectrum, the other \((k-1)\ell\) arcs that start
    from \(\cS\) must end at \(\vin\), which happens with probability \(
    (\vinsize/n)^{(k-1)\ell}\).  So the probability that \(\cS\) induces a
    fixed \(\cD_e\) and \(\cS\) is a partial spectrum is \((1/n)^{\ell}
    (\vinsize/n)^{(k-1)\ell}\).  
    
    By Cayley's formula~\citep{Biggs1976}, there are \(\ell^{\ell-1}\) ways
    that \(\cS\) can form a rooted tree.  In such a tree, there are at most
    \(\ell^2\) ways to add an extra arc to make it an \(\ell\)-eye. In a
    vertex-labeled \(\ell\)-eye, there are at most \(k^{\ell}\) ways to label
    the arcs.  So the number of \(\ell\)-eyes can be induced by \(\cS\) is less
    than \(\ell^{\ell-1} \ell^{2} k^{\ell}\). And there are
    \(\binom{\voutsize}{\ell}\) ways to choose \(\cS\).

    Let \(X_{\ell}\) be the number of \(\ell\)-eyes induced by partial spectra.  
    Recall that \(\nuk \equiv \tauk/k =1- \muk\). Thus
    \(\vinsize \in \Iin \equiv \inrange\) implies that
    \(\voutsize \le \muk n +
    \err \).  So for \(\ell \le {(\log n)^{2}}\), by the
    above arguments,
    \begin{align*}
        \e X_\ell 
        & \le \binom{\voutsize}{\ell} \ell^{\ell-1} \ell^{2} k^{\ell} 
        \left( \frac 1 n \right)^{\ell} \left( \frac{\vinsize}{n}
        \right)^{(k-1)\ell}\\
        & \le \frac{(\muk n + \err)^{\ell}}{(\ell/e)^{\ell}} \ell^{\ell+1} 
        k^{\ell} \left( \frac 1 n \right)^{\ell} 
        \left( \frac{\tauk}{k} + \errp \right)^{(k-1)\ell}
        \\
        & = \left[ e\left(e^{-\tauk}+\errp\right) k \left( \frac{\tauk}{k} +
        \errp \right)^{k-1} \right]^{\ell} \ell 
        \\
        & 
        =
        \left(1 +  \bigO{\ell \errp}\right)
        {\left(k e^{1-\tauk} \left(\frac{\tauk}k\right)^{k-1}\right)^{\ell}}
        \ell
        \\
        & 
        \equiv
        \left(1 +  \bigO{\ell \errp}\right)
        \rho_k^\ell
        \ell 
        .
    \end{align*}
    By Lemma~\ref{lem:constant}, \(\rho_{k} < 1\). Since \({\sqrt{\omega_n}} \to \infty\),
    \begin{align*}
        \sum_{\sqrt{{\omega_n}} \le \ell \le (\log n)^2} \e X_\ell 
        & 
        \le 
        \left[1+\bigO{\frac{(\log n)^2}{\errpd}}\right] \sum_{\sqrt{\omega_n} \le \ell}^{\infty} \ell (\rho_k)^{\ell} 
        %=
        %\bigO{\sqrt{\omega_n} \rho_k^{\sqrt{\omega_n}}} 
        = 
        o(1).
    \end{align*}
    Thus whp there are no \(\ell\)-eyes induced by partial spectra with \(\ell \in [\sqrt{\omega_n},(\log
    n)^2]\).
\end{proof}

\subsubsection{The distance to the giant}

This subsection proves part (c) of Theorem~\ref{thm:dag}.
\begin{lemma}
    For all \(\varepsilon > 0\), 
    \[
    \sup_{\Vcond} \p{\left| \frac{\max_{v \in \vout} \treevw}{\log_k \log n} - 1 \right| > \varepsilon}
    = o(1),
    \]
    where \(\treevw \equiv \min_{u \in \vin}\dist(v,u)\), i.e., \(\treevw\) is the length of the
    shortest path from \(v\) to \(\vin\).
    \label{lem:spec:width}
\end{lemma}

Let \(v \in \vout\) be a vertex.
If \(\treevw > 1\), then all neighbors of \(v\) are in \(\vout\), and most
likely there are \(k\) of them.  So \(\p{\treevw > 1} \approx (\voutsize/n)^k
\approx e^{-\tauk k}\). If \(\treevw > 2\), then the neighbors of \(v\)'s
neighbors are all in \(\vout\), and most likely there are \(k^2\) of them. So
\(\p{\treevw > 2} \approx (\voutsize/n)^{k+k^2} \approx e^{-\tauk (k+k^2)}\).
Repeating this argument shows that \(\p{\treevw > x} \approx \exp\{-\tauk(k +
k^2 \dots k^x)\} = e^{-\tauk \Theta(k^x)}\), which is \(o(1/n)\) when \(x \ge
(1+\varepsilon) \log_k \log n\).

To make the above intuition rigorous, the colouring process defined in the
previous subsection needs to be slightly modified. Let \(v\) be the vertex
where the process has started. When choosing a red arc to colour, instead of
choosing one arbitrarily from all red arcs, choose one arbitrarily from those
that are closest to \(v\). Thus at the end, the yellow arcs consist of not just
a spanning tree but a breadth-first-search (bfs) spanning tree of
\(\specvoutgraph{v}\). If \(\vin\) (the set of green vertices) is contracted
into a single green vertex, then the green arcs together with yellow arcs form
a \dagr{}.  Let \(\treev\) denote this \dagr{}.  Then \(\treevw\) is the length
of the shortest path from \(v\) to the green vertex contracted from \(\vin\).
Figure~\ref{fig:treev} shows an example of \(\treev\).

\begin{figure}[ht!]
\centering
    \begin{tikzpicture}
    \node[anchor=south west,inner sep=0] at (0,0) {\includegraphics{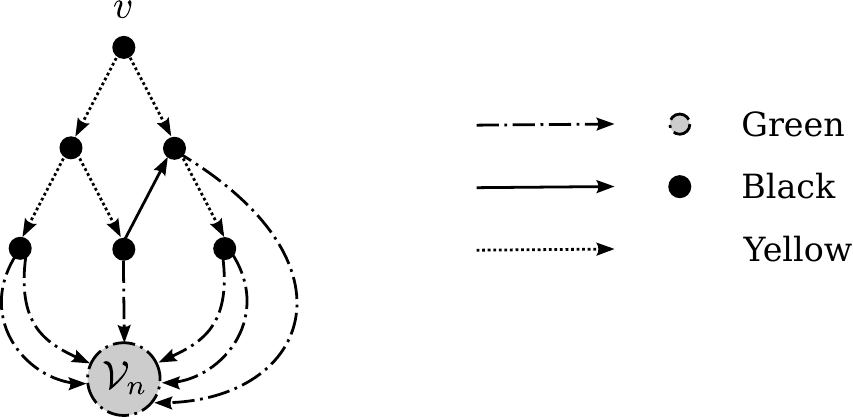}};
    %\node at (5.7,1.6) {\(\giant\)};
    %\draw[step=1cm,gray,very thin] (0,0) grid (6,6);
    \end{tikzpicture}
    \caption{An example of \(\treev\).}
    \label{fig:treev}
\end{figure}

\begin{proof}
    Let \(\omega_n = \floor{(1+\varepsilon) \log_k \log n}\).  Call the arcs
    whose endpoints are at distance \(i\) to \(v\) the \(i\)-th layer of
    \(\treev\).  The event \(\treevw > \omega_n\) implies that the first
    \(\omega_n\) layers of arcs in \(\treev\) are all yellow arcs and thus they
    form a tree of height \(\omega_n\).  By Lemma~\ref{lem:spec:backward}, whp
    there are no \(v \in \vout\) such that \(\specvoutgraph{v}\)
    contains more than one black arc.
    Thus whp in every \(\treev\) all internal (non-leaf) vertices except at
    most one have out degree \(k\).  Let \(A_n\) denote this event.
    Assuming \(A_n\) happens, \(\treevw > \omega_n\) implies that there are at
    least \(\Theta(k^{\omega_n}) = \Theta(\log n)^{1+\varepsilon}\) yellow arcs
    in the first \(\omega_n\) layers of \(\treev\). Thus in the colouring
    process, the first \(\Theta(\log n)^{1+\varepsilon}\) arcs choose their endpoints
    in \(\vout\).  The probability that this happens is at most \((\voutsize/n)^{\Theta(\log
        n)^{1+\varepsilon}}\).  Since \(\vinsize \in \Iin\), \(\voutsize = n -
        \vinsize \le \muk n + \err\).  Then by the union bound,
    \begin{align*}
        \p{\cup_{v \in \vout} [\treevw > \omega_n]} 
        & 
        \le 
        \sum_{v \in \vout} \p{[\treevw > \omega_n] \cap A_n} + \p{A_n^c} 
        \\
        &
        \le 
        n (\voutsize/n)^{\Theta(\log n)^{1+\varepsilon}} 
        + o(1)
        \\ 
        & 
        \le 
        n (\muk + \errp)^{\Theta(\log n)^{1+\varepsilon}} + o(1) 
        = o(1)
        .
    \end{align*}
    Thus whp \(\max_{v \in \vout} \treevw \le \omega_n\).

    Let \(\psi_n = \ceil{(1-\varepsilon) \log_k \log n}\).  To show that whp there is
    a vertex \(v\) with \(\treevw \ge \psi_n\), run the colouring process
    starting from an
    arbitrary yellow vertex \(v\) until either an arc is colored black or green
    (failure), or the first \(\psi_n-1\) layers of \(\treev\) are colored yellow
    (success).  So to succeed, the first \(\psi_n-1\) layers of $\treev$ form a
    full $k$-ary tree, i.e., the first \(k+k^2+\dots+k^{\psi_n-1} = \Theta(k^{\psi_n}) =
    \Theta(\log n)^{1-\varepsilon}\) arcs must be colored yellow.  If the process
    fails, we pick another yellow vertex and try again until one trial succeeds.
    Since the colouring process stops before colouring the \({\psi_n}\) layer of
    \(\treev\), each trial colors at most \(\Theta(k^{\psi_n}) = \Theta(\log
    n)^{1-\varepsilon}\) vertices black. If the process is tried at most
    \(\ceil{n/(\log n)^2}\) times, then at most \(b_n
    \equiv \ceil{n/(\log n)^2} O(\log n)^{1-\varepsilon} = O(n/(\log
    n)^{1+\varepsilon})\) vertices are colored black. Therefore, each arc has probability at least \((\voutsize -
    b_n)/n\) to be colored yellow during the first \(\ceil{n/(\log n)^2}\)
    trials.  Since $\vinsize \in \Iin$, 
    \(\voutsize = n - \vinsize \ge \muk n - \err\). Thus the probability to succeed in
    one trial is at least
    \[
        \left( \frac{\voutsize - b_n}{n} \right)^{O(\log n)^{1-\varepsilon}}
        \ge \left[ \muk - \bigO{\frac 1 {(\log n)^{1+\varepsilon}}}
            \right]^{O(\log n)^{1-\varepsilon}}
            = e^{-\bigO{\log n}^{1-\varepsilon}}
            .
    \]
        Therefore, the probability that the first \(\ceil{n/(\log n)^2}\) trials fail is
    at most
    \[ 
    \left( 1- e^{-O(\log n)^{1-\varepsilon}} \right)^{\ceil{n/(\log n)^2}} 
    \le 
    \exp\left\{- e^{-O(\log n)^{1-\varepsilon}} \frac n {(\log n)^2}\right\} 
    = o(1).
    \]
    Thus whp \(\max_{v \in \vout} \treevw \ge \psi_n\).
\end{proof}

\subsubsection{The longest path outside the giant}

This subsection proves (d) and (e) of Theorem~\ref{thm:dag}.
\begin{lemma}
    For all \(\varepsilon > 0\), we have:
    \[
        \sup_{\Vcond} 
        \p{
            \left|\frac{m(\vout)}{\log n} - \frac{1}{\log(e^{\tauk}/k) }\right| 
            > 
            \varepsilon} 
        = o(1),
    \]
    where \(m(\vout)\) denotes the length of the longest path in \(\voutgraph\);
    and
    \[
        \sup_{\Vcond} 
        \p{
            \left|\frac{d(\vout)}{\log n} - \frac{1}{\log(e^{\tauk}/k) }\right| 
            > 
            \varepsilon} 
        = o(1).
    \]
where \(d(\vout)\) denotes the maximal distance between two connected vertices
in \(\voutgraph\).
    \label{lem:path:len}
\end{lemma}
\noindent Since \(m(\vout) \ge d(\vout)\), it suffices to prove the upper bound
for \(m(\vout)\) and the lower bound for \(d(\vout)\).
\begin{proof}[Proof of the upper bound]
    Let \(\omega_n=(1+\varepsilon) \log n / \log({e^{\tauk}}/{k})\).  Let
    \(X_\ell\) be the number of labeled paths of length \(\ell\) in \(\voutgraph\).
    There are less than $\voutsize^{\ell+1}k^{\ell}$ possible such paths. Each
    of them exists with probability $(1/n)^{\ell}$. Recall that \(\vinsize \in
    \Iin\) implies \(\voutsize \le \muk n + \err\).  Thus
    \[
    \e X_\ell 
    \le \voutsize^{\ell + 1} k^{\ell} \left( \frac 1 n
    \right)^{\ell}
    \le \left(\muk n + \err \right) \left(k \muk + {k}\errp \right)^{\ell}.
    \]
    Since \(k \muk < 1\) (Lemma~\ref{lem:constant}), for \(n\) large enough,
    \begin{align*}
        \sum_{\omega_n < \ell < \voutsize} \e X_{\ell}
        \le n \sum_{\omega_n < \ell} (k \muk + k
        \errp)^{\ell}
        = \bigO{n \left( k \muk \right)^{\omega_n}}
        = \bigO{n^{-\varepsilon}}.
    \end{align*}
    Thus \(\p{m(\vout) > \omega_n} = \bigO{n^{-\varepsilon}}\).
\end{proof}
\begin{proof}[Proof of the lower bound]
    Let \( \psi_n \equiv \ceil*{(1-\varepsilon) {\log n}/{\log(1/k\muk)}}\).
    To show there are two vertices at distance within \([\psi_n, \infty)\),
    pick an arbitrary yellow vertex \(v\) and run the colouring process
    until either a vertex at distance \(\psi_n\) from $v$ has been colored
    (success), or \(\ceil{(\log n)^2}\) vertices have been colored
    (failure), or the process terminates because all vertices 
    that are reachable from \(v\) in \(\voutgraph\) has been discovered (failure).  If the
    process fails, we pick another yellow vertex and try again until one
    trial succeeds.

    If at most \(t_n \equiv \floor{n/(\log n)^4}\) trials are made, then at
    most \(\ceil{(\log n)^2} t_n = \bigO{n/(\log n)^2}\) vertices are colored.
    So in the first \(t_n\) trials, when an arc is colored, the probability
    that it is colored yellow is at least \(\mu_n \equiv
    (\voutsize-\bigO{n/(\log n)^2})/n = \muk - \bigO{1/(\log n)^2}\).  Let
    \((Z_m)_{m\ge 0}\) be a Galton-Watson process with offspring distribution
    \(\Bin(k, \mu_n)\) and \(Z_0 = 1\). In other words, \(Z_{m+1} = \sum_{j
    =1}^{Z_m} X_{m,j}\), where \((X_{m,j})_{m \ge 0, j \ge 1}\) are i.i.d.\ 
    \(\Bin(k, \mu_n)\).  Then the probability that one trial succeeds is at
    least \(\p{Z_{\psi_n} > 0}\) minus the probability that in a trial
    \(\ceil{(\log n)^2}\) vertices have been colored, which is
    \(\bigO{n^{-1-\varepsilon}}\) by \eqref{eq:upper:bound:spectrum} in
    Lemma~\ref{lem:spec:size}.

        Let \(\varphi_m(y) = \e y^{Z_m}\), i.e., \(\varphi_m(y)\) is the probability
        generating function of \(Z_m\). Thus \(\p{Z_m = 0} = \varphi_m(0)\). 
    Since \(k
    \muk < 1/2\) (Lemma~\ref{lem:constant}), for \(n\) large enough \(k \mu_n <
    1/2\). So we can apply Lemma~\ref{lem:gen:f} in the appendix to show that
    \[
    \varphi_m(0) \le 1 - (k \mu_n)^{m} +  \left( 1- \frac{1}{2^{m}} \right) (k
    \mu_n)^{m+1} < 1- \frac 1 2 (k \mu_n)^{m}, \quad \text{for all $m \ge 0$}.
    \]
    Recalling that \(\psi_n \equiv \ceil*{(1-\varepsilon) {\log
    n}/{\log(1/k\muk)}}\),
    \[ 
    \p{Z_{\psi_n} > 0} 
    = 1-\varphi_{\psi_n}(0) 
    > \frac 1 2 \left( k \muk - \bigO{\frac{1}{(\log n)^2}} \right)^{\psi_n}
    = \bigOm{n^{-1+\varepsilon}}.
    \]
    So the probability that one trial succeeds is 
    \(\bigOm{n^{-1+\varepsilon}} - \bigO{n^{-1-\varepsilon}} =
    \bigOm{n^{-1+\varepsilon}}\). (The \(\bigO{n^{-1-\varepsilon}}\) term
    is the probability that one trial colors too many vertices.)
    Thus the probability that the first \(t_n \equiv \floor{n/(\log n)^4}\) trials fail is at most
    \[
    \left(1-\bigOm{n^{-1+\varepsilon}}\right)^{t_n} 
    \le \exp\left\{-\Omega\left( \frac {1} {n^{1-\varepsilon}}
    \floor*{\frac{n}{(\log n)^4}} 
    \right)\right\} 
    = \exp\left\{-\Omega\left( \frac {n^{\varepsilon}} {(\log n)^4} \right)\right\} = o(1).
    \]
    Therefore whp \(d(\vout) \ge \psi_n\).
\end{proof}

\section{Phase transition in strong connectivity}

Now instead of assuming that \(k\) is fixed, let \(k \to \infty\) as \(n \to \infty\). Let \(K\) be
a fixed integer.
We can construct 
\(\randdfa\) by first generating \(\randdfaK\) and then adding arcs with labels in
\(\{K+1,\ldots,k\}\) into it.  By Lemma \ref{lem:giant:size}, for all \(\varepsilon > 0\), there
exists a \(K\) depending only on \(\varepsilon\) such that whp in \(\randdfaK\) the largest closed \scc{} has
size at least \( (1-\varepsilon) n\) and is reachable from all vertices.  Since adding arcs can only
increase the size of this \scc{}, whp \(\randdfa\) has a \scc{} of size at least \(
(1-\varepsilon)n \) that is reachable from all vertices.

In fact, if \(k\) increases fast enough, then whp \(\randdfa\) is strongly connected.
More precisely, \(\randdfa\) exhibits a phase transition
for strong connectivity similar to the analogous event in the Erdős–Rényi model \cite{Erdos1959random}.

\label{sec:phase}

\begin{thm}
    \label{thm:phase}
    If \(k-\log n \to -\infty\), then whp \(\randdfa\) is not strongly connected. If \(k - \log n \to
    \infty\), then whp \(\randdfa\) is strongly connected.
\end{thm}

If there is a vertex with in-degree zero, then obviously the digraph is not strongly connected. Thus
the following lemma proves the lower bound in Theorem \ref{thm:phase}.

\begin{lemma}
    \label{lem:phase:lower}
    If \(k-\log n \to -\infty\), whp \(\randdfa\) contains a vertex of in-degree zero.
\end{lemma}

\iflongversion

\begin{proof}
    Let \(\omega_n = \log n - k\). 
    For vertex \(i \in [n]\), let \(X_i\) be the indicator that \(i\) has in-degree zero. Let \(N=
    \sum_{i=1}^n X_i\).  
    We use second moment method to show that \(N \ge 1\) whp. 

    To have \(X_1 = 1\), \( nk\) arcs need to avoid vertex \(1\) as their endpoints.
    Thus
    \begin{align*}
        \e X_1 
        &
        = 
        \left(1-\frac{1}{n}  \right)^{nk}
        \ge 
        e^{
            -nk\left({1}/{n}+{1}/{n^2}  \right)
        }
        %\\
        %&
        =
        e^{-k \left({1}+{1}/{n}  \right)}
        =
        \left(
        \frac{e^{\omega_n}}{n}
        \right)^{1+1/n}
        .
    \end{align*}
    Since by assumption \(\omega_n \to \infty\),
    \(\e N = n \e X_1 = e^{\omega_n (1+1/n)}/n^{1/n} \to \infty\). 
    
    To have \(X_1 X_2 = 1\), \( nk\) arcs need to avoid vertices \(1\) and \(2\) as their
    endpoints. Thus
        \(
        \e X_1 X_2
        =
        \left( 1-{2}/{n} \right)^{nk}
        %\\
        %&
        %\le
        %\exp \left\{ -\frac{2}{n} (n-2)k - 2 \frac{k}{n} \right\}
        %\\
        %&
        %=
        %\exp \left\{ -2k\left( 1-\frac{1}{n} \right) \right\}
        %= \left[ \frac{e^{\omega_n (1-1/n)}}{n} \right]^2
        .
        \)
    Therefore
    \begin{align*}
        \frac{\E{X_1X_2}}{(\E{X_1})^2}
        &
        =
        \frac{(1-2/n)^{nk}}{(1-1/n)^{2nk}}
        =
        \left(
        \frac{n^2-2n}{n^2-2n+1}
        \right)^{nk}
        =
        \left(
        1
        -
        \frac{1}{(n-1)^2}
        \right)^{nk}
        \to
        1
        ,
    \end{align*}
    since \(nk/(n-1)^2 = o(1)\).
    Thus
    \begin{align*}
        1 
        \le \frac{\E{N^2}}{(\e{N})^2} 
        = \frac{\e{N} + n(n-1)\E{X_1X_2}}{(\e N)^2}
        \le \frac{1}{\e N} + \frac{\E{X_1 X_2}}{(\e X_1)^2}
        \to 1.
    \end{align*}
    Therefore \(\p{N = 0} \le \V{N}/(\e N)^2 = \E{N^2}/(\e N)^2 - 1 \to 0\).
\end{proof}

\fi

Given a set of vertices \(\cS\), if there are no arcs that start from \(\cS^c \equiv [n] \setminus
\cS\) and end at \(\cS\), then call \(\cS\) a \emph{non-leaf}.  If \(\randdfa\) is not strongly
connected, then there must exist a non-leaf set of vertices \(\cS\) with \(|\cS| < n\).  Thus the
following lemma implies the upper bound in Theorem \ref{thm:phase}.
\begin{lemma}
    \label{lem:phase:upper}
    If \(k-\log n \to +\infty\), whp there does not exist a non-leaf set of vertices \(\cS\) with
    \(|\cS| < n\).
\end{lemma}

\iflongversion
\begin{proof}
    By the argument at the beginning of this subsection, whp \(\randdfa\) contains a \scc{} of size
    at least \(n/2\) that is reachable form all vertices.
    So if \(|\cS| \ge n/2\), then \(\cS\) contains part of this \scc{} and cannot be a non-leaf.
    Thus it suffices to prove the lemma for \(\cS\) with \(|\cS| < n/2\).

    Let \(\omega_n = k - \log n\). For \(s \in [
        \floor{n/2}]\), let \(X_s\) be the number of non-leaf sets of vertices of size
        \(s\).  Thus 
    \begin{align*}
        \e X_s 
        &
        = 
        \binom{n}{s}\left( 1-\frac{s}{n} \right)^{k(n-s)}
        \le
        \binom{n}{s} e^{-ks(1-s/n)}
        \label{eq:phase:low}
        \numberthis
        .
    \end{align*}
    Therefore for \(s < n / \log n\),
    \begin{align*}
        \e X_s
        %&
        %\le
        %\sum_{1 \le s < \frac{n}{\log n}}
        %\binom{n}{s} e^{-ks(1-s/n)}
        %\\
        &
        \le
        \frac{n^s}{s!}
        e^{-ks(1-s/n)}
        %\\
        %&
        \le
        \frac{1}{s!}
        \left(
        \frac{n}{e^{k(1-s/n)}}
        \right)^s
        %\\
        %&
        \le
        \frac{1}{s!}
        \left(
        \frac{n}{(n e^{\omega_n} )^{1-1/\log n}}
        \right)^s
        %\\
        %&
        \equiv
        \frac{\alpha_n^s}{s!}
        .
    \end{align*}
    By assumption \(\omega_n \to \infty\). Thus \(\alpha_n \equiv
    n^{1/\log n}/e^{\omega_n(1-1/\log n)} = e^{1-\omega_n(1-1/\log n)}=o(1)\).
    Therefore,
    \begin{align*}
        \sum_{1 \le s < n/\log n}
        \e X_s
        \le 
        \sum_{1 \le s}
        \frac{\alpha_n^s}{s!}
        =
        e^{\alpha_n} - 1
        = o(1).
    \end{align*}
    On the other hand, 
    it follows from \eqref{eq:phase:low} that for \(n/\log n \le s < n/2\),
    \begin{align*}
        \e X_s 
        \le 
        \left( \frac{en}{s} \right)^s e^{-ks(1-s/n)}
        = 
        \left( \frac{en}{s e^{k(1-s/n)}} \right)^s
        \le 
        \left( \frac{en}{\frac{n}{\log n} e^{k/2}} \right)^s
        =
        \left( \frac{e \log n}{(ne^{\omega_n})^{1/2}} \right)^s
        \equiv
        \beta_n^s.
    \end{align*}
    Since \(\beta_n = e \log n/(n e^{\omega_n})^{1/2} = o(1)\),
    \begin{align*}
        \sum_{{n}/{\log n} \le s < n/2}
        \e X_s
        \le
        \sum_{1 \le s}
        \beta_n^s
        = \bigO{\beta_n}
        = o(1).
    \end{align*}
    Thus \(\p{\sum_{1 \le s < {n}/2} X_s \ge 1} \le 
    \sum_{1 \le s < {n}/2} \e X_s = o(1).
    \)
\end{proof}
\fi

\iflongversion
\else
We omit the proofs of the above two lemmas as they use standard first and second moment methods.
\fi

\section{The simple digraph model, the number of self-loops and multiple arcs}

\label{sec:simple}

A \emph{simple} digraph is one in which there are no self-loops and there is
no more than one arc from one vertex to another. Let \(\randsimple\) denote a simple \(k\)-out
digraph with \(n\) vertices chosen uniformly at random from all such digraphs. 
\(\randsimple\) can be viewed as \(\randdfa\) restricted to the event
that \(\randdfa\) is simple. 
%Inspired by \citeauthor{Van2014randomV1}'s treatment of the
%configuration model \cite[sec. 7.4]{Van2014randomV1},
This section proves the following theorem:
\begin{thm}
    \label{thm:simple}
    The probability that \(\randdfa\) is simple converges to
    \(e^{-k-\binom{k}{2}}\) as \(n \to \infty\).
\end{thm}

Theorem \ref{thm:simple} can be proved directly as follows. Let \(\indd{v}\) be the indicator that the
\(k\) arcs starting from vertex \(v\) do not end at \(v\) and do not end at the same vertex. 
Then
\[
    \p{\indd{v} = 1} = \frac{(n-1)(n-2)\cdots (n-k)}{n^{k}}
    = 1 - \frac{k(k+1)}{2n} + \bigO{\frac{1}{n^2}}
    .
\]
Since \(\randdfa\) is simple if and only if \(\cap_{v = 1}^n [\indd{v}=1]\) happens, we have
\begin{align*}
    \p{\text{\(\randdfa\) is simple}}
    &
    = \p{\cap_{v=1}^{n} \left[ \indd{v} = 1 \right]}
    = \prod_{v=1}^{n} \p{\indd{v} = 1}
    \\
    &
    = \left( 1 - \frac{k(k+1)}{2n} + \bigO{\frac{1}{n^2}}  \right)^{n}
    \to e^{-k(k+1)/2}
    = e^{-k-\binom{k}{2}}
    .
\end{align*}

However, we can say more about self-loops and multiple arcs between vertices.
Let \(\cI \equiv [n]\times[k]\).
For \( (v, i) \in \cI\), define the random variable \(\indd{v,i}\) to be the
indicator that the arc with label \(i\) starting from vertex \(v\) forms a self-loop. 
Let 
\(
    \cJ \equiv \{ (v,i,j) \in [n]\times[k]\times[k]:i <j \}.
\)
For \( (v,i,j) \in \cJ\), define the random variable \(\indd{v,i,j}\) to
be the indicator that the two arcs starting from vertex \(v\)
with labels \(i\) and \(j\) both end at the same vertex. 
Let
\(
    S_n 
    = \sum_{\alpha \in \cI} \indd{\alpha}
    \text{ and }
    M_n = \sum_{\alpha \in \cJ} \indd{\alpha}
    .
\)
Then \([S_n = 0] \cap [M_n = 0]\) if and only if \(\randdfa\) is simple.
%So Theorem \ref{thm:simple} is just a consequence of the following lemma:
\begin{lemma}
    \label{lem:simple}
    Let \(S\) and \(M\) be two independent Poisson random variables of means
    \(k\) and \(\binom{k}{2}\) respectively. Then \( (S_n, M_n) \inlaw (S,M)\) as
    \(n \to \infty\). In fact,
    \[
        \dtv{(S_n, M_n), (S,M)} = \bigO{\frac{1}{n}}.
    \]
\end{lemma}
\noindent Indeed the lemma implies that as \(n \to \infty\),
\[
    \p{\randdfa \text{ is simple}}
    =
    \p{S_n = M_n = 0}
    \to
    \p{S = 0} \p{M = 0}
    = e^{-k} e^{-\binom{k}{2}}.
\]

\begin{myRemark*}
    \citet{Bollobas1980311} proved a theorem similar to Lemma \ref{lem:simple} for the configuration model
    (see also \citet[sec. 2.4]{bollobas2001random}).
    Many authors have extended this result under various conditions, see, e.g., 
    \citet{mckay1985asymptotics, mckay1991asymptotic,janson2009, janson2014}.
    Our proof uses Stein's method, which may also be applied
    to self-loops and multiple edges in the configuration model to get proofs shorter than previous ones.
\end{myRemark*}

\begin{proof}[Proof of Lemma \ref{lem:simple}]
    We use the Chen-Stein method \cite{Chen1975}. 
    Since the probability that an arc forms a self-loop is \(1/n\),
    \[
        \e S_n = \sum_{(v,i) \in \cI} \e \indd{v,i} = kn \frac{1}{n} = k.
    \]
    Thus \(\e S = k =\e S_n \).
    Since the probability that two arcs with the same start point have the
    same endpoint is also \(1/n\),
    \[
        \e M_n 
        %= \sum_{(v,i,j) \in \cJ} \e \indd{v,i,j}
        = \sum_{v \in [n]} \sum_{1 \le i < j \le k} \e \indd{v,i,j}
        = n \binom{k}{2} \frac{1}{n} = \frac{k(k-1)}{2}.
    \]
    Thus \(\e M = k(k-1)/2 = \e M_n\).

    For \(\alpha \in \cI \cup \cJ\), let 
    \[
        \cB_{\alpha}
        = 
        \{
            \beta 
            \in \cI \cup \cJ 
            : 
            \text{$\indd{\beta}$ and $\indd{\alpha}$ are dependent} 
        \}
        .
    \]
    (Note that \(\indd{\alpha} \in \cB_{\alpha}\).)
    Define
    \begin{align*}
        b_1 
        \equiv 
        \sum_{\alpha \in \cI \cup \cJ} 
        \sum_{\beta \in \cB_{\alpha}}
        \E{\indd{\alpha}} 
        \E{\indd{\beta}}
        ,
        \quad
        b_2 
        \equiv 
        \sum_{\alpha \in \cI \cup \cJ} 
        \sum_{\beta \in \cB_{\alpha}:\alpha \ne \beta}
        \E{\indd{\alpha}\indd{\beta}} 
        ,
        \quad
        b_3 
        \equiv 
        \sum_{\alpha \in \cI \cup \cJ} s_\alpha,
    \end{align*}
    where
    \[
        s_{\alpha} 
        = \e 
        \left| 
        \E{
            \indd{\alpha} \,| \,
            \sigma\left(\indd{\beta}:{\beta} \in [\cI \cup \cJ] \setminus
            \cB_{\alpha} \right)} 
            - 
            \e \indd{\alpha}
        \right|.
    \]
    By \cite[thm. 2]{Chen1975}, if \(b_1 + b_2 + b_3 \to 0\), then 
    \( (S_n, M_n) \inlaw \allowbreak (S,M)\). Since \(\indd{\alpha}\) is independent of the random
    variables \(\indd{\beta}\) with \(\beta \in [\cI \cup \cJ] \setminus \cB_{\alpha}\), 
    we have \(s_\alpha = 0\) and thus \(b_3 = 0\).

    For \( (v,i) \in \cI\), \(\indd{v,i}\) depends on the random variables \(\indd{v,r,s}\) with \( 1 \le r < s \le k\) and
    \(i \in \{r,s\}\), of which there are \(k-1\). 
    Thus \(|\cB_{v,i}| = 1 + (k-1) = k < 2k\).
    For \( (v,i,j) \in \cJ \), \(\indd{v,i,j}\) depends on 
    \(\indd{v,i}\) and \(\indd{v,j}\).
    It also depends on the random variables \(\indd{v,r,s}\) with
    \( 1 \le r < s \le k\) and \( \{r,s\} \cap \{i,j\} \ne \emptyset\), of which there are \(2(k-1)-1= 2k-3\). 
    Thus \(|\cB_{v,i,j}| = 2 + 2k-3 < 2k\). So for all \(\alpha \in \cI \cup
    \cJ\), \(|\cB_{\alpha}| < 2k\).
    Therefore
    \begin{align*}
        b_1 
        &
        = 
        \sum_{\alpha \in \cI} \sum_{\beta \in \cB_{\alpha}} \E{\indd{\alpha}}
        \E{\indd{\beta}}
        +
        \sum_{\alpha \in \cJ} \sum_{\beta \in \cB_{\alpha}} \E{\indd{\alpha}}
        \E{\indd{\beta}}
        \\
        &
        <
        nk  \times 2k \times  \frac{1}{n} \times \frac{1}{n}
        +
        n  \binom{k}{2} \times 2k \times  
        \frac{1}{n} \times \frac{1}{n}
        = \bigO{\frac{1}{n}}
        .
    \end{align*}

    Consider \( (v,i) \in \cI\). 
    If \(\beta \in \cB_{v,i} \cap \cI\), then \(\beta = (v,i)\).
    If \(\beta \in \cB_{v,i} \cap \cJ\), then \(\beta = (v,r,s)\) for some \(
    (r,s) \) with \(i \in \{r, s\}\).
    Then \(\indd{v,i}\indd{v,r,s} = 1\) if and only if the
    two arcs starting from vertex \(v\) labeled \(r\) and \(s\) respectively both
    end at \(v\).
    Thus \(\E{\indd{v,i}\indd{v,r,s}} = 1/n^2\). Therefore
    \[
        b_{2,\cI} 
        \equiv
        \sum_{\alpha \in \cI} \sum_{\beta \in \cB_\alpha:\beta \ne \alpha}
        \E{\indd{\alpha}\indd{\beta}}
        =
        \sum_{\alpha \in \cI} \sum_{\beta \in \cB_\alpha \cap \cJ}
        \E{\indd{\alpha}\indd{\beta}}
        <
        nk \times 2k \times \frac{1}{n^2}
        = 
        \bigO{\frac{1}{n}}
        .
    \]

    Consider \( (v,r,s) \in \cJ\). If \( (v,i) \in \cB_{v,r,s}\), then \(
    (v,r,s) \in \cB_{v,i}\). 
    Thus by the above argument  \(\E{\indd{v,r,s}\indd{v,i}} = 1/n^2\).
    If \( (v,i,j) \in \cB_{v,r,s}\) and \( (i,j) \ne (r,s) \), 
    then \(|\{r,s\} \cup \{i,j\}| = 3\).
    So \(\indd{v,r,s}\indd{v,i,j} = 1\) iff the three arcs starting from vertex \(v\) with labels in \( \{r,s\} \cup
    \{i,j\} \) all end at the same vertex. Thus \(\E{\indd{v,r,s}\indd{v,i,j}} = 1/n^2\). 
    Therefore
    \begin{align*}
        b_{2,\cJ} 
        \equiv
        \sum_{\alpha \in \cJ} \sum_{\beta \in \cB_\alpha:\beta \ne \alpha} 
        \E{\indd{\alpha}\indd{\beta}}
        < 
        n \binom{k}{2} \times 2k \times \frac{1}{n^2}
        = 
        \bigO{\frac{1}{n}}
        .
    \end{align*}
    Thus \(b_2 \equiv b_{2,\cI} + b_{2, \cJ} = O(1/n)\).
\end{proof}

%A useful corollary of Theorem \ref{thm:simple} is the following:
\begin{corollary}
    \label{cor:simple}
    Let \(\cE\) be a set of digraphs. If \(\randdfa \in \cE\) whp, then
    \(\randsimple \in \cE\) whp.
\end{corollary}

\begin{proof} We have
\[
    \p{\randsimple \notin \cE} 
    =
    \p{\randdfa \notin \cE \,|\, \randdfa \text{ is simple}} 
    \le
    \frac{\p{\randdfa \notin \cE_n}}{\p{\randdfa \text{ is simple}}}
    \to 0.
    \tag*{\qedhere}
\]
\end{proof}
\noindent This corollary implies that all previous results in the form of ``whp \(\randdfa\)
\dots'' can be automatic translated into ``whp \(\randsimple\) \dots''.
For example, the statement of Theorem~\ref{thm:dag} with
\(\randdfa\) replaced by \(\randsimple\) is still true.

\begin{corollary}
    \label{cor:simple:unlabel}
    Let \(\randsimpleu\) be a digraph chosen uniformly at random from all
    simple and arc-unlabeled \(k\)-out digraphs with \(n\) vertices.  If whp
    \(\randdfa\) has property {\textbf P} where {\textbf P} does not depend on
    arc-labels, then whp \(\randsimpleu\) has property {\textbf P}.
\end{corollary}

\begin{proof}
Note that: (a) for each digraph in the space of \(\randsimpleu\), there \(
(k!)^n \) ways to arc-label it to get \( (k!)^n \) different digraphs in the
space of \(\randsimple\); (b) no two different arc-unlabeled digraphs 
can be turned into the same digraph by arc-labeling.  So there exists a \( (k!)^n
\)-to-one surjective mapping from the space of \(\randsimple\) to the space of
\(\randsimpleu\). Thus \(\randsimpleu\) can be viewed as \(\randsimple\) with
arc labels removed. 
Since {\textbf P} does not depend on
arc-labels, it follows from Corollary \ref{cor:simple} that whp \(\randsimpleu\) has property {\textbf P}.
\end{proof}

\iflongversion

\section{The typical distance}

\label{sec:typical:distance}

The typical distance \(H_n\) of \(\randdfa\) is the distance between two
vertices \(v_1\) and \(v_2\) chosen uniformly at random. If \(v_1\) cannot
reach \(v_2\), then \(H_n = \infty\). \citet{Perarnau2014} proved that
conditioned on \(H_n < \infty\), \(H_n/\log_k n \inprob 1\).  This 
section\footnote{In a shorter version of this paper, this section is omitted.}
gives an alternative proof using the path counting technique invented by
\citeauthor{Van2014randomV2} \cite[chap.~3.5]{Van2014randomV2}.

\begin{thm}[The typical distance]
    For all
    \(\varepsilon>0\),
    \[
        \p{
            \left. 
                \left| \frac{H_n}{\log_k n} - 1 \right| > \varepsilon 
                ~
            \right| 
                ~
            H_n < \infty
        } 
        = o(1).
        \]
    \label{thm:typical:dist}
\end{thm}

By Theorem~\ref{thm:CLL}, \(|\spec{v_1}|/n \inprob \nuk\), where \(\spec{v_1}\)
is the spectrum of \(v_1\).  Thus \(\p{H_n < \infty} = \p{v_2 \in \spec{v_1}} \to \nuk > 0\).  Therefore
\begin{align*}
    \p{H_n < (1-\varepsilon)\log_k n~|~H_n < \infty} 
    & = 
        \frac
        {\p{H_n < (1-\varepsilon)\log_k n}}
        {\p{H_n < \infty}} 
        \\
    & \sim 
        \frac
        {1}
        {\nuk} 
        {\p{H_n < (1-\varepsilon)\log_k n}}
    ,
\end{align*}
and 
\begin{align*}
    \p{H_n > (1+\varepsilon)\log_k n~|~H_n < \infty} 
    & = 
        \frac
        {\p{(1+\varepsilon)\log_k n < H_n < \infty}}
        {\p{H_n < \infty}} 
        \\
    & \sim 
        \frac
        {1}
        {\nuk} 
        {\p{(1+\varepsilon)\log_k n < H_n < \infty}}
    \equiv 
        \frac
        {\p{B_n}}
        {\nuk}
    .
\end{align*}
Thus it suffices to show that 
\(\p{H_n < (1-\varepsilon)\log_k n}\)
and \(\p{B_n}\) are both \(o(1)\).
\begin{lemma}[Lower bound of the typical distance]
    \label{lem:typical:dist:lower}
    For all $\varepsilon> 0$, 
    \[
        \p{H_n < (1-\varepsilon)\log_k n} = o(1)
        .
    \]
\end{lemma}
\begin{proof}
    Let \(N_{\ell}\) denote the
    number of paths from \(v_1\) to \(v_2\) of length \(\ell\).
    Consider such a path without labels on internal vertices and arcs. There
    are at most \(n^{\ell-1}\) ways to label its internal vertices and there are at
    most \(k^{\ell}\) ways to label its arcs. And the probability that such a
    labeled path appears is \( (1/n)^{\ell}\).  Thus
    \[
        \e N_{\ell} 
        \le 
        n^{\ell-1} k^{\ell} \left(\frac{1}{n}  \right)^{\ell}
        = 
        \frac{k^{\ell}}{n}
        .
    \]
    Let \(\omega_n = (1-\varepsilon)\log_k n\).  Then
    \[
        \sum_{\ell < \omega_n} \e N_{\ell} 
        \le \sum_{\ell < \omega_n} \frac{k^{\ell}}{n}
        = \frac{\bigO{k^{\omega_n}}}{n} 
        = \frac{\bigO{n^{1-\varepsilon}}}{n} 
        = o(1).
    \]
    Thus \(\p{H_n < \omega_n} = \p{\sum_{\ell < \omega_n} N_{\ell} \ge 1} = o(1)\).
\end{proof}

The rest of this section is organized as follows:
Subsection~\ref{sec:typical:distance:tree} shows that if \(v_1\) can reach
\(v_2\) but only through a very long path, then it is very likely that \(v_1\)
can reach a lot of vertices and a lot of vertices can reach \(v_2\).
Subsection~\ref{sec:typical:distance:counting} computes a lower bound of the
probability that there is a path of specific length from one large set of vertices to
another large set of vertices. Finally, subsection~\ref{sec:typical:dist:upper}
shows that these results together imply the upper bound in
Theorem~\ref{thm:typical:dist}, i.e., \(\p{B_n} = o(1)\).

\subsection{Comparison to Galton-Watson processes}

\label{sec:typical:distance:tree}

\newcommand{\Yone}{{Y}^{+}}
\newcommand{\Ytwo}{{Y}^{-}}

\newcommand{\cYone}{{\cY}^{+}}
\newcommand{\cYtwo}{{\cY}^{-}}

Let \(\cS_{m}^{+}(v)\) and \(\cS_{m}^{-}(v)\) be the sets of vertices at
distance exactly \(m\) from or to vertex \(v\) respectively.  Let \(\cS_{\le
m}^{+}(v)\) and \(\cS_{\le m}^{-}(v)\) be the sets of vertices at distance at
most \(m\) from  or to vertex \(v\) respectively. The following proposition shows
that for fixed \(m\), we can perfectly couple \((|\speconedist{t}|,
|\spectwodist{t}|)_{t=0}^{m} \) with two independent Galton-Watson processes.
It is inspired by a similar result of the
configuration model by \citeauthor{Van2014randomV2} \cite[sec.
5.2]{Van2014randomV2}, but the coupling method used here is new.

\begin{proposition}
    \label{prop:branching:approx}
    Let \( (S_t)_{t \ge 0} \) be a Galton-Watson process with a binomial offspring distribution
    \(\Bin(kn, 1/n)\). For all fixed \(m \ge 1\),
    there exists a coupling 
    \[ 
        \left[ 
        \left(k^{t}, {Y}_{t} \right)_{t=0}^{m} 
        ,
        \left(\Yone_t, \Ytwo_t \right)_{t=0}^{m}
        \right]
        ,
    \]
    of
    \((k^t, {S}_{t})_{t=0}^{m}\)
    and \(
    (|\speconedist{t}|, |\spectwodist{t}|)_{t=0}^{m} \),
    such that 
    \[
    \p{
        \left(k^{t}, {Y}_{t} \right)_{t=0}^{m} 
        \ne
        \left(\Yone_t, \Ytwo_t \right)_{t=0}^{m}
    } 
    = o(1) 
        .  
    \]
\end{proposition}

\begin{proof}
    We construct an incremental sequence of random digraphs, denoted by
    \( (\randdfaidx{t})_{t \ge 0} \), through a signal spreading process.
    Let \(\randdfaidx{0}\) be a digraph of vertex set \([n]\) that has no arcs.
    Without loss of generality, let \(v_1 = 1\) and \(v_2 = 2\).
    At time \(0\), put a \(\oplus\) signal at \(v_1\) and put a \(\ominus\) signal at \(v_2\). 
    
    If a \(\oplus\) signal reaches a vertex \(v\) at time \(t\), then at time
    \(t+1/3\) the vertex \(v\) grows \(k\) out-arcs labeled \(1,\ldots,k\)
    from itself and to \(k\) endpoints chosen independently and uar from all the \(n\)
    vertices. Then the \(\oplus\) signal splits into \(k\) \(\oplus\) signals
    and each of them picks a different newly-grown out-arc and travels
    along the arc's direction to reach its endpoint at time \(t+1\). 

    If a \(\ominus\) signal reaches a vertex \(v\) at time \(t\), then at time
    \(t+2/3\) the vertex \(v\) grows a random number \(X\) in-arcs from itself
    to \(X\) random vertices as follows: Let \( (X_{i,j})_{i \in [n],j \in
    [k]}\) be i.i.d.\ Bernoulli \(1/n\) random variables.  If \(X_{i,j} = 1\),
    then \(v\) grows an in-arc from itself to vertex \(i\) with label \(j\).
    Thus in total \(X \equiv \sum_{i\in[n],j\in[k]}X_{i,j}\) in-arcs are grown
    from \(v\).  Then the \(\ominus\) signal splits into \(X\) \(\ominus\) signals and
    each of them picks a different newly-grown in-arc and travels against the
    arc's direction to reach its starting vertex at time \(t+1\). If \(X=0\),
    then the \(\ominus\) signal vanishes.

    Let \(\randdfaidx{t}\) be the digraph generated in the above process at time \(t\).  Let
    \(\cYone_{t}\) and \(\cYtwo_{t}\) be the sets of vertices that are visited by \(\oplus\) and
    \(\ominus\) signals at time \(t\) respectively.  Let \(\cYone_{\le t}\) and \(\cYtwo_{\le t}\)
    be the sets of vertices that have been visited by \(\oplus\) and \(\ominus\) signals before time
    \(t+1\) respectively.  
    %Let \(\cM_{t-1} = [\cYone_{\le t-1} \cup \cYtwo_{\le t-1}]\).  
    At time
    \(t\), if a signal visits a vertex in \([\cYone_{\le t-1} \cup \cYtwo_{\le t-1}]\) or if two signals visit the same vertex,
    then we say a \emph{collision} happens.  Let \(T\) be the first time when a collision happens.

    Table 1 lists the types of events that make a collision happen.  Three of them need special
    attention for reasons to be clear soon.  First, if multiple \(\ominus\) signals visit the same vertex \(v\), then multiple
    arcs with the same label and \(v\) as the starting point may grow. If this happens we pick an
    arbitrary arc among them and call the others \emph{duplicate}.  Second, a \(\oplus\) signal may
    visit a vertex in \(\cYtwo_{\le T-1}\) through a newly-grown out-arc.  Finally, a \(\ominus\)
    signal may visit a vertex in \(\cYone_{\le T-1}\) through a newly-grown in-arc. We also call the
    newly-grown arcs being passed by in these two cases \emph{biased}.

    \newcommand{\drarrow}{\RightDashedArrow[densely dashed]{}}
    \newcommand{\dlarrow}{\LeftDashedArrow[densely dashed]{}}
    
\begin{table}[ht!]
    \centering
    {
        \tabulinesep=0.6mm
        \begin{tabu}{  c | c | c || c | c | c | c }
            \hline
            \multicolumn{3}{c ||}{Signals visit the same vertex}
            &
            \multicolumn{2}{c | }{Signals visit \(\cYone_{\le t-1}\)} 
            &
            \multicolumn{2}{c  }{Signals visit \(\cYtwo_{\le t-1}\)} 
            \\ \hline
            \(\oplus \drarrow \Vertex \dlarrow \oplus\) 
            & 
            \(\ominus \drarrow \Vertex \dlarrow \oplus\) 
            & 
            \cellcolor{black!5}
            \(\ominus \drarrow \Vertex \dlarrow \ominus\) 
            & 
            \(\oplus \drarrow \Outclaw\)
            & 
            \cellcolor{black!5}
            \(\ominus \drarrow \Outclaw\)
            & 
            \cellcolor{black!5}
            \(\oplus \drarrow \Inclaw\)
            & 
            \(\ominus \drarrow \Inclaw\)
            \\ 
            \hline
        \end{tabu}
    }
    \caption[Events that lead to a collision]{Events that lead to a collision. Three special types of
    events are marked.}
    \label{table:name}
\end{table}

    We construct a random \(k\)-out graph \(\randdfap\) as follows:
    First remove all duplicate and all biased arcs in \(\randdfaidx{T}\).
    Then for each pair \( (v,i) \in [n]\times[k]\), if vertex \(v\) does not
    have an out-arc labeled \(i\), then add such an out-arc with its endpoint
    chosen uar from \([n] \setminus \cYtwo_{\le T-1} \). Denote the result
    digraph by \(\randdfap\).

    The seemingly complicated \(\randdfap\) is nothing but \(\randdfa\) in disguise.  In
    \(\randdfa\), the endpoints of the arcs are chosen uar and simultaneously.
    In \(\randdfap\), the
    endpoints of the arcs
    are still chosen uar but in several steps.  First we mark the arcs whose end (start) vertices
    are at distance \(t\) to \(v_1\) (from \(v_2\)) for \(t =1,\ldots, T\).  To have
    \(\randdfap \eql \randdfa\), obviously duplicate arcs must be
    removed.  The biased arcs also cause trouble as their endpoints are chosen non-uniformly. For example, if
    at time \(T\) a \(\oplus\) signal visits a vertex in \(\cYtwo_{\le T-1}\), then an in-arc is
    added to a vertex whose in-arcs have already been decided by time \(T-1\).  Thus biased
    arcs must also be removed.  Finally, we add arcs that are still missing in \(\randdfap\) and
    choose their
    endpoints uar from \([n] \setminus \cYtwo_{\le T-1}\), i.e., from these vertices whose in-arcs
    have not yet been marked. Thus we have \(\randdfap \eql \randdfa\).  
    Let \(\Yone_t\) and \(\Ytwo_t\) be the number of vertices in \(\randdfap\) at distance \(t\) from \(v_1\) and to \(v_2\) respectively.
    Then
    \[
        (\Yone_t, \Ytwo_t)_{t=0}^{m}
        \eql
        (|\speconedist{t}|, |\spectwodist{t}|)_{t=0}^{m}.
    \]

    A \(\oplus\) signal always splits into \(k\) \(\oplus\) signals after it
    arrives at a vertex. Thus at a non-negative integer time \(t\) there are in
    total \(k^t\) \(\oplus\) signals.  On the other hand, the number of
    \(\ominus\) signals at time \(t\), denoted by \(Y_t\), is random.  Each
    time a \(\ominus\) signal splits, it splits into \(\Bin(kn,1/n)\) signals.
    Because the splits are mutually independent, \( (Y_t)_{t \ge 0}\) has the
    same distribution as \( (S_t)_{t \ge } 0\), the Galton-Watson process with
    offspring distribution \(\Bin(kn, 1/n)\).

    Assume that \(T > m\). Then the part of \(\randdfap\) within distance \(m\)
    from \(v_1\) or to \(v_2\) is determined by \(\randdfaidx{m}\).
    Thus for \(t \le m\), in \(\randdfap\) a vertex is at distance \(t\) from \(v_1\) if and
    only if  it has
    a \(\oplus\) signal at time \(t\) and a vertex is at distance \(t\) to
    \(v_2\) if and only if it has a \(\ominus\) signal at time time \(t\).
    This implies that
    \(
        \left(k^{t}, {Y}_{t} \right)_{t=0}^{m} 
        =
        \left(\Yone_t, \Ytwo_t \right)_{t=0}^{m}
    \).
    Thus to finish the proof, it suffices to show the following lemma:
    \begin{lemma}
        \label{lem:typical:dist:signal}
        For all fixed integers \(m \ge 1\), whp \(T > m\).
        %Let \(A_{m}\) be the event that there are no duplicate arcs at time \(t\) 
        %and that \(\cap_{t=0}^m [Y_{t} \le \log n]\) happens.
        %For all fixed integers \(m \ge 1\), if \(\p{A_{m-1}^c} = o(1)\), then
        %\(\p{A_m^c} = o(1)\). 
    \end{lemma}
    The intuition is that since \(m\) is fixed, 
    for \(t < m\), most likely \(|\cYone_{\le t} \cup \cYtwo_{\le t}|\)
    is small. Thus it is unlikely that a collision happens at time \(t+1\).
    See the end of this subsection for a detailed proof.
\end{proof}

\begin{corollary}
    \label{cor:typical:dist:spectrum}
    Let \(\omega_n \to \infty\) be an arbitrary sequence.
    Let \(M, \delta, \varepsilon\) be three arbitrary positive numbers. 
    Let \(\psi_n \equiv \floor{(1+\varepsilon)\log_k n}\). 
    Let
    \[
        A_n(M,m) 
        \equiv 
        \left[M \le |\speconedistm| \right] 
        \cap
        \left[M \le |\spectwodist{m}| \right] 
        \cap
        \left[|\spectwodistm| \le \omega_n \right]
        .
    \]
    Then there exists
    \(m \ge 1\) such
    that
    \[
        \limsup_{n \to \infty} \p{A_{n}^c(M,m) \cap [\psi_n < H_n < \infty]} < \delta.
    \]
\end{corollary}

\begin{proof}
    Let \(\left(k^{t}, {Y}_{t} \right)_{t=0}^{m}\) be the coupling of
    \( (|\speconedist{t}|, |\spectwodist{t}|)_{t=0}^{m} \) constructed in
    Proposition \ref{prop:branching:approx}.
    Thus \( (Y_t)_{t \ge 0}\) is a Galton-Watson process with \(\Bin(kn, 1/n)\)
    offspring distribution, i.e., \(Y_0 = 1\) and \(Y_t = \sum_{i =
    1}^{Y_{t-1}} X_{t,i}\) for \(t \ge 1\), where \(X_{t,i}\)'s are i.i.d.\ \(\Bin(kn,1/n)\).
    Since \(\e X_{1,1} = k > 1\), the survival probability of this process is a
    constant \(\eta > 0\) (see \cite[thm.~3.1]{Van2014randomV1}). For the same
    reason, \(Y_{t}/k^{t} \to Y_{\infty}\) almost surely for some random
    variable \(Y_{\infty}\) (see \cite[thm.~3.9]{Van2014randomV1}).  Since
    \(\E{ X_{1,1}^2} < \infty\), by the Kesten-Stigum Theorem
    \cite[thm.~3.10]{Van2014randomV1}, \(\p{Y_\infty > 0} = \eta\). Thus by the
    Bounded Convergence Theorem~\cite[thm. 1.5.3]{Durrett2010probability}, 
    \[
        \lim_{m \to \infty} \p{Y_m > M} 
        =
        \lim_{m \to \infty} \p{\frac{Y_m}{k^m} > \frac{M}{k^m}} 
        = \p{Y_{\infty} > 0} = \eta.
    \]
    For the same reason \(\p{Y_{m} \ge 1} \to \eta\) as \(m \to \infty\).
    Thus
    \[
        \lim_{m \to \infty} 
        \p{1 \le Y_m < M}
        =
        \lim_{m \to \infty} 
        \left( 
        \p{Y_m \ge 1} - \p{Y_m \ge M}
         \right)
        = 
        0
        .
    \]
    Thus we can choose \(m\) large enough such that
    \( \p{1 \le Y_m < M} < \delta/2 \)
    and that \(k^m \ge M\).

    Recall that \(B_n \equiv [\psi_n < H_n < \infty]\).  When \(n\) is large
    enough, \(\psi_n > m\).  Thus \(B_n\) implies that \(|\speconedist{m}| \ge
    1\).
    Define the event 
    \[
        C_n 
        \equiv 
        \left[
            \left(k^{t}, {Y}_{t} \right)_{t=0}^{m} 
            =
            \left(
            |\speconedist{t}|, |\spectwodist{t}|
            \right)_{t=0}^{m} 
        \right]
    .
    \]
    By Proposition \ref{prop:branching:approx}, \(\p{C_n^c} = o(1)\) as \(n \to
    \infty\).
    Therefore
    \begin{align*}
        \p{A_{n}(M,m)^c \cap B_n}
        &
        \le
        \p{C_n^c}
        +
        \p{A_{n}(M,m)^c \cap C_n \cap B_n}
        \\
        &
        \le 
        o(1) 
        + 
        \p{
            [k^m < M] 
            \cup
            [1 \le Y_m <M]
            \cup
            \left[
                \omega_n < \sum_{t=0}^m Y_{t} 
            \right]
        }
        \\
        &
        \le 
        o(1) 
        + \p{k^m < M} 
        + \p{1 \le Y_{m} < M}
        + \p{\omega_n < \sum_{t=0}^m Y_{t}}
        \\
        &
        =
        o(1)
        + 0
        + \delta/2
        + o(1)
        ,
    \end{align*}
    where the last equality is due to our choice of \(m\) and that 
    \(
    \E{\sum_{t=0}^m Y_{t}} = \sum_{t=0}^{m} k^t = O(1).
    \)
\end{proof}

\begin{proof}[Proof of Lemma~\ref{lem:typical:dist:signal}]
    Recall that \(\cYone_t\) and \(\cYtwo_t\) are the sets of vertices that are reached
    at time \(t\) by a \(\oplus\) signal or \(\ominus\) signal respectively. 
    Let \(\cM_{m-1} = \cup_{t=0}^{m-1} [\cYone_t \cup \cYtwo_t]\). 
    Define event \(A_m \equiv \cap_{i \in [4]} E_{m,i}\) where \(E_{m,i}\)'s
    are defined as follows:
    \begin{itemize}
        \item \(E_{m,1}\) 
            --- The out-arcs that grow from vertices in
            \(\cYone_{m-1}\) all end at different vertices in \([n] \setminus
            \cM_{m-1}\). 
            Thus at time \(m\) all \(\oplus\) signals visit
            different vertices and these vertices have never been visited by
            signals before.
        \item \(E_{m,2}\) 
            --- There are no in-arcs that grow from vertices in
            \(\cYtwo_{m-1}\) that have starting vertices in \(\cM_{m-1} \cup
            \cYone_{m}\). 
            Thus at time \(m\) all \(\ominus\) signals visit vertices that have
            never been visited by signals before and that are not reached by
            \(\oplus\) signals at time \(m\).
        \item \(E_{m,3}\) 
            --- There are no two in-arcs that grow from vertices in
            \(\cYtwo_{m-1}\) that have the same starting vertex.  Thus at time \(m\)
            all \(\ominus\) signals reach different vertices.
        \item \(E_{m,4}\) 
            --- \(|\cYtwo_m| \le (\log n)^{m}\).  
    \end{itemize}
    \newcommand{\Amall}{\cap_{t=0}^{m} A_{t}}
    \newcommand{\Ammall}{\cap_{t=0}^{m-1} A_{t}}
    The event \(A_t\) implies that no collision happens at time \(t\).  Thus \(\Amall\)
    implies that no collision has happened by time \(m\), and thus \(T > m\).  We show by
    induction that \(\p{\Amall} = 1-o(1)\).

    Since \(|\cYtwo_{0}|=1\) and there are no arc-growing before time \(0\),
    \(\p{A_0} = 1\), which is the induction basis.  Now assume that \(\p{\Ammall} =
    1-o(1)\).
    Then
    \[
        \p{\Amall} = \p{ A_{m} \, | \, \Ammall} \p{\Ammall} = \p{A_m\,|\,
        \Ammall}(1-o(1)).
    \]
    Thus it suffices to show that
    \[
        \p{A_m^c\,|\,\Ammall}
        =
        \p{[\cup_{i\in[4]} E_{m,i}^c]\,|\,\Ammall}
        \le
        \sum_{i \in[4]} \p{E_{m,i}^c|\Ammall}
        = o(1)
        .
    \]
    %by showing that \(\p{E_{m,i}^c\,|\,\Ammall} = o(1)\) for all
    %\(i \in [4]\).

    The event \(\Ammall\) implies that
    \[
        |\cM_{m-1}| 
        \le
        \sum_{t = 1}^{m-1} |\cYone_{t}|
        + 
        \sum_{t = 1}^{m-1} |\cYtwo_{t}|
        \le 
        \sum_{t = 1}^{m-1} k^{t} 
        +
        \sum_{t = 1}^{m-1} (\log n)^{t}
        = \bigO{\log n}^{m}
        .
    \]
    For \(E_{m,1}\) to happen, the \(k^{m}\) arcs that grow out of \(\cYone_{m-1}\)
    must end at different vertices in \([n] \setminus \cM_{m-1}\). Thus
    \begin{align*}
        \p{E_{m,1}|\Ammall} 
        =
        \prod_{0 \le i < k^m} 
        \left[ 
            \frac{n - |\cM_{m-1}| - i}{n} 
        \right]
        \ge 
        \left[ 
            1 
            - 
            \frac{\bigO{\log n}^m}{n} 
        \right]^{k^m}
        = 1 - o(1)
        .
    \end{align*}

    For \(E_{m,2}\) to happen, the vertices in \(\cYtwo_{m-1}\) cannot grow in-arcs
    that have starting vertex in in \(\cM_{m-1} \cup \cYone_{m}\).
    \(\Ammall\) implies that \(|\cYtwo_{m-1}| \le (\log n)^{m-1}\).
    Since deterministically \(|\cYone_{m}| = k^m\), \(|\cM_{m-1} \cup
    \cYone_{m}| =\bigO{\log n}^{m}\).
    Thus the number of in-arcs that need to not grow at time \(m-1/3\) to make
    sure that \(E_{m,2}\) happens is at most
    \[
        k 
        |\cYtwo_{m-1}| 
        |\cM_{m-1} \cup \cYone_{m}| 
        =
        \bigO{\log n}^{2m}
        .
    \]
    Since an in-arc does not grow with probability \(1-1/n\),
    \begin{align*}
        \p{E_{m,2}\,|\,\Ammall}
        \ge
        \left(  
            1
            -
            \frac{1}{n}
        \right)^{\bigO{\log n}^{2m}}
        = 
        1 - o(1)
        .
    \end{align*}

    Let \(X_v\) be the number of in-arcs that grow from \(\cYtwo_{m-1}\) and
    that have starting vertex \(v\).  Conditioned on \(\cYtwo_{m-1}\), \(X_v
    \eql \Bin(k|\cYtwo_{m-1}|, 1/n)\).  Since \(\Ammall\) implies
    \(|\cYtwo_{m-1}| \le (\log n)^{m-1}\), 
    \begin{align*}
        \p{X_v \le 1\,|\,\Ammall}
        &
        \ge
        \p{\Bin\left( k (\log n)^{m-1}, \frac{1}{n} \right) \le 1}
        \\
        &
        =
        \left( 1-\frac{1}{n} \right)^{k(\log n)^{m-1}}
        +
        k(\log n)^{m-1}
        \frac{1}{n}
        \left( 1-\frac{1}{n} \right)^{k(\log n)^{m-1}-1}
        \\
        &
        =
        1 - \bigO{\frac{(\log n)^{2(m-1)}}{n^2}}
        .
    \end{align*}
    Since for two different vertices \(u\) and \(v\), \(X_u\) and \(X_v\)
    depend on disjoint set of arcs, \( (X_{u})_{u \in [n]}\) are mutually
    independent. Thus
    \begin{align*}
        \p{E_{m,3}|\Ammall} 
        &
        =
        \p{\cap_{v \in [n]} [X_v \le 1]\,|\,\Ammall} 
        \\
        &
        \ge
        \left(  
        1 - 
        \bigO{\frac{(\log n)^{2(m-1)}}{n^2}}
        \right)^{n}
        = 
        1 - 
        o(1)
        .
    \end{align*}

    Since \( (|\cYtwo_{t}|)_{t \ge 1} \) is a Galton-Watson process with a \(\Bin(kn,
    1/n)\) offspring distribution, \(\e |\cYtwo_m| = k^m\). Thus
    \(\p{|\cYtwo_m| > (\log
    n)^m} = o(1)\). Therefore
    \[
        \p{E_{m,4}^c|\Ammall} 
        \equiv
        \p{|\cYtwo_m| > (\log n)^{m}|\Ammall} 
        \le 
        \frac
        { \p{|\cYtwo_m| > (\log n)^{m}} }
        {\p{\Ammall}} =
        o(1)
        ,
    \]
    where the last equality is due to the induction assumption that \(\p{\Ammall} = 1-o(1)\).
\end{proof}

\subsection{Path counting}

\label{sec:typical:distance:counting}

\newcommand{\asize}{|\cA|}
\newcommand{\bsize}{|\cB|}
\newcommand{\nlab}{N_{\ell}}

For three disjoint sets of vertices \(\cA, \cB, \cC \subseteq [n]\), let
\(\nlab\) denote the number of paths of length \(\ell\) that
start from \(\cA\) and end at \(\cB\), and that have all internal
vertices in \(\cC\). 
In the next subsection, we use the second moment method to lower bound
\(\p{\nlab \ge 1}\), which requires estimates of \(\E{\nlab}\) and
\(\V{\nlab}\). 
The following lemma does so by using the path counting technique
\cite[chap.~3.5]{Van2014randomV2}.
\begin{proposition}
    \label{prop:typical:dist:path}
    Let \(\omega\), \(\ell\) and \(M\) be three positive integers, possibly depending on \(n\).  Let \(\cA, \cB, \cC
    \subseteq [n]\) be disjoint sets of vertices with \(|\cA| = |\cB| = M \ge 1\) and
    \(|\cC| \ge n-\omega\). 
    There exist constants \(C_1\) and \(C_2\) such that
    \[
        \e \nlab 
        \ge
        \frac{k^{\ell}M^{2}}{n} \left( 1 - {\frac{(\omega+\ell)\ell}{n}} \right)
        ,
        %\le \frac{k^{\ell}M^2}{n},
        \numberthis
        \label{eq:n:expc}
    \]
    and
    \[
        \V{\nlab}
        \le \e \nlab 
        + C_1 {\frac{k^{2\ell} M^3}{n^2}}
        + C_2 {\frac{k^{2\ell} M^4\ell^4}{n^3} 
        }
        .
        \numberthis
        \label{eq:n:var}
    \]
\end{proposition}

\begin{proof}[Proof of \eqref{eq:n:expc}]
    Note that if \(n \le (\omega + \ell) \ell\), then \eqref{eq:n:expc} is trivially true.
    So we assume that \(n > (\omega + \ell) \ell\).
    We simplify by contracting \(\cA\) and \(\cB\) into to two special vertices
    \(v_a\) and \(v_b\). The vertex \(v_a\) has out-degree \(kM\) and the
    vertex \(v_b\) has probability
    \(M/n\) to be chosen as the endpoint of each arc. Consider an unlabeled path
    of length \(\ell \ge 1\) from \(v_a\) to \(v_b\). There are \(kM\) ways to label
    the first arc. There are \(k^{\ell-1}\) ways to label the other
    arcs.  Recall that \( (x)_y \equiv (x-1)(x-2)\cdots(x-y+1) \).  There are
    \( (|\cC|)_{\ell-1}\) ways to label the internal vertices of the path. The
    probability that a vertex-and-arc labeled path of length \(\ell\) from
    \(v_a\) to \(v_b\) exists is \( (1/n)^{\ell-1} (M/n) \). 
    %Thus
    %\[
    %    \e 
    %    \nlab 
    %    = 
    %    (kM)k^{\ell-1} (|\cC|)_{\ell-1}
    %    \left(\frac{1}{n}\right)^{\ell-1} 
    %    \left( \frac{M}{n} \right)
    %    \le
    %    \frac{k^{\ell} M^{2}}{n} 
    %    .
    %    \]
    Thus
    \begin{align*}
        \e \nlab 
        &
        = 
        (kM)k^{\ell-1} (|\cC|)_{\ell-1}
        \left(\frac{1}{n}\right)^{\ell-1} 
        \left( \frac{M}{n} \right)
        \\
        &
        \ge 
        \frac{k^{\ell}M^{2}}{n} \left( 1 - \frac{\omega+\ell}{n}
        \right)^{\ell}
        \ge \frac{k^{\ell}M^{2}}{n} \left( 1 - {\frac{(\omega+\ell)\ell}{n}} \right)
        ,
    \end{align*}
    where the last step is because \( (1-x)^y \ge 1 - xy \) when \(x \ge 0, y \ge 1\).
\end{proof}

\begin{proof}[Proof of \eqref{eq:n:var}]
    \newcommand{\inda}{\indd{\alpha}}
    \newcommand{\indb}{\indd{\beta}}

    \newcommand{\ala}[1]{a_{#1}^{[\alpha]}}
    \newcommand{\alv}[1]{v_{#1}^{[\alpha]}}
    \newcommand{\bea}[1]{a_{#1}^{[\beta]}}
    \newcommand{\bev}[1]{v_{#1}^{[\beta]}}

    Let \(\cL\) be the space of all possible arc-and-vertex labeled paths of length \(\ell\) from
    \(v_a\) to \(v_b\) through \(\cC\). 
    In other words, if \(\alpha \in \cC\), then
    \[ 
        \alpha 
        = 
        \left(
        \alv{0}\equiv v_a, \, \ala{0}, \,
        \alv{1}, \, \ala{1}, \,
        \ldots, \,
        \alv{\ell-1},\,  \ala{\ell-1}, \,
        \alv{\ell} \equiv v_{b}
        \right),
    \]
    where 
    \(\ala{0},\ldots, \ala{\ell-1}\) are arc labels and \(\alv{1},\ldots,
    \alv{\ell-1}\) are different vertex labels in \(\cC\).
    For \(\alpha \in \cL\), let
    \(\indd{\alpha}\) be the indicator that \(\alpha\) appears.  
    Given two paths \(\alpha, \beta \in \cL\), call them \emph{arc-disjoint} if
    there does not exist an \(i\) such that \(\alv{i} = \bev{i}\) and \(\ala{i}
    = \bea{i}\).
    If
    two paths \(\alpha\) and \(\beta\) are arc-disjoint, then \(\inda\) and
    \(\indb\) are independent, since they depend on the endpoints of two
    disjoint sets of arcs.
    Let \(\alpha \sim \beta\) denote
    that \(\alpha\) and \(\beta\) are not arc-disjoint and that \(\alpha\)
    and \(\beta\) can both appear simultaneously. Then
    \begin{align*}
        \V{\nlab} 
        & = \sum_{\alpha,\beta \in \cL}\left(\E{\indd{\alpha} \indd{\beta}}
        - \E{\indd{\alpha}} \E{\indd{\beta}}  \right) \\
        & \le \sum_{\alpha,\beta \in \cL}\ind{\alpha \sim \beta}\left[\E{\indd{\alpha} \indd{\beta}}
        - \E{\indd{\alpha}} \E{\indd{\beta}}  \right] \\
        & \le \e \nlab + \sum_{\alpha,\beta \in \cL}
        \ind{\alpha \sim \beta}
        \ind{\alpha \ne \beta}
        \E{\indd{\alpha} \indd{\beta}} \\
        & \equiv \e \nlab + I
        .
    \end{align*}
    %Thus it suffices to prove that \(I_{\ell} = \bigO{{k^{2\ell}M^3}/{n^2}}\).

    To bound \(I\), we use a technique called path counting.  Consider
    two paths \(\alpha,\beta \in \cL\) with \(\alpha \sim \beta\)
    and \(\alpha \ne \beta\). First
    colour all vertices and arcs in \(\alpha\) and \(\beta\) white.  Then
    colour all vertices and arcs shared by \(\alpha\) and \(\beta\) black.
    After this, \(\alpha\) and \(\beta\) both contain the same number, say
    \(m\), of white paths separated by black paths (possibly a single black
    vertex).  Since both \(\alpha\) and \(\beta\) start and end with black
    paths, each of them contains \(m+1\) black paths.  
    Define:
    \begin{enumerate}
        \item \(\vec{x}_{m+1} = (x_1,\ldots, x_{m+1})\), where \(x_i \ge 0\) denotes the length of the
            \(i\)-th black path in \(\alpha\).
        \item \(\vecs = (s_1, \ldots, s_m)\), where \(s_i > 0\) denotes the
            length of the \(i\)-th
            white path in \(\alpha\).
        \item \(\vect = (t_1, \ldots, t_m)\), where \(t_i > 0\) denotes the
            length of the \(i\)-th white path in \(\beta\).
        \item \(\vec{o}_{m+1} = (o_1, \ldots, o_{m+1})\) records the order in which black
            paths appear in \(\beta\). Note that \(o_1 \equiv 1\),
            \(o_{m+1} \equiv m+1\),
            and \((o_2, \ldots, o_{m})\)
            is a permutation of \(\{2,\ldots,m\}\).
    \end{enumerate}
    Define the shape of \(\alpha\) and \(\beta\) by \(\sh(\alpha,\beta)
    \equiv
    (\vec{x}_{m+1},\vecs,\vect,\vec{o}_{m+1}).\)

    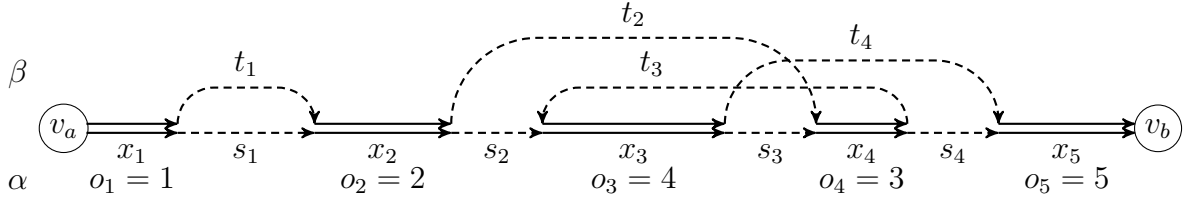
\begin{figure}[ht!]
        \label{fig:path:counting}
        \begin{tikzpicture}[>=stealth',scale=0.6]
            \tikzset{black path/.style={thick, ->}}
            \tikzset{white path/.style={densely dashed,thick, ->}}
            \tikzset{x label/.style={below, align=center}}

            \draw[black path] (0,0) coordinate (astart) {} -- 
            node[x label] {\(x_1\) \\[-3pt] \(o_1=1\)} 
            (2,0);
            \draw[black path] (5,0) -- 
            node[x label] {\(x_2\) \\[-3pt] \(o_2=2\)}
            (8,0);
            \draw[black path] (10,0) -- 
            node[x label] {\(x_3\) \\[-3pt] \(o_3=4\)}
            (14,0) coordinate (aend) {};
            \draw[black path] (16,0) -- 
            node[x label] {\(x_4\) \\[-3pt] \(o_4=3\)}
            (18,0);
            \draw[black path] (20,0) -- 
            node[x label] {\(x_5\) \\[-3pt] \(o_5=5\)}
            (23,0);

            \draw[white path] (2,0) -- node[below] {\(s_1\)} (5,0);
            \draw[white path] (8,0) -- node[below] {\(s_2\)} (10,0);
            \draw[white path] (aend) -- node[below] {\(s_3\)} ++(2,0);
            \draw[white path] (18,0) -- node[below] {\(s_4\)} (20,0);

            \begin{scope}[shift={(0,0.2)}]

            \draw[black path] (0,0) coordinate (bstart) {} -- (2,0);
            \draw[black path] (5,0) -- (8,0);
            \draw[black path] (10,0) -- (14,0) coordinate (bend) {};
            \draw[black path] (16,0) -- (18,0);
            \draw[black path] (20,0) -- (23,0);

            \draw[white path] (2,0) to[out=90,in=180] ++(0.8,0.8) -- 
            node[above] {\(t_1\)}
            ($(5,0) + (-0.8,0.8)$) to[out=0,in=90] (5,0);
            \draw[white path] (8,0) to[out=90,in=180] ++(1.9,1.9)
            -- node[above] {\(t_2\)}
            ($(16,0)+(-1.9,1.9)$) to[out=0, in=90] (16,0);
            \draw[white path] (18,0) to[out=90,in=0] ++(-0.8,0.8) -- 
            node[near end, above] {\(t_3\)}
            ($(10,0) + (0.8,0.8)$) to[out=180,in=90] (10,0);
            \draw[white path] (14,0) to[out=90,in=180] ++(1.4,1.4)
            -- node[above] {\(t_4\)}
            ($(20,0)+(-1.4,1.4)$) to[out=0, in=90] (20,0);

            \node at (-1.5,1.1) {\(\beta\)};

            \end{scope}

            \node[circle,draw=black, inner sep=2pt] at (-0.5,0.1) {\(v_a\)};

            \node[circle,draw=black, inner sep=2pt] at (23.5,0.1) {\(v_b\)};

            \node at (-1.5,-1.1) {\(\alpha\)};

            %\draw[step=1cm,gray,very thin] (0,0) grid (20,6);

        \end{tikzpicture}
        \caption{A pair of paths and their shape.}
    \end{figure}

    Let \(r\) be the number of arcs shared by \(\alpha\) and \(\beta\), i.e.,
    \(r \equiv \sum_{i=1}^{m+1} x_i\). Since \(\alpha \sim \beta\) and \(\alpha
    \ne \beta\), \(1 \le r < \ell\).  
    Thus there are \(\ell -r\) white arcs in \(\alpha\).  
    Since each white path contains at least one white arc,
    there are at most \(\ell-r\) white paths in \(\alpha\), i.e., \(m \le \ell-r\).  As
    \(\alpha\) and \(\beta\) must differ by at least one arc, \(m \ge 1\). Let
    \(\cS_{m, r}\) denote the set of shapes of two paths in \(\cL\)
    that share \(r\) arcs and each contains \(m\) white paths. Then
    \(I\) can be expressed as a sum over \(r\), \(m\) and \(\cS_{m, r}\)
    by
    \begin{align*}
        I
        =
        \sum_{1 \le r < \ell} 
        \sum_{1 \le m \le \ell-r}
        \sum_{\sigma \in \cS_{m, r}}
        \sum_{\alpha,\beta \in \cL}
        \ind{\sh(\alpha,\beta)=\sigma}
        \E{\indd{\alpha}\indd{\beta}} 
        \equiv
        \sum_{1 \le m < \ell}
        \sum_{1 \le r < \ell-m} 
        \sum_{\sigma \in \cS_{m, r}}
        J_{m,r,\sigma}
        .
    \end{align*}

    Now fix \(m,r\) and a shape \(\sigma = (\vec{x}_{m+1}, \vecs, \vect,
    \vec{o}_{m+1}) \in \cS_{m, r}\).
    Consider arcs in two paths \(\alpha, \beta \in \cL\) with \(\cS(\alpha,\beta) =
    \sigma\).  Call those starting
    from \(v_a\) \(a\)-arcs, those ending at \(v_b\) \(b\)-arcs, and other arcs
    middle-arcs.  Let \(z_a \equiv \ind{x_1 = 0}\) and \(z_b \equiv
    \ind{x_{m+1} = 0}\). In other words, \(z_a\) is the indicator that
    \(\alpha\) and \(\beta\) do not share an \(a\)-arc, and \(z_b\) is
    the indicator that they do not share a \(b\)-arc. Then \(\alpha\) and
    \(\beta\) contain \(1+z_a\) \(a\)-arcs and \(1+z_b\) \(b\)-arcs.
    Since \(\alpha\) and \(\beta\) are both of length \(\ell\) and they share
    \(r\) arcs, they contain \(2\ell - r\) arcs in total. Thus they
    contain \(2\ell-r-(1+z_a)-(1+z_b)=2\ell -r -z_a-z_b-2\) middle-arcs.

    Recall that black paths are shared by \(\alpha\) and \(\beta\).  Since the
    \(i\)-th black path is of length \(x_i\), it contains \(x_i+1\) black
    vertices.  So the number of vertices shared by the two paths is
    \(\sum_{i=1}^{m+1} (x_i+1) = r + m+ 1\).  Therefore in total there are
    \(2(\ell+1) - r - m - 1\) vertices in the two paths, and among them \(2\ell
    -r -m -1\) are internal vertices.

    The above argument shows that, given two unlabeled path of the shape
    \(\sigma\), there are at most \( n^{2\ell-r-m-1}\) ways to choose the
    internal
    vertices. 
    There are at most\( (kM)^{1+z_a} \) ways to label \(a\)-arcs. There
    are \( k^{2 \ell - r- z_a - z_b -2 }\) ways to label middle-arcs. There are
    at most
    \( k^{1+z_b}\) ways to label \(b\)-arcs. Thus
    \begin{align*}
        |\{(\alpha,\beta) \in \cL \times \cL: \sh(\alpha,\beta) = \sigma\}| 
        & \le
        n^{2\ell-r-m-1}
        (kM)^{1+z_a}
        k^{2 \ell - r- z_a - z_b -2 }
        k^{z_b + 1}
        \\
        & = 
        n^{2\ell-r-m-1}
        M^{1+z_a}
        k^{2 \ell - r}
        .
    \end{align*}
    And the probability that a pair of paths with shape \(\sigma\) does appear is
    \[
        \left( \frac{1}{n} \right)^{1+z_a} 
        \left( \frac{1}{n} \right)^{2\ell - r - z_a -z_b - 2} 
        \left( \frac{M}{n} \right)^{1+z_b}
        = 
        \frac{M^{1+z_b}}{n^{2\ell-r}}
        .
        \]
    Together,
    \begin{align*}
        J_{m,r,\sigma}
        &
        \equiv
        \sum_{\alpha,\beta \in \cL}
        \ind{\sh(\alpha,\beta)=\sigma}
        \E{\indd{\alpha}\indd{\beta}} 
        \le
        n^{2\ell-r-m-1}
        M^{1+z_a}
        k^{2 \ell - r}
        \frac{M^{1+z_b}}{n^{2\ell-r}}
        \\
        &
        = 
        \frac{k^{2\ell-r} M^{2+z_a+z_b}}{n^{m+1}}
        \equiv K_{m,r,z_a,z_b}
        \numberthis
        \label{eq:path:a:b}
        .
    \end{align*}

    Let \(\cS_{m,r,z_a,z_b}\) be the set of shapes with parameters
    \(m,r,z_a,z_b\). Then we have \(\cS_{m,r} = \cup_{z_a,z_b \in \{0,1\}}
    \cS_{m,r,z_a,z_b}\), where the sets in the union are disjoint.
    Thus
    \begin{align*}
        I
        &
        =
        \sum_{1 \le m < \ell}
        \sum_{z_a,z_b \in \{0,1\}}
        \sum_{1 \le r < \ell-m} 
        \sum_{\sigma \in \cS_{m, r,z_a,z_b}}
        J_{m,r,\sigma}
        \\
        &
        \le 
        \sum_{1 \le m < \ell}
        \sum_{z_a,z_b \in \{0,1\}}
        \sum_{1 \le r < \ell-m} 
        |\cS_{m, r,z_a,z_b}|
        K_{m,r,z_a,z_b}
        \\
        & 
        =
        \sum_{z_a,z_b \in \{0,1\}}
        \sum_{1 \le r < \ell-m} 
        |\cS_{1, r,z_a,z_b}|
        K_{1,r,z_a,z_b}
        +
        \sum_{2 \le m < \ell}
        \sum_{z_a,z_b \in \{0,1\}}
        \sum_{1 \le r < \ell-m} 
        |\cS_{m, r,z_a,z_b}|
        K_{m,r,z_a,z_b}
        \\
        &
        \equiv 
        I^{[1]} + I^{[\ge 2]}
        .
    \end{align*}
    By counting the choices of \(\vec{x}_{m+1}, \vecs, \vect,
    \vec{o}_{m+1}\), we can upper bound \(|\cS_{m,r,z_a,z_b}|\):
    \begin{lemma}
        If \(m \ge z_a + z_b\), then 
        \begin{equation}
            |\cS_{m,r,z_a,z_b}| 
            =
            (r+1)^{m-z_a-z_b} 
            \binom{\ell-r-1}{m-1}
            \binom{\ell-r-1}{m-1}
            (m-1)!.
            \label{eq:path:shape}
        \end{equation}
        If \(m < z_a + z_b\), then \(|\cS_{m,r,z_a,z_b}| = 0\).
        \label{lem:typical:dist:smrab}
    \end{lemma}

\begin{proof}[Proof of Lemma~\ref{lem:typical:dist:smrab}]
    First consider \(m \ge 2\), which implies that \(m \ge z_a + z_b\).
    When \(z_a = 1\), \(x_1 = 0\). 
    When \(z_b = 1\), \(x_{m+1} = 0\). 
    Thus the number of ways to choose \(\vec{x}_{m+1}\) equals the number of
    ways to choose \(m+1-z_a-z_b \ge 1\) ordered non-negative integers such that they
    sum to \(r\), which is well known to be \( (r+1)^{m-z_a-z_b}\),
    which explains the first factor in \eqref{eq:path:shape}.  Similarly the
    second term and the third term are the numbers of ways to
    choose \(\vecs\) and \(\vect\) respectively. The last term is the number of ways to
    choose \(\vec{o}_{m+1}\) since \(o_2,\ldots,o_m\) is a permutation of
    \(\{2,\ldots,m\}\).

    Now assume \(m = 1\). If \(z_a + z_b \le m = 1\), the above argument still
    works.  If \(z_a + z_b > 1\), then \(z_a = z_b = 1\). In other words, the
    two paths do not share arcs at the beginning and at the end, and they must
    meet at least one internal vertex. So in this shape, there must be at least
    two white sub-paths in each of the two paths, i.e., \(m \ge 2\), which is a
    contradiction. Therefore, \(S_{1,r,1,1} = \emptyset\).
\end{proof}

\begin{lemma} 
    \(I^{[1]} \le 6{{ k^{2\ell}M^3}/{n^2}}\).
    \label{lem:typical:dist:Ilmone}
\end{lemma}
\begin{proof}[Proof of Lemma~\ref{lem:typical:dist:Ilmone}]
    By \eqref{eq:path:a:b} and the above lemma,
    \begin{align*}
        \sum_{1 \le r < \ell-1} 
        |\cS_{1,r,0,0}|
        \times
        K_{1,r,0,0}
        & 
        =
        \sum_{1 \le r < \ell-1} 
        (r+1)
        \left[\binom{\ell-r-1}{0}  \right]^{2}
        0!
        \frac{k^{2\ell-r} M^{2}}{n^{2}}
        \\
        &
        \le
        \frac{k^{2\ell}M^2}{n^2}
        \sum_{1 \le r} 
        \frac{r+1}{k^{r}}
        \le
        \frac{k^{2\ell}M^2}{n^2}
        \left[ 
        \sum_{1 \le r} 
        \frac{1}{2^{r}}
        +
        \sum_{1 \le r} 
        \frac{r}{2^{r}}
        \right]
        \\
        &
        =
        \frac{k^{2\ell}M^2}{n^2}
        \left(
            1
            + 
            \frac{1}{2}
            +
            \sum_{2 \le r} 
            \frac{r}{2^{r}}
        \right)
        \le
        4 {\frac{k^{2\ell}M^2}{n^2}}
        ,
    \end{align*}
    where the last step is because \(\sum_{2\le r} r/2^r \le
    \int_{1}^{\infty} x/2^x {\mathrm d}x \le 2\).
    Similarly,
    \begin{align*}
        \sum_{1 \le r < \ell-1} 
        |\cS_{1,r,0,1}|
        \times
        K_{1,r,0,1}
        &
        =
        \sum_{1 \le r < \ell-1} 
        |\cS_{1,r,1,0}|
        \times
        K_{1,r, 1,0}
        \\
        & 
        =
        \sum_{1 \le r < \ell-1} 
        (r+1)^{0}
        \left[\binom{\ell-r-1}{0}  \right]^{2}
        0!
        \frac{k^{2\ell-r} M^{3}}{n^{2}}
        \\
        &
        \le
        \frac{ k^{2\ell}M^3}{n^2}
        \sum_{1 \le r} 
        \frac{1}{k^{r}} 
        \\
        &
        \le
        \frac{ k^{2\ell}M^3}{n^2}
        \sum_{1 \le r} 
        \frac{1}{2^{r}} 
        =
        {\frac{ k^{2\ell}M^3}{n^2}}
        .
    \end{align*}
    Also by Lemma~\ref{lem:typical:dist:smrab}, \(\cS_{1, r, 1,1} = \emptyset\).
    Thus
    \begin{align*}
        I^{[1]}
        &
        \equiv
        \sum_{z_a,z_b \in \{0,1\}}
        \sum_{1 \le r < \ell-1} 
        |\cS_{1,r,z_a,z_b}| \times
        K_{1,r, z_a,z_b}
        \\
        &
        \le
        4 {\frac{ k^{2\ell}M^2}{n^2}}
        +
        2 {\frac{ k^{2\ell}M^3}{n^2}}
        +
        0
        \le
        6 {\frac{ k^{2\ell}M^3}{n^2}}
        .
        \tag*{\qedhere}
    \end{align*}
    %where the last step is because \(M \ge 1\).
\end{proof}

\begin{lemma}
    \(I^{[\ge 2]} = 4{{\ell^4 k^{2\ell}M^{4}}/{n^{3}}}\).
    \label{lem:typical:dist:Ilmtwo}
\end{lemma}

\begin{proof}[Proof of Lemma~\ref{lem:typical:dist:Ilmtwo}]
    By Lemma~\ref{lem:typical:dist:smrab}, for \(r \in [1,\ell)\),
    \begin{align*}
        &
        \sum_{z_a,z_b \in \{0,1\}}
        |\cS_{m,r,z_a,z_b}|
        \times
        K_{m,r,z_a,z_b}
        \\
        &
        =
        \sum_{z_a,z_b \in \{0,1\}}
        (r+1)^{m-z_a-z_b} 
        \left[\binom{\ell-r-1}{m-1} \right]^2
        (m-1)!
        \frac{k^{2\ell-r} M^{2+z_a+z_b}}{n^{m+1}}
        \\
        &
        \le
        \ell^{m}
        \frac{\ell^{2(m-1)}}{(m-1)!}
        \frac{k^{2\ell-r} }{n^{m+1}}
        \sum_{z_a,z_b \in \{0,1\}}
        M^{2+z_a+z_b}
        \\
        &
        \le
        \frac{\ell^{3m-2}k^{2\ell-r}}{(m-1)!n^{m+1}} 4{M^4}
        .
    \end{align*}
    Thus
    \begin{align*}
        \sum_{1 \le r < \ell-m}
        \sum_{z_a,z_b \in \{0,1\}}
        |\cS_{m,r,z_a,z_b}|
        \times
        K_{m,r,z_a,z_b}
        &
        \le
        \sum_{1 \le r < \ell-m}
        \frac{\ell^{3m-2}k^{2\ell-r}}{(m-1)!n^{m+1}} 4{M^4}
        \\
        &
        \le
        \frac{\ell^{3m-2}k^{2\ell}}{(m-1)!n^{m+1}} 4{M^4}
        \sum_{1 \le r } \frac{1}{k^r}
        \\
        &
        \le
        \frac{\ell^{3m-2}k^{2\ell}}{(m-1)!n^{m+1}} 4{M^4}
        .
    \end{align*}
    Therefore,
    \begin{align*}
        I^{[\ge 2]} 
        &
        \equiv
        \sum_{2 \le m <\ell}
        \sum_{1 \le r < \ell-m}
        \sum_{z_a,z_b \in \{0,1\}}
        |\cS_{m,r,z_a,z_b}|
        \times
        K_{m,r,z_a,z_b}
        \\
        &
        \le
        \sum_{2 \le m}
        \frac{\ell^{3m-2}k^{2\ell}}{(m-1)!n^{m+1}} 4M^4
        \\
        &
        \le
        \frac{\ell k^{2\ell}4M^{4}}{n^{2}}
        \sum_{2 \le m}
        \frac{\ell^{3(m-1)}}{n^{m-1}(m-1)!}
        \\
        &
        \le
        \frac{\ell k^{2\ell}4M^{4}}{n^{2}}
        \left( \exp\left\{\frac{\ell^{3}}{n}\right\}-1 \right)
        \le
        4
        {
            \frac{\ell^4 k^{2\ell}M^{4}}{n^{3}}
        }
        .
        \tag*{\qedhere}
    \end{align*}
\end{proof}

By Lemma~\ref{lem:typical:dist:Ilmone} and Lemma~\ref{lem:typical:dist:Ilmtwo},
    \begin{align*}
        I
        =
        I^{[1]} + I^{[\ge 2]}
        \le
        6
        {
            \frac{ k^{2\ell}M^3}{n^2}
        }
        +
        4
        {
            \frac{\ell^4 k^{2\ell}M^{4}}{n^{3}}
        }
        .
    \end{align*}
    Thus \(\V{N_{\ell}} \le \E{N_{\ell}} + I
        =
        \E{N_{\ell}}
        +
        6
        {
            {k^{2\ell}M^3}/{n^2}
        }
        +
        4
        {
            {\ell^4 k^{2\ell}M^{4}}/{n^{3}}
        }
    \).
\end{proof}

\subsection{Finishing the proof of Theorem~\ref{thm:typical:dist}}

\label{sec:typical:dist:upper}

\begin{proof}[Proof of the upper bound of the typical distance]
    %If \(\varepsilon < \varepsilon'\), then \(\p{(1+\varepsilon) \log n < H_n} \ge
    %\p{(1+\varepsilon') \log n < H_n}\).
    We can assume \(\varepsilon < 1/2\).
    Recall that \(\psi_n \equiv \floor{(1+\varepsilon)\log_k n}\) 
    and that \(B_n = [\psi_n < H_n < \infty] \).  
    As argued at the beginning of this section, to finish the proof of
    Theorem~\ref{thm:typical:dist}, it suffices to show that 
    \(\p{B_n} = o(1)\).

\newcommand\vecS{\vec{\cS}}

    Let \(\omega_n \equiv \psi_n\).  Let \(M, m\) be two positive integers which are picked later.
    Recall that \(\cS_{i}^{+}(v)\) and \(\cS_{i}^{-}(v)\) are the sets of vertices at distance exactly
    \(i\) from or to vertex \(v\) respectively, and that \(\cS_{\le i}^{+}(v)\) and \(\cS_{\le
    i}^{-}(v)\) are the sets of vertices at distance at most \(i\) from 
    or to \(v\) respectively. The following argument shows that by properly
    choosing \(M\) and \(m\), the probability that there exists a path of length
    exactly \(\psi_n - 2m\) from \(\speconedistm\) to \(\spectwodistm\) is at
    least \(1-\delta\) for \(n\) large enough, where \(\delta > 0\) is
    arbitrary and fixed.

    \begin{figure}[ht!]
    \centering
        \begin{tikzpicture}
        \node[anchor=south west,inner sep=0] at (0,0)
        {\includegraphics[scale=0.9]{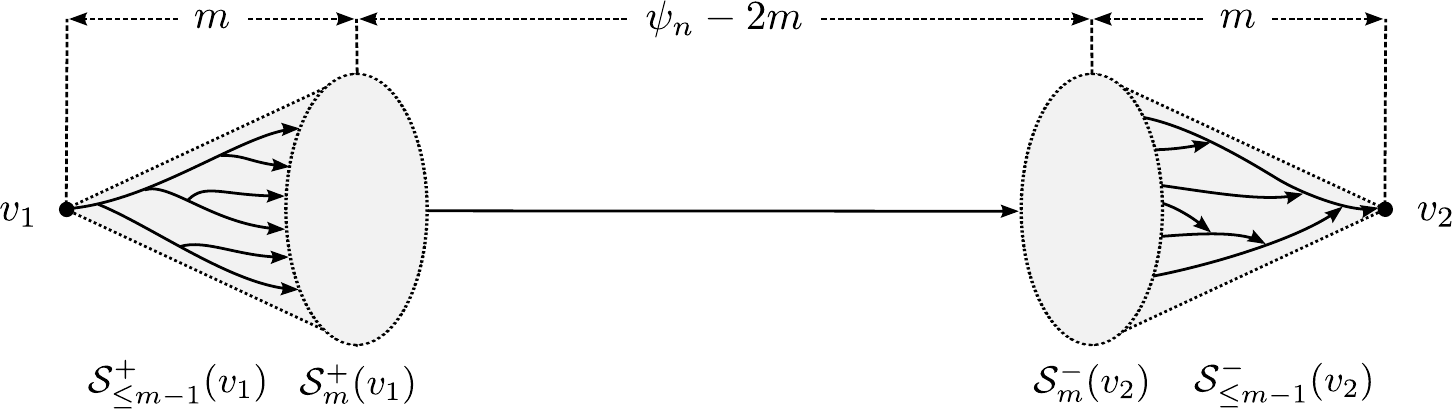}};
        %\node at (5.7,1.6) {\(\giant\)};
        %\draw[step=1cm,gray,very thin] (0,0) grid (6,6);
        \end{tikzpicture}
        \caption{\(\speconedist{\le m-1}, \speconedistm\), and
    \(\spectwodistm\).}
        \label{fig:distance}
    \end{figure}

    Let the event \(A_n(M,m)\) be defined as in
    Corollary~\ref{cor:typical:dist:spectrum}, i.e.,
    \[
        A_n(M,m) 
        \equiv 
        \left[M \le |\speconedistm| \right] 
        \cap
        \left[M \le |\spectwodist{m}| \right] 
        \cap
        \left[|\spectwodistm| \le \omega_n \right]
        .
    \]
    %where \( \left(2^{[n]} \right)^{3} \equiv 2^{[n]} \times 2^{[n]} \times 2^{[n]}\)
    %and \(2^{[n]}\) is the power set of \([n]\).
    Since each vertex has out-degree exactly \(k \ge 2\), deterministically,
    \[
        |\speconedist{\le m-1}| \le 1 + k + \cdots + k^{m-1} < k^m, \qquad
        |\speconedist{m}| \le k^m.
    \] 
    Since \(\psi_n > 2m\) for \(n\) large enough,
    \(B_n\) implies \(\speconedist{\le m}\) and \(\spectwodist{\le m}\) are
    disjoint. Thus the event \(A_n(M,m) \cap B_n\) implies
    that \( ( \speconedist{\le m-1}, \speconedistm, \spectwodist{m}, \spectwodist{\le m-1}) \in \cA \),
    where \(\cA\) is a set of quadruples of disjoint sets of vertices defined by
    \begin{align*}
        \cA 
        \equiv 
        \{
            (\cS_1, \cS_2, \cS_3, \cS_4):\,
            &
            %\text{$\cS_1$, $\cS_2$, $\cS_3$ are disjoint};
            v_1 \in \cS_1;
            v_2 \in \cS_4;
            \\
            &
            |\cS_1| < k^{m};
            M \le |\cS_2| \le k^m; 
            M \le |\cS_3|;
            |\cS_3 \cup \cS_4| \le \omega_n
        \}
        .
    \end{align*}
    For \(\vecS = (\cS_1, \cS_2, \cS_3, \cS_4)\in \cA\), define the event
    \[
        A'_{n}(\vecS) 
        \equiv 
        \left[
            \speconedist{\le m-1} = \cS_1
            \right]
        \cap
        \left[
            \speconedistm = \cS_2
            \right]
        \cap
        \left[
            \spectwodist{m} = \cS_3
            \right]
        \cap
        \left[
            \spectwodist{\le m-1} = \cS_4
            \right]
            .
    \]
    Thus \([B_n \cap A_{m}(M,m)] \subseteq \cup_{\vecS \in \cA} [B_n \cap A'_{n}(\vecS)]\) and the events in
    the union are disjoint.

    Now fix a \(\vecS \in \cA\).  Let \(\cA_{\vecS}\) and \(\cB_{\vecS}\) be
    arbitrary subsets of \(\cS_2\) and \(\cS_3\) respectively with
    \(|\cA_{\vecS}| = M\) and \(|\cB_{\vecS}| = M\).  Let \(N_{\vecS}\) be the
    number of paths of length \(\psi_n - 2m\) that start from \(\cA_{\vecS}\)
    and end at \(\cB_{\vecS}\), and that contain internal vertices only in
    \(\cC_{\vecS} \equiv [n] \setminus \cup_{i\in[4]} \cS_i\).  
    Thus there are  
    \(
        \left|
            %[n]
            %\setminus 
            %(\cup_{i \in [3]} \cS_i) 
        \cC_{\vecS}
        \right| 
        =
        n-|\cup_{i \in [4]} \cS_i|
        \ge n - (\omega_n + 2k^{m})
    \)
    vertices that can be internal vertices of
    these paths.  
    %Since \(\psi_n \equiv \floor{(1+\varepsilon)\log_k n}\),
    %\(k^{\psi_n} \in [n^{1+\varepsilon}/k, n^{1+\varepsilon}]\).
    By \eqref{eq:n:expc} of Proposition~\ref{prop:typical:dist:path}, 
    \begin{align*}
        \e N_{\vecS}
        &
        \ge
        \frac{k^{\psi_n - 2m} M^2}{n} 
        \left( 
        1 
        - 
        \frac
        {(\omega_n + 2k^{m} + \psi_n - 2m)(\psi_n - 2m)}
        {n} 
        \right)
        \\
        &
        \ge
        \frac{k^{(1+\varepsilon)\log_k(n) - 1 - 2m} M^2}{n} 
        \left( 
        1 
        - 
        \frac
        {2 \psi_n^2}
        {n} 
        \right)
        \ge
        \frac{n^{\varepsilon}M^2}{k^{2m+1}} \frac{1}{2}
        ,
    \end{align*}
    for \(n\) large enough.
    %It also follows from \eqref{eq:n:expc} that
    %\[
    %    \e N_{\vecS} 
    %    \le
    %    \frac{k^{\psi_n - 2m}M^2}{n}
    %    \le 
    %    \frac{k^{(1+\varepsilon)\log_k(n) - 2m}M^2}{n}
    %    = 
    %    \frac{n^{\varepsilon}M^2}{k^{2m}}
    %    .
    %\]
    By \eqref{eq:n:var} of Proposition~\ref{prop:typical:dist:path}, 
    \begin{align*}
        \V{N_{\vecS}} 
        &
        \le 
        \e N_{\vecS} 
        + C_1 {
            \frac
            {k^{2(\psi_n-2m)}M^3}
            {n^2}
        }
        + C_2 {
            \frac
            {k^{2(\psi_n-2m)}M^4(\psi_n - 2m)^4}
            {n^3}
        }
        \\
        &
        \le
        \e N_{\vecS} 
        + C_1
        {\frac{n^{2(1+\varepsilon)}M^3}{n^2 k^{4m}}}
        + 
        C_2 {\frac{n^{2(1+\varepsilon)}M^4\psi_n^4}{n^3 k^{4m}}}
        \\
        &
        \le
        \e N_{\vecS} 
        + C_1 {
            \frac
            {n^{2\varepsilon} M^3}
            {k^{4m}}
        }
        + C_3 {\frac{M^4}{k^{4m}}} 
        {\frac{ (\log n)^{4}}{n^{1-2\varepsilon}}}
        ,
    \end{align*}
    where \(C_3\) is a constant that does not depend on \(M\) or \(m\).
    Thus
    \begin{align*}
        \p{N_{\vecS} = 0} 
        &
        \le 
            \frac
            {\V{N_{\vecS}}}
            {\left( \e N_{\vecS} \right)^2}
        \le 
            \frac{2 k^{2m+1}}{n^{\varepsilon}M^2}
            + 
            \frac
            {C_1 {n^{2\varepsilon} M^3k^{-4m}}}
            {\left( n^{\varepsilon} M^22^{-1}k^{-2m-1} \right)^2}
            +
            \frac
            {C_3 {{M^4 (\log n)^{4}n^{2\varepsilon-1}} {k^{-4m}}}}
            {\left( n^{\varepsilon} M^22^{-1}k^{-2m-1} \right)^2}
        \\
        &
        \le
            \frac{2 k^{2m+1}}{n^{\varepsilon}M^2}
            + 
            \frac
            {4 k^2 C_1}
            {M}
            +
            \frac
            {4 k^2 C_3 (\log n)^{4}}
            {n}
            .
    \end{align*}
    Later \(m\) is chosen solely depending on \(M\). Thus
    we can pick \(M\) large enough such that for \(n\) large enough,
    \(\p{N_{\vecS} = 0} \le \delta / 2\) for all \(\vecS \in \cA\).
    %And this bound does not depend on \(m\).

    If \(H_n > \psi_n\), then there cannot exist paths of length \(\psi_n -
    2m\) from \(\speconedistm\) to \(\spectwodist{m}\). Thus \(B_n \cap
    A_n'(\vecS)\) implies that \([N_{\vecS} = 0] \cap A_n'(\vecS)\).
    A crucial observation is that 
    \[\p{N_{\vecS}=0\left|A'_{n}(\vecS)\right.} \le \p{N_{\vecS}=0}.\] This is because \(A'_{n}(\vecS)\)
    implies that arcs starting from vertices in \(\cC_{\vecS}\) cannot choose vertices in
    \(\spectwodist{\le m-1} = \cS_{4}\) as their endpoints. 
    Whereas when we compute \(\p{N_{\vecS}=0}\) without any condition, arcs starting from vertices in
    \(\cC_{\vecS}\) are allowed to choose all vertices as their endpoints.
    Thus some of these arcs are possibly ``wasted'' by choosing their endpoints in \(\cS_4\).
    This increases the probability that \(N_{\vecS} = 0\).
    %A crucial observation here is that
    %the two events \(N_{\vecS} = 0\) and \(A'_n(\vecS)\) are independent
    %because they depend on disjoint sets of arcs --- the former depends on arcs
    %that start from vertices in \(A_{\vecS}\) and from vertices outside
    %\(\cup_{i \in [3]} S_i\), whereas the latter depends on arcs that start
    %from vertices in \(S_1 \cup S_3\).
    Thus
    \begin{align*}
        \p{B_n \cap A'_n(\vecS)}
        &
        \le
        \p{[N_{\vecS} = 0] \cap A'_n(\vecS)}
        =
        \p{N_{\vecS} = 0\left|A'_n(\vecS)\right.}\p{A'_n(\vecS)}
        \\
        &
        \le
        \p{N_{\vecS} = 0} 
        \p{A'_n(\vecS)}
        \le
        \frac{\delta}{2}
        \p{A'_n(\vecS)}
        .
    \end{align*}
    Therefore
    \begin{align*}
        \p{B_n \cap A_n(M,m)}
        &
        \le
        \sum_{\vecS \in \cA}
        \p{B_n \cap A'_n(\vecS)}
        \le
        \frac{\delta}{2}
        \sum_{\vecS \in \cA}
        \p{A'_n(\vecS)}
        \\
        &
        \le
        \frac{\delta}{2}
        \p{(\speconedist{\le m-1}, \speconedist{m}, \spectwodist{m}, \spectwodist{\le m-1}) \in \cA}
        \le 
        \frac
        \delta
        2
        .
    \end{align*}
    By Corollary~\ref{cor:typical:dist:spectrum}, we can choose \(m\)
    depending on \(M\) such that for \(n\) large enough,
    \(\p{B_n \cap A_n^c(M,m)} < \delta/2\).
    Thus
    \[
        \limsup_{n \to \infty}
        \p{B_n} 
        =
        \limsup_{n \to \infty}
        \left(
            \p{B_n \cap A_n(M,m)}
            +
            \p{B_n \cap A_n^c(M,m)}
        \right)
        \le \delta.
        \tag*{\qedhere}
    \]
\end{proof}

\fi

\section{Extensions}

\label{sec:extension}

\iflongversion
\else
The typical distance \(H_n\) of \(\randdfa\) is the distance between two
vertices \(v_1\) and \(v_2\) chosen uniformly at random. If \(v_1\) cannot
reach \(v_2\), then \(H_n = \infty\). \citet{Perarnau2014} proved that
conditioned on \(H_n < \infty\), \(H_n/\log_k n \inprob 1\).  We have found an
alternative proof for this result using the path counting technique invented by
\citeauthor{Van2014randomV2} \cite[chap. 3.5]{Van2014randomV2}. The proof can
be found in a longer version of this paper at [\url{http://arxiv.org/abs/1504.06238}].
\fi

\citet{Perarnau2014} also proved that the diameter of the giant component
divided by $\log n$ converges in probability to $1/\log(k) +
1/\log(1/\lambda_k)$. Recall that the longest path outside the giant divided by
$\log n$ converges in probability to $1/\log(1/\lambda_k)$. This seems to be
a strong indication that it might be possible to derive a new proof for the diameter of
the giant.

Recall that \(\randsimple\) is a simple \(k\)-out digraph with \(n\) vertices chosen uniformly at
random from all such digraphs.  Section \ref{sec:simple} proved that if whp \(\randdfa\) has
property {\textbf P}, then whp \(\randsimple\) has property {\textbf P}.  But results like Theorem
\ref{thm:CLL}, the central limit law of the one-in-core, cannot be transferred to \(\randsimple\)
automatically.  We believe that it might be possible to achieve get the same result for
\(\randsimple\) following the line of \citeauthor{Janson2008}'s treatment of the configuration model
\cite{Janson2008}.

A natural generalization of \(\randdfa\) is to have a deterministic out-degree
sequence, as in the directed configuration model, instead of requiring each
vertex to have out-degree exactly \(k\). With some constraints on the out-degree
sequence, most of our results should hold for this generalized model.
Furthermore, we could let each vertex choose its out-degree independently at
random from an out-degree distribution. Again by adding some restrictions on the
out-degree distribution, most of our results should still hold.

The problem of generating a uniform random surjective function with fixed
domain size is an open problem. Theorem~\ref{thm:CLL} implies a simple
algorithm for choosing a \([km] \to [m]\) surjective function uniformly at
random. Let \(n =\ceil{m/\nuk}\). Then we generate a \(\randdfa\).  
If \(\onecoresize = m\), i.e., if the one-in-core in \(\randdfa\) contains \(m\)
vertices, then \(\randdfa[\onecore]\) is equivalent to a uniform random sample
of a \([km] \to [m]\) surjective function. Otherwise we try again until
\(\onecoresize = m\). Theorem~\ref{thm:CLL} shows that \(\p{\onecoresize = m} =
\Theta(1/\sqrt{m})\). Thus the expected number of \(\randdfa\) needed to be
generated is \(\Theta(\sqrt{m})\). Since generating a \(\randdfa\) takes
\(\Theta(m)\) time, the expected running time of the whole algorithm is
\(\Theta(m^{3/2})\).  But we believe that \(\Theta(m)\) should be achievable.

\section*{Acknowledgment}

The authors thank Laura Eslava, Hamed Hatami, Guillem Perarnau, Bruce Reed, Henning
Sulzbach and Yelena Yuditsky for valuable comments on this work, and Denis Thérien for pointing out
the importance of the model.

\bibliographystyle{abbrplainnat}
\bibliography{citation}

\section*{Appendix}

\setcounter{lemma}{0}
\renewcommand{\thelemma}{A\arabic{lemma}}

\subsection*{1. Inequalities for constants}

\begin{lemma}
    Assume that \(k \ge 2\).
    \begin{enumerate}[(a)]
        \item There exists exactly one \(\tauk > 0\) such that \(1-{\tauk}/k -
            e^{-{\tauk}} = 0\);
        \item \(0 < k - {\tauk} < 1/2\);
        \item \(1/2 < 1-\frac{1}{2k} < \nuk \equiv {\tauk}/k < 1\);
        \item \(\lambdak \equiv (k-{\tauk})\left( \frac{{\tauk}}{k-1} \right)^{k-1}
            < \lambda_k' \equiv (k-{\tauk}) e^{1-k+{\tauk}} < 1\);
        \item \(\gamma_k \equiv \left( \frac k {e{\tauk}} \right)^k(e^{{\tauk}} -1)
            < 1\);
        \item \(\rho_{k} \equiv k e^{1-\tauk} \left( \frac \tauk k \right)^{k-1} <
            1\);
        \item \(\lambdak = \Theta(ke^{-k})\) as \(k \to \infty\).
    \end{enumerate}
    \label{lem:constant}
\end{lemma}
\begin{proof}
    Let \(\eta(x) = 1 - x/k - e^{-x}\). Since \(\eta''(x) = -e^{-x} < 0\),
    \(\eta(x)\) is strictly concave. Since \(\eta(k-1/2) > 0\), and \(\eta(k) <
    0\),
    \(\eta(x) = 0\) must have exactly one positive solution and this solution
    must be in \((k-1/2,k)\). Thus (a) and (b) are proved.
    (c) follows since \({\tauk}/k > 1 -1/k \ge 1/2\).
    For (d) note that \(\lambdak < \lambdak'\)
    as \(1-x < e^{-x}\) for all \(x \ne 0\). For \(\lambda_k' < 1\) note that
    \begin{align*}
        \log \lambdak'
        = \log(k-{\tauk}) + 1 - (k-{\tauk})
        = \log\left[1- (1- (k-{\tauk})) \right] + 1 - (k-{\tauk})
        < 0,
    \end{align*}
    since \(\log(1-x) < -x\) for all \(x \in (0,1)\).

    For (e), first use \({\tauk}/k \equiv 1-e^{-{\tauk}}\) to get
    \begin{align*}
        \gamma_k 
        = \frac 1 {e^{k}( 1 - e^{-{\tauk}} )^{k}} e^{{\tauk}} (1-e^{-{\tauk}})
        = e^{{\tauk}-k}( 1-e^{-{\tauk}} )^{1-k}.
    \end{align*}
    Then use \(k e^{-\tauk} \equiv k - \tauk\) to get
    \begin{align*}
        \log \gamma_k 
        & = {\tauk} - k + (1-k) \log ( 1 - e^{-{\tauk}}) \\
        & = ({\tauk} - k) + \log(1-e^{-{\tauk}}) -k \log(1-e^{-{\tauk}}) \\
        & < ({\tauk} - k) - e^{-{\tauk}} +k (e^{-{\tauk}} +e^{-2{\tauk}}) \\
        %& = ({\tauk} - k) + k e^{-{\tauk}} + e^{-{\tauk}}(k e^{-{\tauk}} - 1) \\
        & = ({\tauk} - k) + (k-{\tauk}) + e^{-{\tauk}}(k-{\tauk} - 1) < 0,
    \end{align*}
    since \(-x > \log(1-x) > -x -x^2\) for all \(x \in (0, 1/2)\) and
    \(e^{-\tauk}=1-\nuk \in (0,1/2)\).

    For (f), use \(\tauk < k\) from (a) to get
    \begin{equation}
    \tauk \equiv k(1-e^{-\tauk}) < k(1- e^{-k})
    .
        \label{eq:k:tau:lower}
    \end{equation}
    Therefore,
    \[
    \frac{\tauk}{k} \equiv 1-e^{-\tauk} < 1- \exp\left\{-k\left(1- e^{-k}\right)\right\}
    .
    \]
    Again by (a), \(\tauk > k-1/2\). Thus
    \begin{equation}
    \tauk \equiv k(1-e^{-\tauk}) > k(1- e^{-k+\frac 1 2})
    .
        \label{eq:k:tau:upper}
    \end{equation}
    Therefore,
    \begin{equation*}
        ke^{-\tauk} < k\exp\left\{-k\left(1- e^{-k+\frac 1
        2}\right)\right\}
        .
    \end{equation*}
    The above bounds imply that
    \begin{align*}
        \rho_{k} 
        \equiv k e^{1-\tauk} \left( \frac \tauk k \right)^{k-1}
        < 
        k \exp \left\{1 - k\left( 1-e^{-k + \frac 1 2} \right) \right\}
        \left( 1- \exp\left\{ -k \left( 1-e^{-k} \right) \right\}
            \right)^{k-1}
            .
    \end{align*}
    Using this bound, numeric computations show that \(\rho_2 < 0.945651\).
    When \(k \ge 3\), the above upper bound is less than
    \[
    k \exp \left\{1 - k\left( 1-e^{-\frac 5 2} \right) \right\},
    \]
    which takes its maximal value at \(k = 3\) for \(k \in [3, \infty)\).
    This maximal value is about \(0.52\). Thus \(\rho_k < 1\) for all \(k \ge 2\).

    By \eqref{eq:k:tau:lower} and \eqref{eq:k:tau:upper}, \(k - \tauk =
    ke^{-k + O(1)}\) and \(\tauk/k = 1- e^{-k+O(1)}\) as \(k \to \infty\).
    Therefore
    \begin{align*}
        \lambdak 
        & \equiv (k-{\tauk})\left( \frac{{\tauk}}{k-1} \right)^{k-1} \\
        & = (k-{\tauk})\left( \frac{{\tauk}}{k} \right)^{k-1} \left(
        \frac{k}{k-1}
        \right)^{k-1} \\
        & = ke^{-k+O(1)} \left(1- e^{-k+O(1)}  \right)^{k-1} e(1+o(1))
        = ke^{-k+O(1)}
        .
    \end{align*}
    Thus (g) is proved.
\end{proof}

\subsection*{2. The sizes of \(k\)-surjections}

In this section we prove Lemma~\ref{lem:k:surj}. Recall that \(K_s\) is the
number of \(k\)-surjections of size \(s\) in \(\randdfa\).  We first deal the case
that \(s\) is small:
\begin{lemma}
    \label{lem:surj:one}
    \(\p{K_1 \ge 1} \le 1/{n^{k-1}} \le 1/n\).
\end{lemma}
\begin{proof}
    A single vertex is a \(k\)-surjection if and only if all its \(k\) arcs are
    self-loops. Thus
    \[
    \p{K_1 \ge 1} 
    \le \sum_{v \in [n]} \p{v \text{ has only self-loops}} 
    = n \left( \frac 1 n \right)^k 
    \le \frac 1 {n^{k-1}}
    \le \frac 1 n.
    \tag*{\qedhere}
    \]
\end{proof}
\begin{lemma}
    \label{lem:surj:small}
    \(
    \p{\sum_{2 \le s \le a n} K_s \ge 1} = \smallo{{1}/{n}},
    \)
    for all fixed \(a \in \left(0, e^{-1/(k-1)}\right)\).
\end{lemma}
\begin{proof}
    We can choose \(\varepsilon \in (0,1)\) such that \(2(k-1)(1-\varepsilon) >
    1\) since
    \(k \ge 2\).  Let \(J = \{2,\ldots, \lfloor a n \rfloor\}\). Then
    \begin{align*}
        \p{\sum_{s \in J} K_s \ge 1} 
        & \le \sum_{s \in J} \sum_{\cS \subseteq [n]:|\cS|=s} \p{\cS
            \text{ is closed}} \\
        & = \sum_{s \in J} \binom{n}{s} \left( \frac s n
            \right)^{ks} \\
        & \le \sum_{s \in J} \left( \frac {en}{s}
        \right)^s \left( \frac s n \right)^{ks} \qquad (\text{Stirling's
        approximation}) \\
        & = \sum_{2 \le s \le n^{\varepsilon}} \left[ e \left( \frac
            s n \right)^{k-1}
            \right]^{s}
            + \sum_{n^{\varepsilon} < s < an } \left[ e \left( \frac s n \right)^{k-1}
            \right]^{s} \\
        %& \equiv I + I\!I.
            & \le \left[ e \left( \frac{n^{\varepsilon}}{n}
                \right)^{k-1}
            \right]^{2}
            \sum_{2 \le s+2} 
            \left[ e \left( \frac{n^{\varepsilon}}{n}
                \right)^{k-1}
            \right]^{s}
            + \sum_{n^{\varepsilon} < s} \left( e \times
                a^{k-1}
            \right)^{s} \\
            & = \bigO{n^{-2(k-1)(1-\varepsilon)}} + \bigO{(e
                a^{k-1})^{n^{\varepsilon}}},
            %= \smallo{\frac 1 n}
            % \\
            %& \le \frac {e^2}{{n^{2(1-\varepsilon)(k-1)}}} \frac {1}{1-e/{n^{(1-\varepsilon)(k-1)}}}
            %+  \frac {\left( e \times a^{k-1}\right)^{n^{\varepsilon}}} {1-e \times
            %    a^{k-1}} = o\left( \frac 1 n \right),
    \end{align*}
    where both terms are \(o(1/n)\) due to our choice of \(\varepsilon\) and
    \(a\).
\end{proof}

When \(s\) is large, we need to take into account the probability that \(\cS\) is
surjective.  Let \(\stirling{x}{y}\) denote Stirling's number of the second
kind, i.e., the number of ways to put \(x\) balls into \(y\) unordered
bins such that there are no empty bins \cite[pp. 64]{Flajolet2009}. Then
\[
    \p{\cS \text{ is surjective}~|~\cS \text{ is closed}} = \frac{\stirling{ks}{s}
s!}{s^{ks}},
\]
where the numerator is the number of ways to choose endpoints for the \(ks\) arcs
in \(\cS\) so that minimum in-degree is one, and the denominator is the total
number of ways to choose endpoints for \(ks\) arcs in \(\cS\). Thus
\begin{align*}
    \p{\cS \text{ is a \(k\)-surjection}}
%& = \p{\cS \text{ is surjective and closed}} \\
& = \p{\cS \text{ is surjective}~|~\cS \text{ is closed}}\p{\cS \text{ is
closed}} \\
& = \frac{\stirling{ks}{s} s!}{s^{ks}} \left(\frac{s}{n}  \right)^{ks}
= \frac{\stirling{ks}{s} s!}{n^{ks}}.
\end{align*}
\citet*{Good1961} established an asymptotic estimation of Stirling's numbers
of the second kind
\begin{equation*}
    \stirling{k s}{s} \sim 
    \frac{(k s)!}{s!}
    \frac{(e^{{\tauk}}-1)^{s}}{{\tauk}^{k s} \sqrt{2 \pi k s (1-k e^{-k})}}.
    %\label{eq:stirling}
\end{equation*}
Applying this and Stirling's approximation for factorials, we have
\begin{align*}
    \p{\cS \text{ is a \(k\)-surjection}}
& \sim   \frac{(k s)!}{s!}
    \frac{(e^{{\tauk}}-1)^{s}}{{\tauk}^{k s} \sqrt{2 \pi k s (1-k e^{-k})}}
    \frac{s!}{n^{ks}} \\
& \sim \frac{1}{\sqrt{1-ke^{-{\tauk}}}}
\left[ \left( \frac{s}{n} \right)^{k} \gamma_k \right]^{s},
\numberthis \label{prob:comm}
\end{align*}
where \(\gamma_k \equiv \left({k}/{e{\tauk}}  \right)^{k}(e^{{\tauk}}-1) < 1\)
(see Lemma~\ref{lem:constant}).
\begin{lemma}
    \label{lem:surj:big}
    There exists a constant \(b \in (\nuk, 1)\) such that \(\p{\sum_{bn \le s
    \le n} K_s \ge 1} = o\left( 1/n \right)\).
\end{lemma}

\begin{proof}
    Let \(b > \nuk\) be a constant decided later.
    If \(|\cS| = s \in [bn, n]\), then by \eqref{prob:comm} 
    \begin{align*}
        \p{\cS \text{ is a \(k\)-surjection}}
        %%= \frac{\stirling{ks}{s} s!}{n^{ks}} %\\
        = \bigO{
        \left[\left(\frac{s}{n}\right)^{k}\gamma_k
            \right]^{s}} %\\
            \le \bigO{\gamma_k^s}
            \le \bigO{\gamma_k^{bn}}.
            %\label{eq:surjection}
    \end{align*}
    Since \(b > \nuk > 1/2\) (Lemma~\ref{lem:constant}),
    \[
    \binom{n}{s} \le \binom{n}{bn} = O\left(\frac 1 {\sqrt n} \left[
        \frac{1}{b^{b}(1-b)^{1-b}} \right]^{n}\right).
        \]
    Therefore
    \begin{align*}
        \p{K_s \ge 1} \le \binom{n}{s} \p{\cS \text{ is a \(k\)-surjection}}
        \le \bigO{ \left[ \frac{\gamma_k^{b}}{b^{b}(1-b)^{1-b}} \right]^{n} }.
    \end{align*}
    Since the quantity in the square brackets goes to \(\gamma_k < 1\) as \(b \to
    1\), we can pick a \(b\) close enough to one such that \(\p{\sum_{bn \le s \le
    n} K_s \ge 1} = o\left( 1/n \right)\).
\end{proof}

Let \(a \in (0,\nuk)\) and \(b \in (\nuk,1)\) be two constants such that the upper
bounds in Lemma~\ref{lem:surj:small} and \ref{lem:surj:big} hold. If \(|\cS| =
x n\) with \(x \in (a, b)\) and \(xn\) integer-valued, then by \eqref{prob:comm} and Stirling's approximation
\begin{align*}
    \e {K_{xn}}
    & = \binom{n}{xn} \p{\cS \text{ is a \(k\)-surjection}} \\
    & \sim \frac{1}{\sqrt{2 \pi x (1- x) n}}
    \left[ \frac{1}{\left( x \right)^{x} (1-x)^{1-x}}
        \right]^{n} 
    \frac{1}{\sqrt{1-ke^{-{\tauk}}}} \left(x^k \gamma_k \right)^{x n} \\
    & = \frac{1}{\sqrt{2 \pi (1-k e^{-{\tauk}})n}}  g\left( x \right) 
    \left[f\left( x \right) \right]^{n}
    \numberthis
    \label{expec:comm}
\end{align*}
where
\[
g(x) \equiv \frac 1 {\sqrt{x(1-x)}}, \qquad \qquad
f(x) \equiv 
    \left[
        \frac{x^{k-1} \gamma_k}{(1-x)^{(1-x)/x} }
    \right]^{x}.
    \]
\begin{lemma}
    \label{lem:surj:middle}
    For all fixed \(a \in (0, \nuk)\), \(b \in (\nuk, 1)\) and \(\delta \in
    (0,1/2)\),
    \(\p{\sum_{s \in J} K_s \ge 1} = o(1/n)\),
    where \(J = [an, \nuk n - n^{\frac 1 2 + \delta}] \cup [\nuk n +  n^{\frac 1 2 +
    \delta}, bn]\).
\end{lemma}
\begin{proof}
    Let \(h(x) \equiv \log f(x)\).  Lemma~\ref{lem:function:h} shows
    that as \(x \to \nuk\),
    \[ h(x) = -\frac{(x-\nuk)^2}{2 \sigma_k^2}+\bigO{|x-
    \nuk|^3},
    \]
    and that \(h(x)\) is strictly increasing on \((a,\nuk)\) and strictly decreasing on
    \((\nuk,b)\).
    It follows from \(|s/n - \nuk| > \errp\) that
    \(h\left( s/n \right) 
    %\le - \frac{(\errp)^{2}}{2 \sigma_k^2} + \bigO{(\errp)^3} 
    \le -{n^{2\delta-1}}/{2 \sigma_k^2} + \bigO{n^{3\delta-3/2}}.
    \)
    As for \(g(x)\), it is bounded on \((a,b)\). Thus by \eqref{expec:comm} and
    Markov's inequality
    \begin{align*}
        \log ( n^2 \p{K_{s} \ge 1})
    & \le \log ( n^2 \e{K_{s}}) \\
    & = \log \left(n^2 \bigO{n^{-1/2}} f\left( \frac s n \right)^{n} \right) \\
    & = \bigO{\log n} + n h\left( \frac s n \right) \\
    & \le \bigO{\log n} -\frac{n^{2\delta}}{2 \sigma_k^2}
    + O\left( n^{3\delta-1/2} \right),
    \end{align*}
    which goes to \(-\infty\). In other words,
    \(
    \p{K_s \ge 1} = o\left( 1/{n^2} \right).
    \)
    So \(\p{\sum_{s \in J} K_s \ge 1} = o\left( 1/n \right)\).
\end{proof}

Lemma~\ref{lem:k:surj} follows immediately from Lemma~\ref{lem:surj:one},
\ref{lem:surj:small}, \ref{lem:surj:big}, and \ref{lem:surj:middle}.

\subsection*{3. Special functions}
\begin{lemma}
    Let \(f(x)\), \(g(x)\) and \(h(x)\) be defined as in the previous subsection.  
    Let \(\nuk\), \(\tauk\) and \(\sigma_k\) be as in Lemma \ref{lem:constant}.
    Then
    \begin{enumerate}[(a)]
        \item As \(x \to \nuk\), 
            \(g\left( x \right) 
            = g(\nuk) + \bigO{|x-\nuk|}
            = \left(1 + \bigO{|x-\nuk|} \right) /{(\sigma_k\sqrt{1-k
                e^{-\tauk}})}\).
            \item \(h(x)\) and \(f(x)\) are strictly increasing on \((1-\frac{1}{k},\nuk)\) and
                strictly decreasing on \(\left(\nuk, 1 \right)\).
            \item As \(x \to \nuk\),
            \[
        h(x) 
        = h(\nuk) + O(|x-\nuk|^3)
        = - \frac{(x-\nuk)^2}{2 \sigma_k^2} + O(|x-\nuk|^3),
        \]
        which implies that
        \[
        f(x) 
        = e^{h(x)}
        = \exp\left\{- \frac{(x-\nuk)^2}{2 \sigma_k^2}\right\} + O(|x-\nuk|^3).
        \]
    \end{enumerate}
    \label{lem:function:h}
\end{lemma}

\begin{proof}
    For (a), recall that \( \sigma_k^2 \equiv {{\tauk}}/(k
    e^{{\tauk}}(1-ke^{{-\tauk}})).\)
    Thus
    \(\sigma_k^2 (1-k e^{-\tauk}) 
    %= \frac {\tauk} {k} e^{-\tauk} 
    = \nuk (1-\nuk).\)
    Then
    \(
    g(\nuk)= {1}/{\sqrt{\nuk(1-\nuk)}} = {1}/{\sigma_k
        \sqrt{1-ke^{-\tauk}}}.
        \)
        Since \(g'(x)\) is bounded around \(\nuk\), by Taylor's theorem,
    \[
    g(x) = g(\nuk) + O(|x - \nuk|) = (1+\bigO{|x-\nuk|})\frac{1}{\sigma_k
        \sqrt{1-ke^{-\tauk}}}, \qquad \text{as \(x \to \nuk\)}.
        \]
        Let \(r(x) = \log \left( f(x)^{1/x} \right) = h(x)/x\). 
        Using \(\tauk/k \equiv 1-e^{-\tauk} \equiv \nuk\) shows that
    \[
    \gamma_k = \left( \frac{1}{e \nuk} \right)^{k} e^{\tauk} \nuk
    = \nuk^{-k+1} e^{-k+\tauk} = \nuk^{-k+1} (e^{-\tauk})^{(k-\tauk)/\tauk} =
    \nuk^{-k+1} (1-\nuk)^{(1-\nuk)/\nuk}.
    \]
    Then 
    \(
    r(\nuk) = \log\left( {\nuk^{k-1}}{\left( 1-\nuk
        \right)^{(\nuk-1)/\nuk}} \gamma_k 
        \right) = \log(1) = 0,
    \)
    \[
    r'(x) = \frac k x + \frac 1 {x^2} \log(1-x),\qquad \text{and} \qquad
    r''(x) = - \frac k {x^2} - \frac {2 \log(1-x)} {x^3} - \frac 1 {x^2(1-x)}.
    \]
    Therefore \(r'(\nuk) = 0\) and \(r''(\nuk)=-1/(\nuk \sigma_k^2)\).

    Since \(h(x) = x r(x)\),
    \[
    h'(x) = r(x) + x r'(x), \qquad h''(x) = 2 r'(x) + x r''(x) = \frac k x -
    \frac 1 {x(1-x)}.
    \]
    Thus \(h(\nuk) = 0\), \(h'(\nuk) = 0\) and \(h''(\nuk) = -1/\sigma_k^2\).
    Also recalling that \(1-\frac{1}{k} 
        < 1 - \frac{1}{2k} < \nuk < 1\) (Lemma \ref{lem:constant}),
    \(h(x)\) is strictly concave on \( (1-\frac{1}{k}, 1)\), reaching maximum at \(\nuk\).
    Thus (b) is proved.  The two asymptotic equations in (c) follow from Taylor's theorem.
\end{proof}

\subsection*{4. Probability generating functions of Galton-Watson processes}

\begin{lemma}
    Let \(\mu \in (0, \frac{1}{2k})\) be a constant where $k \ge 2$.  Let \((Z_{m})_{m \ge 0}\) be a
    Galton-Watson
    process with \(Z_0 \equiv 1\) and offspring distribution \(\Bin(k, \mu)\). Let
    \(\varphi_m(y) \equiv \e y^{Z_m}\). Then
    \[
    \varphi_{m}(0) \le 1 - (k \mu)^{m} + \left( 1- \frac{1}{2^{m}} \right) (k
    \mu)^{m+1}.
    \]
    \label{lem:gen:f}
\end{lemma}
\begin{proof}
    We use induction. 
    Let \(c_m = 1 - 1/2^{m}\).
    For \(m=1\),
    \[
    \varphi_1(y) = \e y^{Z_1} 
    %= \sum_{0 \le i \le k} \binom{k}{i} \mu^{i} (1-\mu)^{k-i} y^{i} 
    = (1 - \mu(1-y))^{k}.
    \]
    Since \(\mu > 0\) and \(k \ge 2\), by Taylor's theorem,
    \[
    \varphi_1(0) 
    = (1-\mu)^k \le 1 - k \mu + \frac{(k \mu)^2}{2}
    = 1 - k \mu + c_1 (k \mu)^2.
    \]
    It is well known that for \(m > 1\), \(\varphi_{m}(y) =
    \varphi_1(\varphi_{m-1}(y))\) (see~\citep{Durrett2010probability}). Assuming
    the lemma holds for \(m\), then
    \begin{align*}
        \varphi_{m+1}(0)
        & = \varphi_1(\varphi_{m}(0)) = \left( 1- \mu \left( 1-\varphi_m(0)
        \right) \right)^{k} \\
        %& \le \left( \mu \left( 1- (k \mu)^m + c_m (k \mu)^{m+1}\right) + 1 - \mu
        %\right)^k \\
        & \le \left( 1- \mu\left( (k\mu)^{m} - c_m(k \mu)^{m+1} \right) \right)^k
        \\
        & \le 1 - k \mu \left( (k\mu)^{m} - c_m(k \mu)^{m+1} \right)
        + \frac{k^2}{2} \mu^2 \left( (k\mu)^{m} - c_m(k \mu)^{m+1} \right)^2
        \\
        & = 1 - (k\mu)^{m+1} + c_m (k \mu)^{m+2} +
        \frac{(k \mu)^m}{2}(1- c_m k \mu)^2 (k \mu)^{m+2} \\
        & \le 1 - (k\mu)^{m+1} + c_{m+1}(k \mu)^{m+2}
        ,
    \end{align*}
    since \(k \mu < 1/2\) and \(c_{m+1} = c_m + 1/2^{m+1}\).
\end{proof}

\end{document}